\documentclass[11pt,a4paper]{article}

\usepackage{amssymb}
\usepackage{amsfonts} 
\usepackage{amsmath,amsthm}  
\usepackage{latexsym}
\usepackage{graphicx}          \DeclareGraphicsExtensions{.eps}
\usepackage{a4wide}
\usepackage{url}
\usepackage[all]{xy}
\usepackage{enumerate}
\usepackage[T1]{fontenc}

\usepackage{times}

\numberwithin{equation}{section}
\newtheorem{theorem}{THEOREM}[section]
\newtheorem{lemma}[theorem]{LEMMA}
\newtheorem{corollary}[theorem]{COROLLARY}
\newtheorem{proposition}[theorem]{PROPOSITION}
\newtheorem{remark}[theorem]{REMARK}
\newtheorem{problem}[theorem]{PROBLEM}

\newtheorem{defin}[theorem]{DEFINITION}

\newtheorem{fact}[theorem]{FACT}

\newenvironment{definition}{\begin{defin}\rm}{\end{defin}}

\newcommand{\rats}{\mbox{\(\mathbb Q\)}}
\newcommand{\reals}{\mbox{\(\mathbb R\)}}

\def\zilch{\vrule height0pt width0pt depth0pt}

\def\endproof{\ifvmode\else\unskip\fi\zilch%
\penalty9999 \hbox{}\nobreak\hfill\quad$\Box$\endtrivlist}

\def\ba{boolean algebra}

\def\claim{\par\medskip\noindent{\bf Claim. }}
\def\Claim#1{\par\medskip\noindent{\bf Claim #1. }}
\def\pfclaim{\par\noindent{\bf Proof of claim. }}

\def\cl{\mathop{\rm cl}}
\def\dom{\mathop{\rm dom}}
\def\int{\mathop{\rm int}} %interior
\def\rng{\mathop{\rm rng}}  

\def\Var{\mathop{\rm Var}} 

\def\G{{\sf G}}

\def\g{{\sf g}}
\def\bb{{\sf b}} %%%avoid \b for `bar'%%%%%%%%% Also \BB for bold index

\def\r{{\sf r}}

\def\ws{winning strategy}

\def \rb{\,\right\}}
\def \halfthinspace{\relax\ifmmode\mskip.5\thinmuskip\relax\else\kern.8888em\fi}

\def\axC{\mathrm{C}}
\def\axD{\mathrm{D}}
\def\axG{\mathrm{G}}
\def\axK{\mathrm{K}}
\def\axS{\mathrm{S}}
\def\axU{\mathrm{U}}

\def\a{a}

\def \rb{\,\right\}}

\def\ang#1{\langle #1\rangle}
\def\vec#1#2{#1_1,\ldots{},#1_{#2}}
\def\Vec#1#2{#1_0,\ldots{},#1_{#2}}

\def\G{{\bf G}}

\def\c #1{{\cal #1}}
\def\rep{representation}
\def\rb{representable}

\def\pa{$\forall$}
\def\pe{$\exists$}

\def\str{structure}

\def\:={\buildrel \rm def \over =}

\def\bo{\Box}
\def\di{\Diamond}

\def\dit{{\ang t}}
\def\did{{\ang d}}
\def\didt{{\ang{dt}}}
\def\bod{{[d]}}

\def\nhd{neighbourhood}
\def\tp{\mathop{\rm tp}}
\def\nice{dense-in-itself metric}
\def\nices{dense-in-themselves metric}
\def\R{\reals}
\def\ro{regular open}
\def\sem #1{[\![ #1 ]\!]}

\def\setG{\mathbb{G}}

\def\setI{\mathbb{I}}
\def\setB{\mathbb{B}}
\def\setS{\mathbb{S}}
\def\setT{\mathbb{T}}
\def\setU{\mathbb{U}}

\def\logicbomu{\axS4\mu}

\def\a{\alpha}
\def\app{\approx}

\def\ab#1{|#1|}

\def\At{\mathsf{At}}

\def\Di{\langle t\rangle}
\def\Dim{\Diamond}

\def\F{\mathcal{F}}

\def\g{\gamma}
\def\G{\Gamma}

\def\M{\mathcal{M}}

\def\ph{\varphi}

\def\sub{\subseteq}

\def\Var{\mathsf{Var}}
\def\Lbig{\c L^{\mu\dit\didt}_{\bo\bod\forall}}
\def\infrule #1#2{\frac{\hbox{\strut$#1$}}{\hbox{\strut$#2$}}}

\title{Spatial logic of modal mu-calculus and tangled closure operators}
\author{Robert Goldblatt\thanks{School of Mathematics and Statistics, Victoria University of Wellington, New Zealand.
{\tt sms.vuw.ac.nz/\~{}rob/}}
\hskip6pt and
 Ian Hodkinson\thanks{Department of Computing, Imperial College London, UK.
{\tt www.doc.ic.ac.uk/\~{}imh/}}}

\date{6 March 2016}

\begin{document}
\maketitle

\begin{abstract}
There has been renewed interest in recent years in McKinsey and Tarski's interpretation of modal logic in topological spaces and their proof that S4 is the logic of any  separable dense-in-itself metric space. Here we extend this work to the modal mu-calculus and to a logic of tangled closure operators that was developed by Fern\'andez-Duque after these two languages had been shown by Dawar and Otto to have  the same expressive power over finite transitive Kripke  models.  We  prove that this equivalence remains true over topological spaces. 

We establish the finite model property in Kripke semantics
for various tangled closure logics with and without the universal modality $\forall$.
We also extend the McKinsey--Tarski topological `dissection lemma'.
These results are used to construct a 
representation map (also called a d-p-morphism) from
any dense-in-itself metric space $X$ onto any finite connected locally connected serial 
transitive  Kripke frame.

This yields completeness theorems over $X$ for
a number of languages:
(i) the modal mu-calculus with the closure operator $\Diamond$;
(ii) $\Diamond$ and the tangled closure operators $\langle t \rangle$;
(iii) $\Diamond,\forall$; 
(iv) $\Diamond,\forall,\langle t \rangle$;
(v) the derivative operator $\langle d \rangle$;
(vi) $\langle d \rangle$ and the associated tangled closure operators $\langle dt \rangle$;
(vii) $\langle d \rangle,\forall$;
(viii) $\langle d \rangle,\forall,\langle dt \rangle$.
Soundness also holds, if: (a) for languages with $\forall$, $X$ is connected; and
(b) for languages with $\langle d \rangle$, $X$ validates the well known axiom $\mathrm{G}_1$.
For countable languages without $\forall$, we  prove strong completeness.
We also show that  in the presence of $\forall$, strong completeness fails 
if $X$ is  compact and locally connected.
\end{abstract}

%\begin{keyword}
%derivative operator\sep
%dense-in-itself metric space\sep
%modal logic\sep
%finite model property\sep
%strong completeness
%\MSC[2010] primary 03B45; secondary 54E35
%\end{keyword}

\section{Introduction}

Modal logic can be given semantics over topological spaces.
In this setting, the  modality $\di$ can be interpreted in more than one way.
The first and most obvious way is as \emph{closure.}
Writing $\sem\varphi$ for
the set of points (in a topological model) at which a formula $\varphi$ is true,
$\sem{\di\varphi}$ is defined to be the \emph{closure}
of $\sem\varphi$,
so that $\di\varphi$ holds at a point $x$ if and only if
every open \nhd\  of $x$ contains a point $y$ satisfying $\varphi$.
Then, $\bo$ becomes the interior operator:
$\sem{\bo\varphi}$ is the interior of $\sem\varphi$.
Early studies of this semantics include
 \cite{Tang38,Tar38,McKi42,McKiT44,McKiT46}.
 
In a seminal result, McKinsey and Tarski \cite{McKiT44}
proved that the logic of any given separable\footnote{The separability assumption was removed in \cite{RS:mm}.}
\nice\ space
in this semantics is S4: it can be  axiomatised by the basic modal Hilbert system $\axK$
augmented by the two axioms
$\bo \varphi\to \varphi$ (T) and $\bo \varphi\to\bo\bo \varphi$ (4).

Motivated perhaps by 
the current wide interest in spatial logic,
a wish to present simpler proofs in `modern language',
growing awareness of the work of particular groups
such as Esakia's and Shehtman's,
or involvement in new settings such as
dynamic topology,
interest in McKinsey and Tarski's result has revived in recent years.
A number of new proofs of it
have appeared, some for specific spaces or embodying other variants
\cite{Mints:S4no1,GuramMai02:S4,Aiello-vB-Guram:space,Mints05aproof,Slavnov05,
Lando-Sarenac:11,Hodk:mckt_R}.
Very recently, strong
completeness (every countably infinite S4-consistent set of modal
formulas is satisfiable in every \nice\ space) was established by Kremer \cite{Kremer2010:stro}.

In this paper, we seek to extend McKinsey and Tarski's theorem  to 
more powerful languages.
We will extend the modal syntax  in two separate ways:
first, to the mu-calculus,
which adds least and greatest fixed points to the basic modal language,
and second, by adding
 an infinite sequence of 
new modalities $\di_n$ of arity $n$ $(n\geq1)$
introduced in the context of Kripke semantics by Dawar and Otto \cite{DO09}.
The semantics of $\di_n$ is given by the mu-calculus formula
\[
\di_n(\vec\varphi n)\equiv\nu q\bigwedge_{1\leq i\leq n}\di(\varphi_i\wedge q),
\]
for a new atom $q$ not occurring in $\vec\varphi n$.
The order and multiplicity of arguments to a $\di_n$ is immaterial, so
we will abbreviate $\di_n(\vec\gamma n)$ to $\dit\{\vec\gamma n\}$.
Fern\'andez-Duque used this to give the modalities
topological semantics, 
dubbed them \emph{tangled closure modalities}
(this is why we use the notation $\dit$),
and studied them in \cite{FD:SL11,FD:ijcai11,FD:apal12,FD:jsl12b}.

Dawar and Otto \cite{DO09} showed that, somewhat surprisingly, the
mu-calculus and the tangled modalities have exactly the same expressive power over
finite  Kripke models with  transitive frames. We will prove
that this remains true over topological spaces.
So the tangled closure modalities offer a viable alternative to the mu-calculus in both these settings.

We go on to determine the logic of an arbitrary \nice\ space $X$ 
in these languages.
We will show that in the mu-calculus, the logic of $X$
is axiomatised by a system called $\axS4\mu$ comprising 
Kozen's basic  system for the mu-calculus augmented by
the S4 axioms, and  the tangled logic
of $X$ is axiomatised by a system called $\axS4t$ 
similar to one in \cite{FD:ijcai11}.
We will establish strong completeness for countable sets of formulas.

We will also consider the extension of the tangled language
with the \emph{universal modality,} `\pa'.
(Earlier work on the universal modality in topological spaces
includes \cite{Sheh:everywhere99,LucBry11}.)
This language can express connectedness:
there is a formula $\axC$ valid in precisely the connected spaces.
Adding this and some standard machinery for \pa\ to the system $\axS4t$
gives a system called `$\axS4t.\axU\axC$'.
We will show that every $\axS4t.\axU\axC$-consistent formula
is satisfiable in every \nice\ space.
Thus, the logic of an arbitrary connected \nice\ space is $\axS4t.\axU\axC$.
We also show that strong completeness fails in general, even
for the modal language plus the universal modality.

\smallskip

A second and more powerful spatial interpretation of $\di$ is as the \emph{derivative operator.}
Following tradition,  when considering this interpretation we will generally
write the modal box and diamond as $\bod$ and $\did$.
In this interpretation, $\sem{\did\varphi}$ 
is defined to be the set of \emph{strict limit points}
of $\sem\varphi$:
so $\did\varphi$ holds at a point $x$ precisely when
every open \nhd\ of $x$ contains a point $y\neq x$ satisfying $\varphi$.
The original closure diamond is expressible by the derivative operator:
$\di\varphi$ is equivalent in any topological model to $\varphi\vee\did\varphi$,
and $\bo\varphi$ to $\varphi\wedge\bod\varphi$.
So in passing to $\did$, we have 
not reduced the power of the language.

Already in \cite[Appendix I]{McKiT44}, McKinsey and Tarski  
discussed the derivative operator and asked a number of questions about it.
It has since been studied by, among others, Esakia and his Tbilisi group 
(\cite{Esa04:apal,GEG05}, plus many other publications),
Shehtman \cite{Sheh:d90,Sheh:hab},  Lucero-Bryan \cite{LucBry11},
and Kudinov--Shehtman
\cite{KudShe14}, section 3 of which contains a survey of results.

In the derivative semantics,
determining the logic of a given \nice\ space is not a simple matter, for
the logic can vary with the space.
As McKinsey and Tarski observed, 
$\did\big((x\wedge\did\neg x)\vee(\neg x\wedge\did x)\big)\leftrightarrow\did x\wedge\did\neg x$
is valid in $\R^2$
but  not in $\R$. 
This formula is valid in the same topological spaces as the formula
$\axG_1$, where
for each integer $n\geq1$,
\[
\axG_n=\Big(\bod\bigvee_{0\leq i\leq n}\bo Q_i\Big)
\to\bigvee_{0\leq i\leq n}\bod\neg Q_i.
\]
Here,  $\Vec pn$ are pairwise distinct atoms, and 
for $i=0,\ldots,n$,
\[
Q_i=p_i \land  \bigwedge_{i\ne j\leq n}\neg p_j.
\]
In \cite{Sheh:d90}, Shehtman proved that
the logic of $\R^n$ for finite $n\geq 2$ is $\axK\axD4\axG_1$,
axiomatised by the basic system $\axK$
together with the axioms $\did\top$ (D),  $\bod p\to\bod\bod p$ (4), and $\axG_1$.
The logic of $\R$ was shown by Shehtman 
\cite{Sheh:hab}
 and Lucero-Bryan \cite{LucBry11} to be $\axK\axD4\axG_2$.
The logic of
every separable zero-dimensional \nice\  space
(such as $\rats$ and the Cantor space) is just $\axK\axD4$ \cite{Sheh:d90},
the smallest possible logic of a \nice\ space in the derivative semantics.
\cite{GBezLucBry12} proves that there are continuum-many logics 
of subspaces of the rationals in the language with $\bod$.

It is plain that $\axG_1\vdash \axG_2\vdash \axG_3\vdash\cdots$,
so the logics $\axK\axD4\axG_1\supseteq \axK\axD4\axG_2\supseteq\cdots$ form a decreasing chain,
and by \cite[corollary 3.11]{LucBry11}, its  intersection is $\axK\axD4$.
Shehtman  \cite[problem 1]{Sheh:d90} asked if $\axK\axD4\axG_1$ is the largest possible logic of a \nice\ space in the derivative semantics.

In this paper, we  answer  Shehtman's question affirmatively:
every $\axK\axD4\axG_1$-consistent formula of the language with $\did$
is satisfiable in every \nice\ space.
Thus, the logic of every \nice\ space that validates $\axG_1$
is exactly $\axK\axD4\axG_1$.
We also establish strong completeness for such spaces.

Adding the tangled closure operators, we 
prove similarly that the logic of every \nice\ space that validates $\axG_1$
is axiomatised by $\axK\axD4\axG_1t$  (including the tangle axioms).
We also prove strong completeness.

Further adding the universal modality, we 
show similarly
that $\axK\axD4\axG_1t.\axU\axC$ 
(and $\axK\axD4\axG_1.\axU\axC$
if the tangle closure operators are dropped)
axiomatises the logic of every connected \nice\ space that validates $\axG_1$.
Strong completeness  fails in general, 
as a consequence of the proof that it already fails for
the weaker language with $\bo$ and \pa.

The reader can find a summary of our results in table~\ref{table1} in section~\ref{sec:the end}.

\smallskip

Our proof works in a fairly familiar way, similar in spirit to 
McKinsey and Tarski's original argument in \cite{McKiT44} ---
indeed, we use some results from that paper.
There are three main steps.
\begin{enumerate}
\item We  establish the \emph{finite model property} for the various logics,
in Kripke semantics.  
This work may be of independent interest: earlier related results were proved in \cite{Sheh:d90,FD:ijcai11}.

\item We then prove a topological theorem 
that establishes  Tarski's `dissection lemma' \cite[satz 3.10]{Tar38}, \cite[theorem 3.5]{McKiT44} and a variant of it.

\item These topological results are used to construct a map
from an arbitrary \nice\ space onto any finite connected $\axK\axD4\axG_1$ Kripke frame, that  preserves the required formulas.

\end{enumerate}
Putting the three steps together proves completeness for all the 
languages, which is then
lifted by a separate argument
to strong completeness for languages without \pa.

It can be seen that our results concern
the logic of each individual space within a large class of spaces
(the \nices\  spaces), rather than the logic of a large class of spaces,
or of particular spaces such as $\R$.
This is as in \cite{McKiT44}.
We do not assume separability,
we consider languages that have not previously
been much studied in the topological setting,
and we obtain
some results on strong completeness, a matter that has only recently been 
investigated in this setting.

The paper is divided into two parts of roughly equal length.
Part 1 is devoted to proving the finite model property
in Kripke semantics for the logics of concern in the paper.
Part 2 covers topology and spatial completeness results, and can be read
independently by taking the results of part 1 on trust.
The short section~\ref{sec:basics 1} preceding part 1 contains
foundational material needed in both parts.

\section{Basic definitions}\label{sec:basics 1}

In this section, we lay out the main definitions, notation, and
some basic results.

\subsection{Notation for sets and binary relations}
Let $X,Y,Z$ be sets.
We let $\wp(X)$ denote the power set (set of all subsets) of $X$.
We write $X\setminus Y$ for $\{x\in X:x\notin Y\}$.
Note that $(X\cap Y)\setminus Z=X\cap(Y\setminus Z)$, so we may omit
the parentheses in such expressions.
For a partial function $f:X\to Y$, we let $\dom f$ denote the domain of $f$, and
$\rng f$ its range.

A \emph{binary relation} on a set $W$ is a subset of $W\times W$.
Let $R$ be a binary relation on $W$.
We write any of
$R(w_1,w_2)$, $Rw_1w_2$, and $w_1Rw_2$ to denote that $(w_1,w_2)\in R$.
We say that $R$ is \emph{reflexive} if
$R(w,w)$  for all $w\in W$, and \emph{transitive} if 
$R(w_1,w_2)$ and $R(w_2,w_3)$ imply $R(w_1,w_3)$.
We write $R^*$ for the \emph{reflexive transitive closure} of $R$:
the smallest reflexive transitive binary relation that contains $R$.
We also write 
\[
\begin{array}{cclcl}
R^{-1}&=&\{(w_2,w_1)\in W\times W:R(w_1,w_2)\},
\\
R^\circ&=&\{(w_1,w_2)\in W\times W: R(w_1,w_2)\wedge R(w_2,w_1)\}
&=&R\cap R^{-1},
\\
R^\bullet&=&\{(w_1,w_2)\in W\times W:R(w_1,w_2)\wedge\neg R(w_2,w_1)\}
&=&R\setminus R^{-1}.
\end{array}
\]

For $w\in W$, we let $R(w)$ denote
the set $\{w'\in W:R(w,w')\}$, 
sometimes called the set of \emph{$R$-successors} or \emph{$R$-alternatives} of $w$.
For $W'\subseteq W$, we write $R\restriction W'$ for the binary relation
$R\cap(W'\times W')$ on $W'$.

We write $\R$ for the set of real numbers,
On for the class of ordinals, and $\omega$ for the first infinite ordinal.

\subsection{Kripke frames}\label{ss:kripke frames}

A \emph{(Kripke) frame} is a pair $\c F=(W,R)$,
where $W$ is a non-empty set of `worlds' and $R$ is a binary relation on $W$.
We attribute properties to a frame by the usual extrapolation
from the frame's components.
So, we say that $\c F$ is \emph{finite} if $W$ is finite,
\emph{reflexive} if $R$ is reflexive, and \emph{transitive} if $R$ is transitive.
Two frames are said to be \emph{disjoint} if their respective sets of worlds are disjoint.
And so on.

A \emph{root of $\c F$} is an element $w\in W$
such that $W=R^*(w)$. Roots of a frame may not exist, nor be unique when they do.
We say that $\c F$ is \emph{rooted} if it has a root.
At the other end, an element $w\in W$ is said to be \emph{$R$-maximal} if
$R^\bullet(w)=\emptyset$. Such an element has no `proper' $R$-successors, of which
it is not itself an $R$-successor.

A \emph{subframe} of $\c F$
is a frame of the form $\c F'=(W',R\restriction W')$, for non-empty $W'\subseteq W$.
It is simply a sub\str\ of $\c F$ in the usual model-theoretic sense.
We call $\c F'$ the \emph{subframe of $\c F$ based on $W'$.}
We say that $\c F'$ is a \emph{generated} or \emph{inner subframe}  of $\c F$
if $R(w)\subseteq W'$ for every $w\in W'$ --- equivalently, $R\restriction W'=R\cap(W'\times W)$.
For $w\in W$, we write:
\begin{itemize}
\item  $\c F(w)$ for the subframe $(R(w),R\restriction R(w))$ of $\c F$
based on $R(w)$,

\item  $\c F^*(w)$
for the subframe $(R^*(w),R\restriction R^*(w))$
of $\c F$ generated by $w$.
\end{itemize}

For an integer $n\geq1$, we say that $\c F$ is \emph{$n$-connected}
if it is not the union of $n+1$ pairwise disjoint generated subframes
(recall that subframes are non-empty),
\emph{connected} if it is 1-connected,
and \emph{locally $n$-connected} if for each $w\in W$, the subframe $\c F(w)$
is $n$-connected.
Note that $\c F$ is $n$-connected iff the equivalence relation $(R\cup R^{-1})^*$ on $W$ has at most $n$ 
equivalence classes.
Every rooted frame is connected.
Connectedness will be discussed in more detail in section~\ref{sec:path conn}.

\subsection{Fixed points}\label{ss:lfp}

Let $X$ be a set and $f:\wp(X)\to\wp(X)$ be a map.
We say that
$f$ is \emph{monotonic} if $f(S)\subseteq f(S')$
whenever $S\subseteq S'\subseteq X$.
By a  well known theorem of Knaster and Tarski \cite{T-Knaster55},
actually formulated for complete lattices,
every monotonic $f:\wp(X)\to\wp(X)$ has \emph{least and greatest fixed points} ---
there is a unique $\subseteq$-minimal subset $L\subseteq X$
such that $f(L)=L$, and
a unique $\subseteq$-maximal $G\subseteq X$
such that $f(G)=G$.
We write $L=LFP(f)$ and $G=GFP(f)$.

There are a couple of useful ways to `compute' these fixed points.
First,
define by  recursion a subset $S_\alpha\subseteq X$
for each ordinal $\alpha$, by $S_0=\emptyset$,
$S_{\alpha+1}=f(S_\alpha)$, and
$S_\delta=\bigcup_{\alpha<\delta}S_\alpha$ for limit ordinals $\delta$.
The $S_\alpha$ form an increasing chain terminating in $LFP(f)$, so
\[
LFP(f)=\bigcup_{\alpha\in\rm On}S_\alpha.
\]
A similar expression can be given for $GFP(f)$.
Second, a subset $S\subseteq X$ is said to be a \emph{pre-fixed point} of $f$ if
 $f(S)\subseteq S$, and a \emph{post-fixed point}
if $f(S)\supseteq S$.
In~\cite{T-Knaster55} it is proved that
$LFP(f)$ is the intersection of all pre-fixed points of $f$, and dually for
$GFP(f)$:
\[
\begin{array}{rcl}
LFP(f)&=&\bigcap\{S\subseteq X:f(S)\subseteq S\},
\\
GFP(f)&=&\bigcup\{S\subseteq X:f(S)\supseteq S\}.
\end{array}
\]

For $f:\wp(X)\to\wp(X)$, define $f':\wp(X)\to\wp(X)$ by
$f'(S)=X\setminus f(X\setminus S)$.
It is an exercise to check that $f$ is monotonic iff $f'$ is,
and in that case, 
$GFP(f)=X\setminus LFP(f')$.

Least fixed points are used in the semantics of the mu-calculus, coming up next.

\subsection{Languages}

We assume some familiarity with modal languages and  the mu-calculus.
We fix a set $\Var$ of \emph{propositional variables,} or \emph{atoms.}
Sometimes we may make assumptions on $\Var$ --- for example, that it is finite.
We will be considering various logical languages.
The biggest of them is
denoted by 
$\Lbig$, 
which is a set of formulas defined as follows:
\begin{enumerate}
\item each $p\in\Var$ is a formula (of $\Lbig$),

\item $\top$ is a formula,

\item if $\varphi,\psi$ are formulas then so are
$\neg\varphi$, $(\varphi\wedge\psi)$,
$\bo\varphi$,  $\bod\varphi$, and $\forall\varphi$,

\item if $\Delta$ is a non-empty finite set of formulas then
$\dit\Delta$ and $\didt\Delta$ are formulas,

\item if  $q\in \Var$
and $\varphi$ is a formula that is \emph{positive in $q$} (that is, every
free
occurrence of $q$ as an atomic subformula of $\varphi$
is in the scope of an even number of negations in $\varphi$;
\emph{free} means `not in the scope of any $\mu q$ in $\psi$'), then
$\mu q\varphi$ is a formula, in which all occurrences of $q$ are \emph{bound}.
Bound atoms arise only in this way.
\end{enumerate}
For formulas $\varphi,\psi$, and $q\in\Var$,
the expression $\varphi(\psi/q)$ denotes the result of replacing every 
free occurrence of $q$ in $\varphi$ by $\psi$, where the result is well-formed
--- that is, all of its subformulas of the form $\mu p\theta$ are such
that $\theta$ is positive in $p$.
For example, if $\varphi=\mu p\, q$ then $\varphi(\neg p/q)=\mu p\,\neg p$ is not well-formed.

We use standard abbreviations:
$\bot$ denotes $\neg\top$,
$(\varphi\vee\psi)$
denotes $\neg(\neg\varphi\wedge\neg\psi)$,
$(\varphi\to\psi)$ denotes $\neg(\varphi\wedge\neg\psi)$,
$(\varphi\leftrightarrow\psi)$ denotes
$(\varphi\to\psi)\wedge(\psi\to\varphi)$,
$\di\varphi$ denotes $\neg\bo\neg\varphi$,
$\did\varphi$ denotes $\neg\bod\neg\varphi$,
$\exists\varphi$ denotes $\neg\forall\neg\varphi$,
and if $\varphi$ is positive in $q$ then
$\nu q\varphi$ denotes $\neg\mu q\neg\varphi(\neg q/q)$ (this is well-formed).
For a non-empty finite set $\Delta=\{\vec\delta n\}$ of formulas,
we let $\bigwedge\Delta$ denote $\delta_1\wedge\ldots\wedge\delta_n$
and $\bigvee\Delta$ denote $\delta_1\vee\ldots\vee\delta_n$
(the order and bracketing of the conjuncts and disjuncts will always be immaterial). 
We set $\bigwedge\emptyset=\top$ and $\bigvee\emptyset=\bot$.
Parentheses will be omitted where possible, by the usual methods.

The connectives $\dit,\didt$ are called \emph{tangle connectives,}
or (more fully) \emph{tangled closure operators.}

We will be using various \emph{sublanguages} of $\Lbig$,
and they will be denoted in the obvious way by omitting prohibited operators
from the notation.
So for example, $\c L^{\mu}_{\bo\forall}$
denotes the language consisting of all $\Lbig$-formulas
that do not involve $\bod,$ $\dit$, or $\didt$.

\subsection{Kripke semantics}\label{ss:Kripke sem}

An \emph{assignment} or \emph{valuation} into a frame $\c F=(W,R)$ is
a map $h:\Var\to\wp(W)$.
A \emph{Kripke model} is a triple $\c M=(W,R,h)$, where $(W,R)$ is a frame
and $h$ an assignment into it.
The \emph{frame of $\c M$} is $(W,R)$, and
we say that $\c M$ is finite, reflexive, transitive, etc., if its frame is.

For every Kripke model $\c M=(W,R,h)$ and every world $w\in W$, we define the notion 
$\c M,w\models\varphi$ of a formula $\varphi$ of
$\Lbig$ being \emph{true at $w$ in $\c M$.}
The definition is by induction on $\varphi$,
as follows:
\begin{enumerate}
\item $\c M,w\models p$ iff $w\in h(p)$, for $p\in\Var$.

\item $\c M,w\models\top$.

\item $\c M,w\models\neg\varphi$ iff $\c M,w\not\models\varphi$.

\item $\c M,w\models\varphi\wedge\psi$ iff $\c M,w\models\varphi$ and $\c M,w\models\psi$.

\item $\c M,w\models\bo\varphi$ iff $\c M,v\models\varphi$ for every
$v\in R(w)$.

\item The truth condition for $\bod\varphi$ 
is exactly the same as for $\bo\varphi$.

\item $\c M,w\models\forall\varphi$ iff $\c M,v\models\varphi$
for every $v\in W$.

\item\label{item: sem tangle kripke} $\c M,w\models\dit\Delta$ iff there
are worlds $w=w_0,w_1,\ldots\in W$ with $R(w_n,w_{n+1})$ for each $n<\omega$
and such that for each $\delta\in\Delta$ there are infinitely many $n<\omega$
with $\c M,w_n\models\delta$.

\item The truth condition for $\didt\Delta$ 
is exactly the same as for $\dit\Delta$.

\item The truth condition for $\mu q\varphi$ takes longer to explain.
For an assignment $h:\Var\to\wp(W)$  and $S\subseteq W$, define a new
assignment $h[S/q]:\Var\to\wp(W)$ by
\[
h[S/q](p)=
\begin{cases}
S,&\mbox{if }p=q,
\\
h(p),&\mbox{otherwise,}
\end{cases}
\]
for $p\in\Var$.
Inductively,
the set $\sem\varphi_h=\{w\in W:(W,R,h),w\models\varphi\}$ is well defined, for every assignment $h$ into $(W,R)$.
Define a map $f:\wp(W)\to\wp(W)$  by
\[
f(S)=\sem\varphi_{h[S/q]}\quad\mbox{for }S\subseteq W.
\] 
Since $\varphi$ is positive in $q$,
it can be shown that
$f$ is monotonic, so it has a least fixed point, $LFP(f)$ (see section~\ref{ss:lfp}).
We define $\c M,w\models\mu q\varphi$ iff $w\in LFP(f)$.
\end{enumerate}
In the notation of the last clause, it can be checked that
$\c M,w\models\nu q\varphi$ iff
$w\in GFP(f)$.

A word on the semantics of $\did$ and $\didt$.
Let us temporarily write $\varphi\equiv\psi$ to mean 
that $\c M,w\models\varphi\leftrightarrow\psi$ for every \emph{transitive} Kripke model
$\c M=(W,R,h)$ and every $w\in W$.
Then it can be checked that for every non-empty finite set $\Delta$ of formulas,
\begin{equation}\label{e:mu and tangle}
\begin{array}{rcl}
\dit\Delta&\equiv&\displaystyle\nu q\bigwedge_{\delta\in\Delta}\di(\delta\wedge q),
\\[16pt]
\didt\Delta&\equiv&\displaystyle\nu q\bigwedge_{\delta\in\Delta}\did(\delta\wedge q),
\end{array}
\end{equation}
if $q\in\Var$ is a `new' atom that does not occur in any formula in $\Delta$.
For more details, see lemma~\ref{lem:mu trans}.
In a sense, \eqref{e:mu and tangle} is the `official' definition of the semantics of
the tangle connectives, which boils down to clause~\ref{item: sem tangle kripke} above
in the case of transitive Kripke models.

\subsection{Kripke semantics in generated submodels}

Let $\c M=(W,R,h)$ be a Kripke model.
A \emph{generated submodel} of $\c M$
is a model of the form
$\c M'=(W',R',h')$,
where $(W',R')$ is a generated subframe of $(W,R)$
and  $h':\Var\to\wp(W')$ is given
by $h'(p)=h(p)\cap W'$ for $p\in\Var$.
The following is an easy extension
to $\c L^{\mu\did\didt}_{\bo\bod}$
of a well known result in modal logic:
\begin{lemma}\label{lem:gen submodels}
Let $\c M'=(W',R',h')$ be a generated submodel of $\c M=(W,R,h)$.
Then for each $\varphi\in\c L^{\mu\did\didt}_{\bo\bod}$ and $w\in W'$, we have
\[
\c M,w\models\varphi\iff\c M',w\models\varphi.
\]
\end{lemma}

There is no distinction between $\bo$ and $\bod$ or between $\dit$ and $\didt$ in Kripke semantics.
This is not so in topological semantics, to be studied in part 2.

\subsection{Hilbert systems}

These are familiar, and we will be informal.
A \emph{Hilbert system $H$} in a given
language $\c L\subseteq\Lbig$ is a set of \emph{axioms,} which are $\c L$-formulas, and \emph{inference rules,} which have the form
\begin{equation}\label{e:inf rule}
\infrule{\vec\varphi n}\psi,
\end{equation}
for $\c L$-formulas $\vec\varphi n,\psi$.
A \emph{derivation in $H$ (of length $l$)} is a sequence $\vec\varphi l$ 
of $\c L$-formulas
such that each $\varphi_i$ $(1\leq i\leq l$) is either an $H$-axiom or
is derived from earlier $\varphi_j$ by an $H$-rule --- that is,
there are $1\leq\vec jn<i$ such that
\[
\infrule{\varphi_{j_1},\ldots,\varphi_{j_n}}{\varphi_i}
\]
is an instance of a rule of $H$.

A \emph{theorem of $H$} is a formula that occurs in some derivation in $H$.
An \emph{$H$-logic} is a set of $\c L$-formulas that
contains all $H$-axioms and is closed under 
all $H$-rules. The set of theorems of $H$ is the smallest $H$-logic.
Sometimes we identify (notationally) $H$ with this set,
or present $H$ implicitly by defining an $H$-logic.

A formula $\varphi$ is \emph{consistent} with $H$ if $\neg\varphi$ is not a theorem of $H$. A set $\Gamma$ of formulas is \emph{consistent} with $H$ if
$\bigwedge\Gamma_0$ is consistent with $H$, 
for every finite $\Gamma_0\subseteq\Gamma$.

Some familiar Hilbert systems used later are:
\begin{description}
\item [\boldmath$\axK$:] 
the axioms comprise (i) all instances of
propositional tautologies (e.g., $\varphi\to(\psi\to\varphi)$, etc.)
and (ii) all formulas of the form
$\bo(\varphi\to \psi)\to(\bo \varphi\to\bo \psi)$
(the so-called `normality' schema).
The
inference rules are:
\begin{itemize}
\item modus ponens:
$\infrule{\varphi,\;\varphi\to\psi}\psi$

\item  $\bo$-generalisation:
$\infrule \varphi{\bo\varphi}$
\end{itemize}

\item[\boldmath$\axK4$:] this is $\axK$ plus 
all instances of the `4' schema: $\bo \varphi\to\bo\bo \varphi$.

\item[\boldmath$\axS4$:] this is $\axK$ plus 
all instances of the  S4 schemata: $\bo \varphi\to \varphi$ and $\bo \varphi\to\bo\bo \varphi$. 
\end{description}
The well known substitution rule
$\infrule\varphi{\varphi(\psi/q)}$ is not always sound in the mu-calculus
and is not needed in other systems, so we omit it.

As usual, we denote particular Hilbert systems by sequences of letters and 
numbers indicating the axioms present.
For example, $\axS4.\axU\axC$ denotes the extension of $\axS4$ by the axioms generated by two 
schemes $\axU$ and $\axC$
to be seen later. 
The letter $t$ will denote the schemata for the tangle operator
given in section~\ref{ss:tangle logics}.

\subsection{Satisfiability, validity, equivalence}\label{ss:validity}

Let $\c F=(W,R)$ be a Kripke frame.
A set $\Gamma$ of $\Lbig$-formulas is said to be 
\emph{satisfiable in $\c F$}
if there exist an  assignment $h$ into $\c F$ and a world $w\in W$
such that $(W,R,h),w\models\gamma$ for every $\gamma\in\Gamma$.

Let $\varphi$ be an $\Lbig$-formula.
We say that $\varphi$ is \emph{satisfiable in $\c F$}
if the set $\{\varphi\}$ is so satisfiable.
We say that $\varphi$ is \emph{valid in $\c F$}
if $\neg\varphi$ is not satisfiable in $\c F$.
We may also say in this case that $\c F$ \emph{validates} $\varphi$.

We also say that $\varphi$ is \emph{equivalent} to a formula $\psi$
in $\c F$
if $\varphi\leftrightarrow\psi$ is valid in $\c F$.

\subsection{Logics}

Let $\c K$ be a class of Kripke frames.
In the context of a given language $\c L\subseteq\Lbig$,
the \emph{($\c L$)-logic of $\c K$} is the set of all $\c L$-formulas 
that are valid in every member of $\c K$.
A Hilbert system $H$ for $\c L$ 
whose set of theorems is $T$, say, is said to be 
\begin{itemize}
\item \emph{sound over $\c K$}
if  $T$ is a subset of the logic of $\c K$
(all $H$-theorems are valid in $\c K$),

\item  \emph{weakly complete}, or simply \emph{complete, over $\c K$}
if $T$ contains the logic of $\c K$
(all $\c K$-valid formulas are $H$-theorems),

\item \emph{strongly complete over $\c K$} if
every countable $H$-consistent set $\Gamma$ of $\c L$-formulas
is satisfiable in some \str\ in $\c K$.
(The restriction to countable sets will be discussed    
at the beginning of subsection \ref{subsec:strongcomp}.)
\end{itemize}
The logic of a single frame $\c F$ is defined to be
the logic of the class $\{\c F\}$; similar definitions are used for the other terms here. 

We say that a Kripke frame $\c F$  is  an \emph{$H$-frame,} or
that $\c F$ \emph{validates $H$,}
if $H$ is sound over~$\c F$.
To establish this, it is enough to check that each axiom of $H$ is valid in $\c F$,
and that each rule of $H$ preserves $\c F$-validity
(in the notation in \eqref{e:inf rule} above, this means that if $\vec\varphi n$ are valid in $\c F$ then so is $\psi$). 
We assume familiarity with basic results about modal validity:
for example, that a frame is a $\axK4$-frame iff it is transitive,
and an $\axS4$-frame iff it is reflexive and transitive.

It can be checked that $H$ is weakly complete
over $\c K$ iff every \emph{finite} $H$-consistent
set of formulas is satisfiable in some \str\ in $\c K$.
Hence, every strongly complete Hilbert system is also weakly complete.

 A system $H$  is said to have the \emph{finite model property over $\c K$} if each $H$-consistent formula is satisfiable
 in some \emph{finite} member of $\c K$. Equivalently, this means that $H$ is weakly complete over the class of finite members of $\c K$ (i.e.\ any formula valid in all finite members of $\c K$ is an $H$-theorem).

\newpage
\part{}

In this part of the paper, we look briefly at Hilbert systems
for the mu-calculus, but mainly we establish the finite model property
for the logics of concern in the paper.

\section{Hilbert systems for mu-calculus}\label{ss:hs mu}
We now present a very brief diversion on 
a Hilbert system for the mu-calculus that is sound and complete
over the class of finite reflexive transitive Kripke frames. 
It will be used to translate $\mu$ to $\dit$ 
and to axiomatise the $\c L^\mu_\bo$-logic of \nices\ spaces.
In this section, all formulas are $\c L^\mu_\bo$-formulas,
 all Hilbert systems are for this language, and we assume that $\Var$ is infinite.

\begin{definition}\label{def:mu systems}
Consider  the two Hilbert systems:
\begin{description}
\item [\boldmath$\axK\mu$:] standard modal logic $\axK$ with 
the axioms comprising all instances of propositional tautologies
and of normality
($\bo(\varphi\to \psi)\to(\bo \varphi\to\bo \psi)$), and the
inference rules
modus ponens,   $\bo$-generalisation,
plus the following for each
formula $\varphi$ positive in $q$:
\begin{itemize}
\item fixed point axiom: $\varphi(\mu q\varphi/q)\to\mu q\varphi$,
provided that no free occurrence of
an atom in $\mu q\varphi$
gets bound in $\varphi(\mu q\varphi/q)$ --- consequently,
$\varphi(\mu q\varphi/q)$ is well formed
(the idea is roughly that $\mu q\varphi$ is
a pre-fixed point of $\varphi$)

\item fixed point rule:
$\infrule{\varphi(\psi/q)\to\psi}{\mu q\varphi\to\psi}$,
provided that no free occurrence of
an atom in $\psi$
gets bound in $\varphi(\psi/q)$ --- hence,
 $\varphi(\psi/q)$ is well formed
(the idea this time is roughly that $\mu q\varphi$ is the least pre-fixed point of $\varphi$).
\end{itemize}
We write $\axK\mu\vdash\varphi$ if $\varphi$ is a theorem of this system.
It is well known (see, e.g., \cite[\S6]{BS-mucalc07})
that the system is equivalent to
the original equational system of Kozen \cite{Koz:mu82}.

\item[\boldmath$\axS4\mu$:] this is $\axK\mu$ plus the S4 
schemata $\bo \varphi\to\varphi$, $\bo\varphi\to\bo\bo\varphi$.
We write $\axS4\mu\vdash\varphi$ if $\varphi$ is a theorem of this system.
\end{description}
\end{definition}

The following combines some famous and difficult work in the mu-calculus.

\begin{fact}[\cite{Koz:mu82,Wal95,DBLP:conf/concur/JaninW96}]\label{fact:major mu}
$\axK\mu$ is sound and complete
over the class of all finite Kripke frames.
\end{fact} 

We are going to extend it to show that $\axS4\mu$ is sound and complete over the class of finite reflexive transitive frames (and, 
in Part 2, over every \nice\ space).
First, a form of the substitution rule can be established.

\begin{lemma}\label{lem:limited sub}
Suppose $\varphi,\psi$ are formulas such that 
for each atom $s$ occurring free in $\psi$,
there is no subformula of $\varphi$ of the form $\mu s\theta$.
If $\axS4\mu\vdash\varphi$, then $\axS4\mu\vdash\varphi(\psi/p)$ for any atom $p$.
\end{lemma}

\begin{proofsk}
Let $\varphi,\psi,p$ be as stipulated.
For a formula $\alpha$, write $\alpha^\dag=\alpha(\psi/p)$.
We show that $\axS4\mu\vdash\alpha\Rightarrow \axS4\mu\vdash\alpha^\dag$
(when the stipulation holds)
by induction on the length of a derivation of $\varphi$ in $\axS4\mu$.

Suppose that $\varphi$ is an instance
$\alpha(\mu q\alpha/q)\to\mu q\alpha$ of the fixed point axiom.
Then $\varphi^\dag$ is valid in all Kripke frames, so 
by fact~\ref{fact:major mu}, 
$\axK\mu\vdash\varphi^\dag$ and hence certainly $\axS4\mu\vdash\varphi^\dag$.

Suppose that $\varphi$ is derived by the fixed point rule,
so that $\varphi=\mu q\alpha\to\beta$ for some $\alpha,\beta,q$ meeting the condition of the rule, and $\alpha(\beta/q)\to\beta$ occurs earlier in the derivation.
If $s$ occurs free in $\psi$ then there is no $\mu s$ in 
$\mu q\alpha\to\beta$, so none in
$\alpha(\beta/q)\to\beta$ either.
So the inductive hypothesis applies,
to give $\axS4\mu\vdash(\alpha(\beta/q)\to\beta)^\dag$.
Let us evaluate this.
If $p=q$,  it is $\axS4\mu\vdash\alpha(\beta^\dag/q)\to\beta^\dag$.
By our stipulation,  the fixed point rule applies, giving
 $\axS4\mu\vdash\mu q\alpha\to\beta^\dag$.
But $(\mu q\alpha)^\dag=\mu q\alpha$. 
So $\axS4\mu\vdash\varphi^\dag$ as required.
If instead $p\neq q$,
then it is $\axS4\mu\vdash\alpha^\dag(\beta^\dag/q)\to\beta^\dag$.
Again, the rule applies, to give
$\axS4\mu\vdash\mu q\alpha^\dag\to\beta^\dag$. But this is exactly
$\axS4\mu\vdash\varphi^\dag$.

All other cases of the induction are easy and left to the reader.
\end{proofsk}

\begin{definition}\label{def:phi*}
For a formula $\varphi$, define a new formula $\varphi^*$
by induction:
\begin{itemize}
\item $p^*=p$ for $p\in\Var$;
\item 
$-^*$ \emph{commutes} with the boolean connectives and $\mu$.
That is, $\top^*=\top$,
$(\neg\varphi)^*=\neg\varphi^*$, 
$(\varphi\wedge\psi)^*=\varphi^*\wedge\psi^*$,
 and $(\mu q\varphi)^*=\mu q\varphi^*$.

\item  $(\bo\varphi)^*=\nu q(\varphi^*\wedge\bo q)$, where $q\in\Var$ is a `new'
atom not occurring in $\varphi^*$. 
\end{itemize}
The formula $\varphi^*$ is plainly well formed, for all $\varphi\in\c L^\mu_\bo$.
\end{definition}

\begin{lemma}\label{lem:semantic *}
Let $\varphi$ be any formula.
Then for  every Kripke model $(W,R,h)$ and $w\in W$, 
we have
$(W,R,h),w\models\varphi^*$ iff $(W,R^*,h),w\models\varphi$,
where (recall) $R^*$ is the reflexive transitive closure of $R$.
\end{lemma}

\begin{proof}
The proof is by induction on $\varphi$.
The atomic and boolean cases are easy.
Assuming the result for $\varphi$, 
it is a well known exercise in the mu-calculus to check that
$(W,R,h),w\models(\bo\varphi)^*$ iff
$(W,R,h),u\models\varphi^*$ for every $u\in R^*(w)$.
Inductively, this is iff  
$(W,R^*,h),u\models\varphi$ for every $u\in R^*(w)$,
iff $(W,R^*,h),w\models\bo\varphi$ as required.

Finally assume that the result holds for $\varphi$, positive in $q$,
for every Kripke model.
For a formula $\psi$ and Kripke model $(W,R,h)$,
write $\sem\psi_{(W,R,h)}=\{w\in W:(W,R,h),w\models\psi\}$.
Then
$(W,R,h),w\models(\mu q\varphi)^*$
iff $(W,R,h),w\models\mu q\varphi^*$, iff $w$ is in the least fixed point of the map $f:\wp(W)\to\wp(W)$
given by $f(S)=\sem{\varphi^*}_{(W,R,h[S/q])}$.
But inductively, $f(S)=\sem{\varphi}_{(W,R^*,h[S/q])}$.
So this is iff $(W,R^*,h),w\models\mu q\varphi$ as required.
\end{proof}

\begin{lemma}\label{lem:*}
$\axS4\mu\vdash\varphi\leftrightarrow\varphi^*$ for every $\varphi$.
\end{lemma}

\begin{proof}
Again, the proof is by induction on $\varphi$. 
We write just `$\vdash$' for `$\axS4\mu\vdash$' in the proof.
We also write 
$\alpha\equiv\beta$ for $\vdash\alpha\leftrightarrow\beta$.
First, replace all bound atoms in $\varphi$ by fresh ones,
to give a formula $\overline\varphi$.
More formally, $\overline\psi$ is defined for each subformula $\psi$
of $\varphi$ by induction:
 $\overline{\mu q\psi}=\mu s(\overline\psi(s/q))$, where $s$ is a 
 new atom associated with $\psi$ and not occurring in $\varphi$, and
$\overline{\,\cdot\,}$ commutes with all other operators.
By fact~\ref{fact:major mu},  $\overline\varphi\equiv\varphi$
and $(\overline\varphi)^*\equiv\varphi^*$.
So, replacing $\varphi$ by $\overline\varphi$, we can
suppose without loss of generality that
for each atom $q$ that occurs free in $\varphi$, 
there is no subformula of $\varphi$ of the form $\mu q\theta$.
The $-^*$ operator preserves this condition, so it holds for $\varphi^*$ as well.

For atomic $\varphi$, the result is trivial since $\varphi^*=\varphi$, and
booleans are fine.

Assume inductively that $\varphi\equiv\varphi^*$ and consider
$\bo\varphi$.
We need to show that $\bo\varphi\equiv\nu q(\varphi^*\wedge\bo q)$,
for `new' $q$
--- that is, 
$\bo\varphi\equiv\neg\mu q\neg(\varphi^*\wedge\bo \neg q)$.
By a tautology, it is enough to show
 $\neg\bo\varphi\equiv\mu q\neg(\varphi^*\wedge\bo \neg q)$.
 By fact~\ref{fact:major mu},
$\neg\bo\varphi\equiv\di\neg\varphi$ and
$\mu q\neg(\varphi^*\wedge\bo \neg q)\equiv\mu q(\neg\varphi^*\vee\di q)$.
So, letting $\psi=\neg\varphi$,
it is enough to prove
\begin{equation}\label{e:mu to prove}
\di\psi\equiv\mu q\chi,\mbox{ where }\chi=\psi^*\vee\di q.
\end{equation}
Note that the inductive hypothesis gives $\psi\equiv\psi^*$,
and that $\chi(\theta/q)$ is well-formed for any well-formed $\theta$.
Let $\chi^0=\bot$, and $\chi^{n+1}=\chi(\chi^n/q)$ for $n<\omega$.
The following claim, needed only for $n=2$, is an instance of 
a more general result.

\claim  $\vdash \chi^n\to\mu q\chi$ for each $n<\omega$.

\pfclaim By induction on $n$.
For $n=0$, it is $\vdash\bot\to\mu q\chi$, a tautology.
Assume inductively that  $\vdash \chi^n\to\mu q\chi$.
We desire  $\vdash \psi^*\vee\di \chi^n\to\mu q\chi$.
By the fixed point axiom, it is enough to prove that
$\vdash \psi^*\vee\di \chi^n\to\chi(\mu q\chi/q)$
--- that is,
$\vdash \psi^*\vee\di \chi^n\to\psi^*\vee\di\mu q\chi$.
But the inductive hypothesis plus standard uses of
generalisation and normality yield
$\vdash \di\chi^n\to\di\mu q\chi$, and the result follows using
tautologies and modus ponens.
This proves the claim.

\smallskip

Towards \eqref{e:mu to prove}, we first show that $\vdash\di\psi\to\mu q\chi$.
Observe that inductively, $\chi^1=\psi^*\vee\di\bot\equiv\psi$ and
$\chi^2=\psi^*\vee\di\chi^1\equiv \psi\vee\di\psi$.
By the claim for $n=2$, and tautologies,
$\vdash \psi\vee\di\psi\to\mu q\chi$ and applying more tautologies yields
$\vdash\di\psi\to\mu q\chi$.

Now we show $\vdash\mu q\chi\to\di\psi$.
By the fixed point rule, it is enough to show $\vdash\chi(\di\psi/q)\to\di\psi$.
That is, $\vdash\psi^*\vee\di\di\psi\to\di\psi$.
But given the inductive hypothesis, this is just what the S4 axioms say.
This proves \eqref{e:mu to prove} and completes the case of $\bo\varphi$.

Finally assume the result for $\varphi$ positive in $q$, and consider the case 
$\mu q\varphi$.
All formulas below meet all necessary conditions because
of our initial assumption on $\varphi$.
By the inductive hypothesis and lemma~\ref{lem:limited sub} we get 
$\vdash\varphi(\mu q\varphi^*/q)\to\varphi^*(\mu q\varphi^*/q)$.
The fixed point axiom gives 
$\vdash\varphi^*(\mu q\varphi^*/q)\to\mu q\varphi^*$.
Putting the two together gives $\vdash\varphi(\mu q\varphi^*/q)\to\mu q\varphi^*$.
This says that $\mu q\varphi^*$ is a pre-fixed point of $\varphi$,
so the fixed point rule gives
$\vdash\mu q\varphi\to\mu q\varphi^*$. The converse, $\vdash\mu q\varphi^*\to\mu q\varphi$, is similar.
\end{proof}

\begin{theorem}\label{thm:S4mu sc}
The system $\axS4\mu$ is sound and complete
over the class of finite reflexive transitive Kripke frames
(finite S4 frames).

\end{theorem}

\begin{proof}
Soundness is easily checked.
Conversely, assume that $\varphi$ is consistent with $\axS4\mu$.
By lemma~\ref{lem:*}, $\varphi^*$ is consistent with $\axS4\mu$ and hence with
$\axK\mu$ as well.
By fact~\ref{fact:major mu}, there is a finite Kripke model $\c M=(W,R,h)$ in which $\varphi^*$ is satisfied at $w$, say.
We do not know that $(W,R)$ is reflexive or transitive.
However, by lemma~\ref{lem:semantic *} we have $(W,R^*,h),w\models\varphi$ 
as well, and $R^*$ is reflexive and transitive.
\end{proof}

\section{Finite model property}\label{sec:fmp}

% !TEX root = paper.tex

The main work of our paper starts here.
In this section, we establish a number of finite model property results
for sublanguages of $\c L^{\dit\didt}_{\bo\bod\forall}$, by modifying a filtration approach 
pioneered in the context of $\c L_\bod$ by Shehtman \cite{Sheh:d90}
and used later by Lucero-Bryan for $\c L_{\bod\forall}$ \cite{LucBry11}.
The finite model property for the systems  $\axK\axD4\axG_n$ (and others)
was proved by Zakharyaschev \cite{Zakh93}, using canonical formulas.
 The finite model property for an S4-like tangle system
 was proved by Fern\'andez-Duque in \cite{FD:ijcai11}, by a different method,
 and the scheme {\bf Fix} and a variant of {\bf Ind} 
 in section~\ref{ss:tangle logics} below
appear in \cite[\S3]{FD:ijcai11}.

\subsection{Clusters in Transitive Frames}    
We work within models on \emph{K4 frames} $(W,R)$, i.e.\ $R$ is a  transitive binary relation on $W$. 
If $xRy$, we may say that $y$ comes \emph{$R$-after} $x$, or is \emph{$R$-later than} $x$, or is an \emph{$R$-successor} of $x$. If $xR^\bullet y$, i.e.\ $xRy$ but not $yRx$, then $y$ is \emph{strictly} after/later, or is a \emph{proper} $R$-successor. A point $x$ is \emph{reflexive} if $xRx$, and \emph{irreflexive} otherwise. $R$ is (ir)reflexive on a set $X\sub W$ if every member of $X$ is (ir)reflexive.

An \emph{$R$-cluster} is a subset $C$ of $W$ that is an equivalence class under the equivalence relation
$$
\{(x,y):x=y\text{ or } xRyRx\}.
$$
A cluster is \emph{degenerate} if it is a singleton $\{x\}$ with $x$ irreflexive. Note that a cluster $C$ can only contain an irreflexive point if it is a singleton. For,  if $C$ has more than one element, then for each $x\in C$ there is some $y\in C$ with $x\ne y$, so $xRyRx$ and thus $xRx$ by transitivity. On a non-degenerate cluster $R$ is universal. For $C$ to be non-degenerate it suffices that there exist $x,y\in C$ with $xRy$, regardless of whether $x=y$ or not.

Write $C_x$ for the $R$-cluster containing $x$. Thus $C_x=\{x\}\cup\{y:xRyRx\}$. The relation $R$ lifts to a well-defined \emph{partial} ordering of clusters   by putting $C_xRC_y$ iff $xRy$.  A cluster $C$ is \emph{$R$-maximal} when there is no cluster that  comes strictly $R$-after it, i.e.\ when $CRC'$ implies $C=C'$. A point $x\in W$ is \emph{$R$-maximal,}
or just \emph{maximal} if $R$ is understood, if $C_x$ is a maximal cluster, or equivalently if $xRy$ implies $yRx$.

An \emph{$R$-chain} is a sequence $C_1,C_2,\dots$ of pairwise distinct clusters with $C_1RC_2R\cdots$. In a finite frame, such a chain is of finite length. Hence we can define a notion of \emph{rank} in a finite frame by declaring the rank of a cluster $C$ to be the number of clusters in the longest chain of clusters starting with $C$. So the rank is always $\geq 1$, and a rank-1 cluster is maximal. The rank of a point $x$ is defined to be the rank of $C_x$. The key property of this notion is that if $xR^\bullet y$, equivalently if  $C_y$ comes strictly $R$-after $C_x$, then $y$ has smaller rank than $x$. 

An \emph{endless $R$-path} is a sequence $\{x_n:n<\omega\}$ such that $x_nRx_{n+1}$ for all $n$. Such a path \emph{starts at/from} $x_0$. The terms of the sequence need not be distinct: for instance, any reflexive point $x$ gives rise to the endless $R$-path  $xRxRxR\dots$.  In a finite frame, an endless path must eventually enter some \emph{non-degenerate} cluster $C$ and stay there, i.e.\ there is some $n$ such that $x_m\in C$ for all $m\geq n$.

Recall that $R(x)=\{y\in W:xRy\}$ is the set of $R$-successors of $x$, and that $(W',R')$ is an \emph{inner} subframe of 
$(W,R)$ if  $(W',R')$  is a subframe of  $(W,R)$ that is \emph{$R$-closed}. This means that $R'$ is the restriction of $R$ to $W'\sub W$, and  $x\in W' $ implies $R(x)\sub W'$. In this situation every $R'$-cluster is an $R$-cluster, and every $R$-cluster that intersects $W'$ is a subset of $W'$ and is an $R'$-cluster.

  \subsection{Syntax and Semantics}
  We will work initially in the language $\c L_\bo^\dit$.
Recall that we assume a set $\Var$ of propositional variables, which may be finite or infinite. Formulas are constructed from these variables by the standard Boolean connectives, the unary modality $\Box$ (with dual $\Dim$)
and the \emph{tangle} connective $\Di$ which assigns a formula  $\Di\G$ to each finite set $\G$ of formulas.
  
  Later we will want to add additional connectives, such as the universal modality $\forall$ and its dual $\exists$.
  
  We use the standard notion from section~\ref{ss:Kripke sem} of a Kripke model $\M=(W,R,h)$  on a (transitive) frame as given by a valuation function $h:\Var\to\wp W$, giving rise to a truth/satisfaction relation $\M,x\models\ph$ with $\M,x\models p$ iff $x\in h(p)$ for all $p\in\Var$ and $x\in W$.
The modality $\Dim$ is modelled by $R$ in the usual Kripkean way:
\begin{equation}\label{krip}
 \text{$\M,x\models\Dim\ph$ iff there is a $y$ with $xRy$ and $y\models \ph$.}
 \end{equation} 
 The condition for $\M,x\models\Di\G$ is that
 \begin{quote}
there exists an endless $R$-path $\{x_n:n<\omega\}$ with $x=x_0$
along which each member $\g$ of $\G$ is true infinitely often, i.e.\ $\{n<\omega:\M,x_n\models\g\}$ is infinite.
\end{quote}
 A set $\G$ of formulas is  \emph{satisfied by the cluster $C$} if  each member of $\G$ is true in $\M$ at some point of $C$.  So $\G$ fails to be satisfied by $C$ if some member of $\G$ is false at every point of $C$.  In a \emph{finite} model,
 since an endless path must eventually enter some non-degenerate cluster  and stay there,
  we get that 
 \begin{equation}\label{semants}
 \text{$x\models\Di\G$ iff there is a $y$ with $xRy$ \textbf{and} {\boldmath{$yRy$}} and $\G$ is satisfied by $C_y$ }
 \end{equation}
To put this another way, $x\models\Di\G$ iff $\G$ is satisfied by some \emph{non-degenerate} cluster following $C_x$.
 
Write $\Di\ph$ for the formula $\Di\{\ph\}$. Then $\Di\ph$ is true at $x$ iff there is an endless path starting at $x$ along which $\ph$ is true infinitely often. For finite models we have
\begin{equation*}\label{}
 \text{$x\models\Di\ph$ iff there is a $y$ with $xRy$ and $yRy$ and $y\models \ph$,}
 \end{equation*}
 i.e.\ the meaning of $\Di\ph$ is that there is a \emph{reflexive} alternative at which $\ph$ is true.
Thus for finite \emph{reflexive} models (i.e.\ S4 models) this reduces to the standard Kripkean interpretation  \eqref{krip} of $\Dim$. More strongly, it is evident that $\Di\ph\leftrightarrow\Dim\ph$ is valid in all S4 frames (and $\Di\ph\to\Dim\ph$ is valid in all K4 frames).

Write $\Dim^*\ph$ for the formula $\ph\lor\Dim\ph$, and  $\Box^*\ph$ for $\ph\land\Box\ph$. In any transitive frame, define  $R^*=R\cup\{(x,x):x\in W\}$. Then $R^*$ is the reflexive-transitive closure of $R$, and in any model on the frame we have
 \begin{equation*}\label{}
 \text{$\M,x\models\Box^*\ph$ iff for all $y$,  if $xR^*y$ then $\M,y\models \ph$.}
 \end{equation*} 
and
 \begin{equation*}\label{}
 \text{$\M,x\models\Dim^*\ph$ iff for some $y$,  $xR^*y$ and $\M,y\models \ph$.}
 \end{equation*} 
Note that if $C_x=C_y$, then $xR^*y$. For each $x$ let $R^*(x)=\{y\in W:xR^*y\}$. Then $R^*(x)=\{x\}\cup R(x)$.

\subsection{Tangle Systems and Logics}\label{ss:tangle logics}

A \emph{tangle system}  is any Hilbert system whose axioms
 include all tautologies and all instances of the schemes
\begin{description}
\item[K:]
$\Box(\ph\to\psi)\to(\Box\ph\to\Box\psi)$
\item[4:]
 $\Dim\Dim\ph\to\Dim\ph$
\item[Fix:]
$\Di\G\to  \Dim(\g\land\Di\G)$, \quad all $\g\in\G$.
\item[Ind:]
$\Box^*(\ph\to \bigwedge_{\g\in\G}\Dim(\g\land\ph))\to(\ph\to\Di\G)$.
\end{description}
and whose rules include  modus ponens and $\Box$-generalisation. The smallest tangle system will be denoted K4$t$.

A \emph{tangle logic} (or just \emph{logic} in this section) is a set $L$ of formulas 
that is a K4$t$-logic.
Any logic includes the following:

\begin{description}
\item[\ ]
\quad$\Di\ph\to\Dim\ph$
\item[4$_*$: ] 
$\Dim\Dim^*\ph\to\Dim\ph$
\item[4$_t$: ] 
$\Dim\Di\G\to\Di\G$
\end{description}
$4_t$ will be explicitly needed in our finite model property proof, in relation to a condition  called (r4).
Here is a derivation of $4_t$, in which the justification ``Bool'' means by principles of Boolean logic, ``Reg'' is the rule \emph{from $\ph\to\psi$ infer $\Dim\ph\to\Dim\psi$}, and ``Nec'' is the rule \emph{from $\ph$ infer $\Box^*\ph$.}

For each $\g\in\G$ we derive

\bigskip
\begin{tabular}{ll}
1. $\Di\G\to  \Dim(\g\land\Di\G)$   & Fix
\\
2. $ \Dim(\g\land\Di\G)\to \Dim\Di\G$  &K-theorem (Bool + Reg)
\\
3. $\Di\G\to \Dim\Di\G$  &1, 2 Bool
\\
4. $\g\land\Di\G\to\g\land \Dim\Di\G$  &3, Bool
\\
5. $\Dim(\g\land\Di\G)\to\Dim(\g\land \Dim\Di\G)$  &4, Reg
\\
6. $\Di\G\to \Dim(\g\land \Dim\Di\G)$  &1, 5 Bool
\\
7.  $\Dim\Di\G\to \Dim\Dim(\g\land \Dim\Di\G)$  &6, Reg
\\
8. $\Dim\Di\G\to\Dim(\g\land \Dim\Di\G)$ &7, \textbf{Axiom 4}, Bool
\end{tabular}

\bigskip\noindent
Since this holds for every $\g\in\G$ we can continue with

\bigskip
\begin{tabular}{ll}
\, 9. $\Dim\Di\G\to \bigwedge_{\g\in\G}\Dim(\g\land \Dim\Di\G)$ &8 for all $\g\in\G$, Bool
\\
10. $\Box^*(\Dim\Di\G\to\bigwedge_{\g\in\G}\Dim(\g\land \Dim\Di\G))$ &9, Nec
\\
11. $\Box^*(\Dim\Di\G\to\bigwedge_{\g\in\G}\Dim(\g\land \Dim\Di\G))\to (\Dim\Di\G\to\Di\G)$ & Ind with $\ph=\Dim\Di\G$
\\
12. $\Dim\Di\G\to\Di\G$ &10, 11 Bool
\end{tabular}

\subsection{Canonical Frame}\label{ss:can frame}
For a tangle logic $L$, the canonical frame is 
 $\F_L=(W_L,R_L)$, with $W_L$ the set of maximally $L$-consistent sets of formulas, and $xR_Ly$ iff $\{\Dim\ph:\ph\in y\}\sub x$ iff $\{\ph:\Box\ph\in x\}\sub y$. $R_L$ is transitive, by the K4 axiom 4.

Suppose $\F=(W,R)$ is an inner subframe of $\F_L$, i.e.\  $W$ is an $R_L$-closed subset of $W_L$,  and $R$ is the restriction of $R_L$ to $W$.

By standard canonical frame theory,  we have that for all formulas $\ph$ and all $x\in W$:
\begin{eqnarray}
\Dim\ph\in x &&\text{iff\quad for some } y\in W, \ xRy \text{ and }\ph\in y.  \label{Dican}
\\
\Dim^*\ph\in x &&\text{iff\quad for some } y\in W, \ xR^*y \text{ and }\ph\in y.  \label{Distar}
\\
\label{Boxcan}
\Box\ph\in x &&\text{iff\quad  for all } y\in W, \ xRy \text{ implies }\ph\in y.
\\
\label{Boxstar}
\Box^*\ph\in x &&\text{iff\quad for all } y\in W, \ xR^*y \text{ implies }\ph\in y.
\end{eqnarray}

We will say that a sequence $\{x_n:n<\omega\}$ in $\F$ \emph{fulfils} the formula $\Di\G$ if 
each member of $\G$ belongs to $x_n$ for infinitely many $n$. The role of the axiom Fix is to provide such sequences:

\begin{lemma}\label{one}
In $\F$, if $\Di\G\in x$ then there is an endless $R$-path starting from $x$ that fulfils $\Di\G$. Moreover, $\Di\G$ belongs to every member of this path. 
\end{lemma}

\begin{proof}
Let $\G=\{\g_1,\dots,\g_k\}$.
Put $x_0=x$. From $\Di\G\in x_0$ by axiom Fix 
%\marginpar{\small \em Fix} 
we get $\Dim(\g_1\land\Di\G)\in x_0$, so by \eqref{Dican} there exists $x_1\in W$ with $x_0Rx_1$ and $\g_1,\Di\G\in x_1$. Since $\Di\G\in x_1$, by Fix again there exists $x_2\in W$ with $x_1Rx_2$ and $\g_2,\Di\G\in x_2$. Continuing in this way ad infinitum cycling through the list $\g_1,\dots,\g_k$ we generate a sequence fulfilling $\Di\G$, with  $\g_i\in x_n$ whenever $n\equiv i\mod k$, and  $\Di\G\in x_n$ for all $n<\omega$.
\end{proof}

The canonical model $\M_L$ on $\F_L$ has $\M_L,x\models \ph$ iff $\ph\in x$, \emph{provided that $\ph$ is $\Di$-free}.
But this `Truth Lemma' can fail for formulas containing the tangle connective, even though all instances of the tangle axioms belong to every member of $W_L$. For this reason we will work directly with the structure of $\F_L$ and the relation
 $\ph\in x$, rather than with truth in $\M_L$.
 
 For an example of failure of the Truth Lemma, consider the set
 $$
 \Sigma=\{p_0,q,\Box(p_{2n}\to\Dim(p_{2n+1}\land\neg q)),\Box(p_{2n+1}\to\Dim(p_{2n+2}\land q)):n<\omega\},
 $$
 where $q$ and the $p_n$'s are distinct variables.
Each finite subset of 
$\Sigma\cup\{\neg\Di\{q,\neg q\}\}$
is satisfiable in a transitive frame, and so is $L_{\axK4t}$-consistent where $L_{\axK4t}$ is the smallest logic. Explanation:
if $\G$ is a finite subset,
$\M$ a model with transitive frame, and $\M,x\models\G$, then
$\{\ph:\M,y\models\ph$ for all worlds $y$ of $\M\}$ is a logic that excludes $\neg\bigwedge\G$, so $\neg\bigwedge\G\notin L_{\axK4t}$.

Since the proof theory is finitary, it follows that $\Sigma\cup\{\neg\Di\{q,\neg q\}\}$ is $L_{\axK4t}$-consistent, so is included in some member $x$ of  $W_{L_{\axK4t}}$. Using the fact that $\Sigma\sub x$, together with \eqref{Dican} and \eqref{Boxcan}, we can construct an endless $R_{L_{\axK4t}}$-path starting from $x$ that fulfills $\{q,\neg q\}$, hence satisfies each of $q$ and $\neg q$ infinitely often in 
$\M_{L_{\axK4t}}$. Thus $\M_{L_{\axK4t}},x\models\Di\{q,\neg q\}$. But $\Di\{q,\neg q\}\notin x$, since  $\neg\Di\{q,\neg q\}\in x$ and $x$ is $L_{\axK4t}$-consistent.

\subsection{Definable Reductions}
Fix a finite set $\Phi$ of formulas closed under subformulas.
Let $\Phi^t$ be the set of all formulas in $\Phi$ of the form $\Di\G$, and
$\Phi^\Dim$ be the set of all formulas in $\Phi$ of the form $\Dim\ph$.

Let $\F=(W,R)$ be an inner subframe of $\F_L$. Then by a \emph{definable reduction of $\F$ via $\Phi$} we mean a pair $(\M_\Phi,f)$,  where $\M_\Phi=(W_\Phi,R_\Phi,h_\Phi)$ is a model on a finite transitive frame, and $f:W\to W_\Phi$
is a surjective function, such that the following hold for all $x,y\in W$:

\begin{enumerate}[(r1):]
\item 
$p\in x$ iff $f(x)\in h_\Phi(p)$, for all $p\in\Var\cap\Phi$.
\item
$f(x)=f(y)$ implies $x\cap\Phi= y\cap\Phi$.
\item
$xRy$ implies $f(x)R_\Phi f(y)$.
\item
$f(x)R_\Phi f(y)$ implies  $y\cap \Phi^t\sub x\cap\Phi^t$ and 
$\{\Dim\ph\in\Phi:\Dim^*\ph\in y\}\sub x$.
\item 
For each subset $C$ of $W_\Phi$ there is a formula $\ph$ that defines $f^{-1}(C)$ in $W$, i.e.\ $\ph\in y$ iff $f(y)\in C$.
\end{enumerate}
We will make crucial use of the following  consequence of this definition.

\begin{lemma} \label{import}
If $f(x)$ and $f(y)$ belong to the same $R_\Phi$-cluster, then $x\cap\Phi^t=y\cap\Phi^t$ and $x\cap\Phi^\Dim=y\cap\Phi^\Dim$.
\end{lemma}
\begin{proof}
If $f(x)=f(y)$, then $x\cap\Phi=y\cap\Phi$ by (r2) and so $x\cap\Phi^t=y\cap\Phi^t$ and $x\cap\Phi^\Dim=y\cap\Phi^\Dim$. But if  $f(x)\ne f(y)$, then $ f(x)R_\Phi f(y)R_\Phi f(x)$, and so 
 $y\cap \Phi^t\sub x\cap\Phi^t\sub y\cap\Phi^t$ by (r4).  Also if
 $\Dim\ph\in y\cap \Phi$ then $\Dim^*\ph=\ph\lor\Dim\ph\in y$, and so $\Dim\ph\in x$ by (r4), and likewise $\Dim\ph\in x\cap \Phi$ implies
 $\Dim\ph\in y$.
\end{proof}
Note that the second conclusion of (r4) is a concise way of expressing that both 
$$
\{\Dim\ph\in\Phi:\ph\in y\}\sub x\quad\text{and}\quad 
\{\Dim\ph\in\Phi:\Dim\ph\in y\}\sub x.
$$

Given a definable reduction  $(\M_\Phi,f)$ of $\F$,  we will  replace $R_\Phi$ by a weaker relation $R_t$, producing a new model 
$\M_t=(W_\Phi,R_t,h_\Phi)$, the \emph{untangling} of $\M_\Phi$, with the property that satisfaction  in  $\M_t$ of any formula $\ph\in\Phi$ corresponds exactly via $f$ to membership of $\ph$ in points of $\F$. In other words, $\ph\in x$ iff $\M_t,f(x)\models\ph$, a result we refer to as the \emph{Reduction Lemma}. The definition of $R_t$ will cause
each $R_\Phi$-cluster to be decomposed into a partially ordered set of smaller $R_t$-clusters.

In what follows we will write  $\ab{x}$ for $f(x)$. Then as $f$ is surjective, each member of $W_\Phi$ is equal to $\ab{x}$ for some $x\in W$. In later applications  the set $W_\Phi$ will  be a set of equivalence classes $\ab{x}$ of points $x\in W$, under a suitable equivalence relation, and $f$ will be the natural map $x\mapsto\ab{x}$.

Our first step  makes the key use of the axiom Ind:

\begin{lemma} \label{useind}
Let $\Di\G\in\Phi$. Suppose that $\Di\G\notin x$, where $x\in W$, and let $\ab{x}\in C\sub W_\Phi$. Then there is a formula $\g\in\G$ and some $y\in W$ such that $xR^*y$, $\ab{y}\in C$ and
\begin{equation} \label{final}
\text{if $yRz$ and $\ab{z}\in C$, then $\g\notin z$.}
\end{equation}
\end{lemma}

\begin{proof}
By (r5) there is a formula $\ph$ that defines $\{y\in W:\ab{y}\in C\}$, i.e.\ $\ph\in y$ iff $\ab{y}\in C$.
Then $\ph\in x$ and $\Di\G\notin x$, so by the axiom Ind,  % \marginpar{\small \em Ind} 
$
\Box^*(\ph\to \bigwedge_{\g\in\G}\Dim(\g\land\ph))\notin x
$.
Hence by \eqref{Boxstar} there is a $y$ with $xR^*y$ and $(\ph\to \bigwedge_{\g\in\G}\Dim(\g\land\ph))\notin y$.
Then
$\ph\in y$, so $\ab{y}\in C$, and  for some $\g\in\G$ we have $\Dim(\g\land\ph)\notin y$. Hence by \eqref{Dican}, if $yRz$ and  $\ab{z}\in C$, then $\g\land\ph\notin z$ and $\ph\in z$, so $\g\notin z$, which gives \eqref{final}.
\end{proof}

\begin{lemma} \label{extend}
Let formulas $\Di\G_1,\dots,\Di\G_k$ belong to $\Phi$ but not to $x$.  Suppose that  $\ab{x}\in C\sub W_\Phi$. Then there are formulas $\g_1\in\G_1,\dots,\g_k\in\G_k$ and some $y\in W$ such that $xR^*y$, $\ab{y}\in C$ and
\begin{equation} \label{final2}
\text{if $yRz$ and $\ab{z}\in C$, then $\{\g_1,\dots,\g_k\}\cap z=\emptyset$.}
\end{equation}
\end{lemma}

\begin{proof}
If $k=0$, take $y=x$; we are done.
Now assume $k>0$.
By Lemma \ref{useind}, there exists $\g_1\in\G_1$ and $y_1\in W$ such that $xR^*y_1$, $\ab{y_1}\in C$ and
\begin{equation} \label{finaly1}
\text{if $y_1Rz$ and $\ab{z}\in C$, then $\g_1\notin z$.}
\end{equation}
Now $\Di\G_2\notin x$, so $\Dim\Di\G_2\notin x$ by scheme $4_t$.
Hence $\Dim^*\Di\G_2=\Di\G_2\lor\Dim\Di\G_2\notin x$. As
$xR^*y_1$, this implies  $\Di\G_2\notin y_1$ by \eqref{Distar}.
So by Lemma \ref{useind} again, with $y_1$ in place of $x$, there exists $\g_2\in\G_2$ and $y_2\in W$ such that $y_1R^*y_2$, 
$\ab{y_2}\in C$ and
\begin{equation} \label{finaly2}
\text{if $y_2Rz$ and $\ab{z}\in C$, then $\g_2\notin z$.}
\end{equation}
Now by transitivity of $R^*$ we have $xR^*y_2$. Also if $y_2Rz$ and $\ab{z}\in C$, then from  $y_1R^*y_2Rz$ we get $y_1Rz$, and so $\g_1\notin z$ by \eqref{finaly1}. Together with \eqref{finaly2} this shows that  $\{\g_1,\g_2\}\cap z=\emptyset$.

If $k=2$ this proves \eqref{final2} with $y=y_2$. Otherwise we repeat, applying Lemma \ref{useind} again with $y_2$ in place of $x$ and so on, eventually obtaining the desired $y$ as $y_k$.
\end{proof}

Define a formula $\ph\in\Phi$ to be \emph{realised} at a member $\ab{z}$ of $W_\Phi$ iff $\ph\in z$. 
Note that this definition does not depend on how the member is  named, for if $\ab{z}=\ab{z'}$, then $z\cap\Phi=z'\cap\Phi$ by (r2), and so $\ph\in z$ iff $\ph\in z'$.

\begin{lemma}\label{notreal}
Let   $C$ be any $R_\Phi$-cluster. Then there is some $y\in W$ with $\ab{y}\in C$, such that for any formula 
$\Di\G\in \Phi^t-y$ there is a formula in $\G$ that is not realised at any $\ab{z}\in C$ such that $yRz$.
\end{lemma}
\begin{proof}
Take any $\ab{x}\in C$, and put
 $\Phi^t-x=\{\Di\G_1,\dots,\Di\G_k\}$. By Lemma \ref{extend} there is some $y$ with $xR^*y$ and $\ab{y}\in C$, and  formulas $\g_i\in\G_i$ for $1\leq i\leq k$ such that if $yRz$ and $\ab{z}\in C$, then $\g_i\notin z$, hence $\g_i$ is not realised at $\ab{z}$.
 
Now $\ab{x}$ and $\ab{y}$ belong to the same $R_\Phi$-cluster $C$, so $y\cap \Phi^t= x\cap\Phi^t$ by Lemma \ref{import}. Hence
 $\Phi^t-y=\Phi^t-x$.
 So if $\Di\G\in \Phi^t-y$, then $\G=\G_i$ for some $i$, and then $\g_i$ is a member of $\G$ not realised at any $\ab{z}\in C$ such that $yRz$.
\end{proof}

Now for each $R_\Phi$-cluster $C$, choose and fix a
 point $y$ as given by Lemma \ref{notreal}. Call $y$  the \emph{critical point for} $C$, and put
$$
C^\circ=\{\ab{z}\in C: yRz \}.
$$ 
Lemma \ref{notreal} states that if $\Di\G\in \Phi^t-y$, then there is a formula in $\G$ that is not realised at any point of 
$C^\circ$. 

We call $C^\circ$ the \emph{nucleus} of the cluster $C$. If $yRy$ then $\ab{y}\in C^\circ$, but in general $\ab{y}$ need not belong to $C^\circ$. Indeed the nucleus could be empty. For instance, it must be empty when $C$ is a degenerate cluster. To show this, suppose that $C^\circ\ne\emptyset$. Then there is some $\ab{z}\in C$ with $yRz$, hence $\ab{y}R_\Phi\ab{z}$ by (r3), so as $\ab{y}\in C$ this shows that $C$ is \emph{non}-degenerate. Consequently, if the nucleus is non-empty then the relation $R_\Phi$ is universal on it.

  We introduce the subrelation $R_t$ of $R_\Phi$ to refine the structure of $C$ by decomposing it into the nucleus
  $C^\circ$ as an $R_t$-cluster together with a singleton \emph{degenerate} $R_t$-cluster $\{w\}$ for each $w\in C-C^\circ$. These degenerate clusters all have $C^\circ$ as an $R_t$-successor but are incomparable  with each other. So the structure replacing $C$ looks like
%\begin{figure} \label{fignuc}
$$
\xymatrix{
*{\bullet} \ar[drr]^<{}  &*{\bullet} \ar[dr]^<{\textstyle\{w\}}  & {\qquad\cdots\cdots\cdots}   &*{\ \bullet^{}} \ar[dl]   \\
& &*{\xy ;<1pc,0pc>:\POS(0,0) +(0,-1.5)*+{C^\circ}*\cir<20pt>{} \endxy}  &{\hspace{-2.3cm}}
}
$$
%\caption{ 1}
%\end{figure}
with the  black dots being the degenerate clusters determined by the  points of $C-C^\circ$. 
Doing this to each cluster of $(W_\Phi,R_\Phi)$ produces a new transitive frame $\F_t=(W_\Phi,R_t)$ with $R_t\sub R_\Phi$.  

$R_t$ can be more formally defined on $W_\Phi$ simply  by specifying, for each $w,v\in W_\Phi$, that $wR_tv$ iff  $wR_\Phi v$ and either
\begin{itemize}
\item 
$w$ and $v$ belong to different $R_\Phi$-clusters; \enspace or
\item
$w$ and $v$ belong to the same $R_\Phi$-cluster $C$, and $v\in C^\circ$.
\end{itemize}
This ensures that each member of $C$ is $R_t$-related to every member of the nucleus of $C$. The restriction of $R_t$ to $C$ is equal to $C\times C^\circ$, so we could also define $R_t$ as the union of  the relations $C\times C^\circ$ for all $R_\Phi$-clusters $C$, plus all inter-cluster instances of $R_\Phi$.

If the nucleus is empty, then so is the relation $R_t$ on $C$, and $C$ decomposes into a set of pairwise incomparable degenerate clusters. If $C=C^\circ$, then $R_t$ is universal  on $C$, identical to the restriction of $R_\Phi$ to $C$.

\begin{lemma}[Reduction lemma]
Every formula in $\Phi$  is true in $\M_t=(W_\Phi,R_t,h_\Phi)$ precisely at the points at which it is realised, i.e.\ for all $\ph\in\Phi$ and all $x\in W$,
 \begin{equation}\label{filtlem}
\text{$\M_t,\ab{x}\models \ph$\enspace iff\enspace $\ph\in x$.}
 \end{equation}
\end{lemma}

\begin{proof}
This is by  induction on the  formation of formulas. For the base case of  a variable $p\in\Phi$, we have
$\M_t,\ab{x}\models p$ iff $\ab{x}\in h_\Phi(p)$, which holds iff $p\in x$ by (r1). The inductive cases of the Boolean connectives are standard.

 Next, take the case of a formula $\Dim\ph\in\Phi$,  under the induction hypothesis that \eqref{filtlem} holds for all $x\in W$.
 Suppose first that $\M_t,\ab{x}\models \Dim\ph$. Then there is some $y\in W$ with $\ab{x}R_t\ab{y}$ and $\M_t,\ab{y}\models \ph$, hence $\ph\in y$ by the induction hypothesis on $\ph$. Then $\Dim^*\ph\in y$. But $R_t\sub R_\Phi$, so  $\ab{x}R_\Phi\ab{y}$, implying that $\Dim\ph\in x$, as required, by  (r4).  Conversely, suppose that $\Dim\ph\in x$. Let $C$ be the $R_\Phi$-cluster of $\ab{x}$, and $y$ the critical point for $C$. Then $\Dim\ph\in y$ by Lemma \ref{import}, so there is some $z$ with $yRz$ and $\ph\in z$, hence $\M_t,\ab{z}\models \ph$ by induction hypothesis. Now if $\ab{z}\in C$, then $\ab{z}$ belongs to the nucleus of $C$ and hence $\ab{x}R_t\ab{z}$. But if $\ab{z}\notin C$, then as $\ab{y}R_\Phi\ab{z}$ by (r3), and hence $\ab{x}R_\Phi\ab{z}$, the $R_\Phi$-cluster of $\ab{z}$ is strictly $R_\Phi$-later than $C$, and again $\ab{x}R_t\ab{z}$. So in any case we have 
 $\ab{x}R_t\ab{z}$ and $\M_t,\ab{z}\models \ph$, giving $\M_t,\ab{x}\models \Dim\ph$. That completes this inductive case of $\Dim\ph$.
 
 Finally we have the most intricate case of a   formula $\Di\G\in\Phi$, under the induction hypothesis that  \eqref{filtlem} holds for every member of $\G$ for all $x\in W$. Then we have to show that for all $z\in W$, 
\begin{equation}\label{filt}
\text{
$\M_t,\ab{z}\models \Di\G$\enspace iff\enspace $\Di\G\in z$.}
\end{equation}
The proof  proceeds by strong induction on the \emph{rank} of $\ab{z}$.
Take $x\in W$ and suppose that \eqref{filt} holds for every $z$ for which the rank of $\ab{z}$ is  \emph{less than} the rank of $\ab{x}$. We show that
$\M_t,\ab{x}\models \Di\G$ iff $\Di\G\in x$. Let  $C$ be the $R_\Phi$-cluster of $\ab{x}$, and $y$  the critical point for $C$.

Assume first that $\Di\G\in x$. Then $\Di\G\in y$ by Lemma \ref{import}. By Lemma \ref{one}, there is an endless $R$-path $\{y_n:n<\omega\}$ starting from $y=y_0$ that fulfills $\Di\G$ and has $\Di\G$ belonging to each point. Then by (r3) the sequence 
$\{\ab{y_n}:n<\omega\}$ is an endless $R_\Phi$-path in $W_\Phi$ starting at $\ab{y}\in C$.

Suppose that $\ab{y_n}\in C$ for all $n$. Then for all $n>0$, since $yRy_n$ we get $\ab{y_n}\in C^\circ$. So there is the endless $R_t$-path $\pi=\ab{x}R_t\ab{y_1}R_t\ab{y_2}R_t\cdots$  starting at $\ab{x}$.
As $\{y_n:n<\omega\}$ fulfills $\Di\G$, for each $\g\in\G$ there are infinitely many $n$ for which $\g\in y_n$  and so $\M_t,\ab{y_n}\models\g$ by the induction hypothesis on members of $\G$. Thus each member of $\G$ is true  infinitely often along $\pi$, implying that $\M_t,\ab{x}\models \Di\G$.

If however there is an $n>0$ with $\ab{y_n}\notin C$, then the $R_\Phi$-cluster of $\ab{y_n}$ is strictly $R_\Phi$-later than $C$, so $\ab{x}R_t\ab{y_n}$ and $\ab{y_n}$ has smaller rank than $\ab{x}$. Since $\Di\G\in y_n$, the induction hypothesis \eqref{filt} on rank then implies that $\M_t,\ab{y_n}\models \Di\G$. So there is an endless $R_t$-path $\pi$ from $\ab{y_n}$ along which each member of $\G$ is true infinitely often. Since $\ab{x}R_t\ab{y_n}$, we can append $\ab{x}$ to the front of $\pi$ to obtain such an $R_t$-path starting from $\ab{x}$, showing that $\M_t,\ab{x}\models \Di\G$ (this last part is an argument for soundness of  $4_t$). So in both cases we get
$\M_t,\ab{x}\models \Di\G$. That proves the forward implication of \eqref{filtlem} for $\Di\G$.

For the converse implication, suppose $\M_t,\ab{x}\models \Di\G$. Since $W_\Phi$ is finite, it follows by \eqref{semants}  that there exists  a $z\in W$ with $\ab{x}R_t\ab{z}$ and $\ab{z}R_t\ab{z}$ and the $R_t$-cluster of $\ab{z}$ satisfies $\G$.
By the induction hypothesis \eqref{filtlem} on members of $\G$,
every formula in $\G$  is  realised at some point of this cluster.
Suppose first there is such a $z$ for which the rank of $\ab{z}$ is less than that of $\ab{x}$. Then  as the $R_t$-cluster of $\ab{z}$ is non-degenerate and satisfies $\G$, we have $\M_t,\ab{z}\models \Di\G$. Induction hypothesis \eqref{filt} then implies that $\Di\G\in z$. But  $\ab{x}R_\Phi\ab{z}$, as $\ab{x}R_t\ab{z}$, so by (r4) we get the required conclusion that  $\Di\G\in x$.

If however there is no such $z$ with $\ab{z}$ of lower rank  than $\ab{x}$, then the $\ab{z}$ that does exist must have the same rank as $\ab{x}$, so it belongs  to $C$. Hence as  $\ab{x}R_t\ab{z}$, the definition of $R_t$ implies that $\ab{z}\in C^\circ$.      
Thus  the $R_t$-cluster of $\ab{z}$ is $C^\circ$. Therefore every formula in $\G$  is  realised at some point of $C^\circ$, i.e.\ at some $\ab{z'}\in C$ with $yRz'$. But Lemma \ref{notreal} states that if $\Di\G\notin y$, then some member of $\G$ is not realised in $C^\circ$. Therefore we must have  $\Di\G\in y$. Then  $\Di\G\in x$ as required, by Lemma \ref{import}.
That finishes the inductive proof that $\M_t$ satisfies the Reduction Lemma. 
\end{proof}

\subsection{Adding Seriality} 

Suppose the logic $L$ contains the D-axiom $\Dim\top$. Then $R_L$  is \emph{serial}: $\forall x\exists y(xR_Ly)$.  Hence the relation $R$ of the inner subframe $\F$ is serial.
From this we can show that $R_t$ is serial. The key point is that any maximal $R_\Phi$-cluster $C$ must have a \emph{non-empty} nucleus. For, if $y$ is the critical point for $C$, then there is a $z$ with $yRz$, as $R$ is serial. But then $\ab{y}R_\Phi\ab{z}$ by (r3) and so $\ab{z}\in C$ as $C$ is maximal. Hence  $\ab{z}\in C^\circ$, making the nucleus non-empty. Now every member of $C$ is $R_t$-related to any member of $C^\circ$ so altogether this implies that $R_t$ is serial on the rank 1 cluster $C$. But any point of rank $>1$ will be $R_t$-related to points  of lower rank, and indeed to points in the nucleus of some rank 1 cluster. Since $R_t$ is reflexive on a nucleus,  this shows that $R_t$ satisfies the stronger condition that  $\forall w\exists v(wR_t vR_t v)$ --- ``every world sees a reflexive world''.

\subsection{Adding Reflexivity} 

Suppose that $L$ contains the scheme
\begin{description}
\item[T:]
$\ph\to\Dim\ph$.
\end{description}
Then it  contains
\begin{description}
\item[T$_t$: ] 
$\bigwedge\G\to\Di\G$.
\end{description}
To see this, let $\ph=\bigwedge\G$. Then
$
\ph\to \bigwedge_{\g\in\G}(\g\land\ph)
$
is a tautology, hence derivable. From that we derive
\begin{equation} \label{antind}
\Box^*(\ph\to{\textstyle \bigwedge}_{\g\in\G}\Dim(\g\land\ph))
\end{equation}
using  the instances $(\g\land\ph)\to\Dim(\g\land\ph)$ of axiom T and K-principles. But \eqref{antind} is an antecedent of axiom Ind, so we apply it  to derive  $\ph\to\Di\G$, which is T$_t$ in this case.

Axiom T ensures that the canonical frame relation $R_L$ is reflexive, and hence so is $R_\Phi$ by (r3). Thus no $R_\Phi$-cluster is degenerate.
We modify the definition of $R_t$ to make it  reflexive as well. The change occurs in the case of an $R_\Phi$-cluster $C$ having $C\ne C^\circ$. Then instead of making the singletons $\{w\}$ for $w\in C-C^\circ$ be degenerate, we make them all into \emph{non}-$R_t$-degenerate clusters by requiring that $wR_tw$. Formally this is done by adding to the definition of $wR_tv$ the third possibility that
\begin{itemize}
\item 
$w$ and $v$ belong to the same $R_\Phi$-cluster $C$, and $w=v\in C-C^\circ$.
\end{itemize}
Equivalently, the restriction of $R_t$ to $C$ is equal to $(C\times C^\circ) \cup\{(w,w):w\in C-C^\circ\}$.

The proof of the Reduction Lemma  for the resulting reflexive and transitive model $\M_t$ now requires an adjustment in one place, in its last paragraph, where $\ab{x}R_t\ab{z}\in C$. In the original proof above, this implied that the $R_t$-cluster of $\ab{z}$ is $C^\circ$. But now we have the new possibility  that $\ab{x}=\ab{z}\in C-C^\circ$. Then the $R_t$-cluster of $\ab{z}$ is $\{\ab{z}\}$, so every formula of $\F$ is realised at $\ab{z}$, implying $\bigwedge\G\in z$.
The  scheme T$_t$ now
 ensures that $\Di\G\in z$, so  by Lemma \ref{import}  we still get the required result that $\Di\G\in x$, and the Reduction Lemma still holds for this modified reflexive version of $\M_t$.

 \subsection{Finite model property over K4, KD4 and S4}\label{sec:fmp S4}
 
 Given a logic $L$ and a finite set $\Phi$ of formulas closed under subformulas, 
we can construct a definable reduction of any  inner subframe  $\F=(W,R)$ of $\F_L$    by  filtration through 
$\Phi$. An equivalence relation $\sim$ on $W$ is given by putting $x\sim y$ iff $x\cap\Phi=y\cap\Phi$. 
Then with $\ab{x}=\{y\in W:x\sim y\}$ we put  $W_\Phi=\{\ab{x}:x\in W\}.$
 
Letting $R_\lambda=\{(\ab{x},\ab{y}):xRy\}$ (the least filtration of $R$  through $\Phi$), we define
 $R_\Phi\sub W_\Phi\times W_\Phi$ to be the \emph{transitive closure} of $R_\lambda$.
Thus $wR_\Phi v$ iff there exist $w_1,\dots, w_n\in W_\Phi$, for some $n>1$, such that 
$w=w_1R_\lambda\cdots R_\lambda w_n=v$. The definition of $\M_\Phi$ is completed by putting
$h_\Phi(p)=\{\ab{x}: p\in x\}$ for $p\in\Phi$, and $h_\Phi(p)=\emptyset$ (or anything) otherwise. We call $\M_\Phi$ the
\emph{standard transitive filtration through $\Phi$}.

The surjective function 
$f:W\to W_\Phi$ is given by $f(x)=\ab{x}$.
 The conditions (r1) and (r2) for a definable reduction are then immediate, and the definability condition (r5) is standard. For (r3) observe that $xRy$ implies $\ab{x}R_\lambda\ab{y}$ and hence $\ab{x}R_\Phi\ab{y}$.
 
 (r4) takes more work, but is also standard for the case of $\Dim$, and similar for $\Di$. To prove it,
 let  $\ab{x}R_\Phi\ab{y}$. Then by definition of $R_\Phi$ as the transitive closure of $R_\lambda$, there are finitely many elements $x_1,y_1,\dots,x_n,y_n$ of $W$ (for some $n\geq 1$)
such that
$$
x\sim x_1R y_1\sim x_2R y_2\sim \cdots \sim x_nR y_n\sim y.
$$
Then $\Di\G\in y\cap\Phi^t$ implies $\Di\G\in y_n$ as $y_n\sim y$, hence  $\Dim\Di\G\in x_n$ as $x_nR y$, which implies $\Di\G\in x_n$ by the scheme $4_t$. 
%\marginpar{\small \em  $4_t$ used} 
If $n=1$ we then get $\Di\G\in x$  because $x\sim x_1$.
But if $n>1$, we
 repeat this argument back along the above chain of relations, leading to $\Di\G\in x_{n-1}$, \dots ,$\Di\G\in x_1$, and then $\Di\G\in x$ as required to conclude that $y\cap \Phi^t\sub x\cap\Phi^t$. 

To show that $\{\Dim\ph\in\Phi:\Dim^*\ph\in y\}\sub x$, note that if $\Dim^*\ph\in y$, then either $\ph\in y$ or $\Dim\ph\in y$.
If $\ph\in y$, then  $\ph\in y_n$ as $y_n\sim y$ and $\ph\in\Phi$, hence $\Dim\ph\in x_n$ as $x_nRy_n$. But if $\Dim\ph\in y$ then $\Dim\ph\in y_n$, hence $\Dim\Dim\ph\in x_n$, and so again $\Dim\ph\in x_n$, this time by scheme 4.
Repeating this  back along the chain leads to $\Dim\ph\in x$ as required.
 
 Thus $(\M_\Phi,f)$ as defined is a definable reduction of $\F$.
 
 \medskip
 From this we can obtain a proof that the the smallest tangle system $\axK4t$ has the finite model property  over transitive frames. If
  $L_{\axK4t}$ is its set of theorems, put $\F=\F_{L_{\axK4t}}$. If $\ph$ is a
  $\axK4t$-consistent formula then $\ph\in x$ for some point $x$ of $\F$. Let $\Phi$ be the set of  subformulas of $\ph$, and  $\M_t$  the model derived from the model $\M_\Phi$ just defined. Then $\M_t,\ab{x}\models\ph$ by the Reduction Lemma. But the finite frame $\F_t=(W_\Phi,R_t)$ is transitive, so $\axK4t$ has the finite model property over  transitive frames, i.e.\ K4 frames.
 
 If we replace K4$t$ here by the smallest tangle system KD$4t$ containing $\Dim\top$, then the frame $\F_t$ of the last paragraph is serial, so $\{\psi:\F_t\models\psi\}$ is then a logic that contains $\Dim\top$, hence includes $L_{\axK\axD4t}$. Thus KD$4t$ has the finite model property over serial transitive  (i.e.\ KD4) frames.
 
 Similarly, since $\M_t$ is reflexive when $L$ contains the scheme T, we get that the smallest tangle system S4$t$ containing T has the finite model property over reflexive transitive (i.e.\ S4) frames.

 \subsection{Universal Modality}
 
 Extend the syntax to include the universal modality $\forall$ with semantics $\M,x\models\forall\ph$ iff for all $y$, 
 $\M,y\models\ph$.  Let K4$t$.U be the smallest tangle system that 
  includes the S5 axioms and rules for $\forall$, and the scheme
 \begin{description}
\item[U:]
 $\forall\ph\to\Box\ph$,
\end{description}
equivalently $\Dim\ph\to\exists\ph$, where $\exists=\neg\forall\neg$ is the dual modality to $\forall$.
 
 Let $L$ be any K4$t$.U-logic.
 Define a relation $S_L$ on $W_L$ by: $xS_Ly$ iff $\{\ph:\forall\ph\in x\}\sub y$ iff
 $\{\exists\ph:\ph\in y\}\sub x$. Then $S_L$ is an equivalence relation with $R_L\sub S_L$. Also
 $$
 \forall\ph\in x \text{ iff\enspace  for all } y\in W_L, \ xS_Ly \text{ implies }\ph\in y.
 $$
 For any fixed $x\in W_L$, let $W^x$ be the equivalence class $S_L(x)=\{y\in W_L:xS_Ly\}$. Then for $z\in W^x$,
 
\begin{equation}  \label{allsem}
 \forall\ph\in z \text{ iff\enspace  for all } y\in W^x, \ \ph\in y.
\end{equation} 
Let $R^x$ be the restriction of $R_L$ to $W^x$. Since $R_L\sub S_L$ it follows that $\F^x=(W^x,R^x)$ is an inner subframe of 
$(W_L,R_L)$.
If $\M_\Phi$ is a definable reduction of $\F^x$, and $\M_t$ its untangling, then using \eqref{allsem} it can be shown that if 
a formula $\ph\in\Phi$ satisfies the Reduction Lemma
 \begin{equation*}
\text{
$\M_t,\ab{z}\models \ph$\enspace iff\enspace $\ph\in z$}
\end{equation*}
for all $z$ in $\M_t$, then so does $\forall\ph$. So the Reduction Lemma holds for all members of $\Phi$.

Now the standard transitive filtration can be applied to $\F^x$ to produce a definable reduction of it.
Consequently,  if $\ph$ is an
  $L$-consistent formula,  $x$ is a point of $W_L$ with  $\ph\in x$, and $\Phi$ is the set of all subformulas of $\ph$, then 
  $\M_t,\ab{x}\models \ph$ where $\M_t$ is the untangling of the standard transitive filtration of $\F^x$ through $\Phi$.
That establishes the finite model property for K4$t$.U over transitive frames.

This construction preserves seriality and reflexiveness in passing from $R_L$ to $R^x$ and then $R_t$.
The outcome  is that the finite model property continues to hold for the tangle systems  
  KD4$t$.U and S4$t$.U over the KD4 and S4 frames, respectively.

\subsection{Path Connectedness}\label{sec:path conn}

A \emph{connecting path between $w$ and $v$} in a frame $(W,R)$ is a finite sequence $w=w_0,\dots, w_n=v$, for some $n\geq 0$, such that for all $i<n$, either $w_iRw_{i+1}$ or $w_{i+1}Rw_i$. We say that such  a path has \emph{length $n$}. 
The points $w$ and $v$ of $W$ are \emph{path connected}  if there exists a connecting path between them of some finite length. Note that any point $w$ is  connected to itself by a path of length 0 (put $n=0$ and $w=w_0$). 
The relation ``$w$ and $v$ are path connected'' is an equivalence relation whose equivalence classes are the \emph{path components} of the frame. The frame is \emph{path connected} if it has a single path component, i.e.\   any two points have a connecting path between them.
This is iff the frame
is connected in the sense of section~\ref{ss:kripke frames}.

Later we will make use of the fact that a path component $P$ is $R$-closed. For if $x\in P$  and $xR y$, then $x$ and $y$ are  path connected, so $y\in P$. It follows that any $R$-cluster $C$ that intersects $P$ must be included in $P$, for if $x\in P\cap C$ and $y\in C$, then $xR^*y$ and so $y\in P$, showing that $C\sub P$.

We now wish to show that in passing from  the frame $\F_\Phi=(W_\Phi,R_\Phi)$ to its untangling $\F_t$, there is no loss of path connectivity. The two frames have the same path connectedness relation and so have the same path components.
The idea is that the relations that
are broken by the untangling only occur between elements of the same $R_\Phi$-cluster, so it suffices to show that such elements are still path connected in $\F_t$. \emph{For this we need to make the assumption that $\Phi$ contains the formula $\Dim\top$. }This is harmless as we can always add it and its subformula $\top$, preserving finiteness of $\Phi$.

\begin{lemma} \label{repair}
Let $\Dim\top\in\Phi$. If $w,w'$ are points in $W_\Phi$ with $wR_\Phi w'$ or $w'R_\Phi w$, but neither $wR_tw'$ or $w'R_tw$, then there exist a $v$ with $wR_t v$ and $w'R_t v$.
\end{lemma}
\begin{proof}
If $wR_\Phi w'$, then since not  $wR_tw'$ we must have $w$ and $w'$ in the same cluster. The same follows if $w'R_\Phi w$, since not $w'R_tw$.

Thus there is an $R_\Phi$-cluster $C$ with $w,w'\in C$, so both $wR_\Phi w'$ and $w'R_\Phi w$. If $C$ is not $R_\Phi$-maximal, then there is an  $R_\Phi$-cluster $C'$ with $CR_\Phi C'$ and  $C\ne C'$. Taking any $v\in C'$ we then get $wR_tv$ and $w'R_tv$.

The alternative is that $C$ is  $R_\Phi$-maximal. Then we show that the nucleus $C^\circ$ is non-empty. 
Let $w=\ab{u}$ and $w'=\ab{t}$.
Since  $\ab{u}R_\Phi\ab{t}$ and $\top\in t$, and $\Dim\top\in\Phi$, property (r4) implies that 
$\Dim\top\in u$. Now if $y$ is the critical point for $C$, then $\Dim\top\in y$ by Lemma \ref{import}. Hence there is a $z$ with $yRz$. So $\ab{y}R_\Phi\ab{z}$ by (r3).  Maximality of $C$ then ensures that  $\ab{z}\in C$, so this implies that $\ab{z}\in C^\circ$.  Then by definition of $R_t$, since $w,w'\in C$ we have $wR_t\ab{z}$ and  $w'R_t\ab{z}$. 
\end{proof}

\begin{lemma} \label{pathcon}
If $\Dim\top\in\Phi$, then two members of $W_\Phi$ are path connected in $\F_\Phi$ if, and only if, they are path connected in $\F_t$. Hence the two frames have the same path components.
\end{lemma}
\begin{proof}
Since $R_t\sub R_\Phi$, a connecting path in $\F_t$ is a connecting path in $\F_\Phi$, so points that are path connected in $\F_t$ are path connected in $\F_\Phi$. 

Conversely, let $\pi=w_0,\dots,w_n$ be a connecting path in $\F_\Phi$. If, for all $i<n$, either $w_iR_tw_{i+1}$ or $w_{i+1}R_tw_{i}$, then $\pi$ is a connecting path in  $\F_t$. If not, then for each $i$ for which this fails, by Lemma \ref{repair} there exists some $v_i$ with $w_iR_t v_i$ and $w_{i+1}R_t v_i$. Insert $v_i$ between $w_i$ and $w_{i+1}$ in the path. Doing this for
all ``defective'' $i<n$, creates a new sequence that is now a connecting path in $\F_t$ between the same endpoints.
\end{proof}

Now let K4$t$.UC be the smallest extension of system  K4$t$.U in the language with $\forall$ that includes the scheme

\begin{description}
\item[C:]
$\forall(\Box^*\ph\lor\Box^*\neg\ph)\to(\forall\ph\lor\forall\neg\ph)$,
\end{description}
or equivalently
$\exists\ph\land\exists\neg\ph\to\exists(\Dim^*\ph\land\Dim^*\neg\ph)$.

Let $L$ be any K4$t$.UC-logic.
Let $\F^x$ be a point-generated subframe of $(W_L,R_L)$ as above, and $\M_\Phi$ its standard transitive filtration through 
$\Phi$. Then the frame $\F_\Phi=(W_\Phi,R_\Phi)$ of $\M_\Phi$ is path connected, as shown by 
Shehtman \cite{Sheh:everywhere99} as follows. If $P$ is the path component of $\ab{x}$ in $\M_\Phi$, take a formula $\ph$ that defines $f^{-1}(P)$ in $W^x$, i.e.\ $\ph\in y$ iff $\ab{y}\in P$, for all 
$y\in W^x$. Suppose, for the sake of contradiction, that $P\ne W_\Phi$. Then there is some $z\in W^x$ with $\ab{z}\notin P$, hence $\neg\ph\in z$. Since $\ph\in x$, this gives  $\exists\ph\land\exists\neg\ph\in x$. By the scheme C it follows that for some $y\in W^x$, 
$\Dim^*\ph\land\Dim^*\neg\ph\in y$. Hence there are $z,w\in W^x$ with $yR^*z$, $\ph\in z$, $yR^* w$ and $\neg\ph\in w$.

From this we get $\ab{y}R_\Phi{}^*\ab{z}$ and $\ab{y}R_\Phi{}^*\ab{w}$ so the sequence $\ab{z},\ab{y},\ab{w}$ is a connecting path between  $\ab{z}$ and  $\ab{w}$ in $\F_\Phi$. But $\ab{z}\in P$ as $\ph\in z$, so this implies $\ab{w}\in P$. Hence $\ph\in w$, contradicting the fact that $\neg\ph\in w$. The contradiction forces us to conclude that $P= W_\Phi$, and hence that $\F_\Phi$ is path connected.

From Lemma \ref{pathcon} it now follows that the untangling $\F_t$ of $\F_\Phi$ is also path connected when $L$  includes scheme C and $\Dim\top\in\Phi$. Hence the finite model property holds for K4$t$.UC over path-connected transitive frames.

The arguments for the preservation of seriality and reflexiveness by $\F_t$ continue to hold here. This gives us proofs of the
finite model property  for the systems, KD4$t$.UC and S4$t$.UC over path-connected KD4 and S4 frames, respectively.

Note that for the $\c L_{\bo\forall}$-fragments of these logics (i.e.\ their restrictions to the language without $\dit$), our analysis reconstructs the finite model property proof of \cite{Sheh:everywhere99} by using $\M_\Phi$ instead of $\M_t$. For, restricting to this language, if $\M_\Phi$ is a standard transitive filtration of an inner subframe of $\F_L$, then any $\dit$-free formula is true in $\M_\Phi$ precisely at the points at which it is realised (for $\c L_{\bo}$ this is a classical result first formulated and proved in \cite{sege:deci68}). Thus a finite satisfying model for a consistent 
$\c L_{\bo\forall}$-formula can be obtained as a model of this form $\M_\Phi$. Since seriality and reflexivity are preserved in passing from $R_L$ to $R_\Phi$, and $\F_\Phi$ is path connected in the presence of  axiom C, it follows that the finite model property holds for each of the systems K4.UC, KD4.UC and S4.UC in the language 
$\c L_{\bo\forall}$.

\subsection{The Schemes G$_n$}

Fix $n\geq 1$ and take $n+1$ variables $p_0,\dots,p_{n}$. For each $i\leq n$, define the formula
\begin{equation} \label{Qi}
Q_i=p_i \land  \bigwedge_{i\ne j\leq n}\neg p_j.
\end{equation}
G$_n$ is the scheme consisting of all uniform substitution instances of the formula
\begin{equation} \label{Gn}
\bigwedge_{i\leq n}\Dim Q_i\to\Dim(\bigwedge_{i\leq n}\Dim^* \neg Q_i ).
\end{equation}
This is equivalent in any  logic to
\begin{equation*}
\Box(\bigvee_{i\leq n}\Box^*  Q_i )\to \bigvee_{i\leq n}\Box\neg Q_i,
\end{equation*}
the form in which the G$_n$'s were introduced in \cite{Sheh:d90}. When $n=1$, \eqref{Gn} is
\begin{equation}  \label{G1}
\Dim(p_0\land\neg p_1)\land \Dim(p_1\land\neg p_0) \to \Dim(\Dim^*\neg(p_0\land\neg p_1)\land \Dim^*\neg(p_1\land\neg p_0)).
\end{equation}
As an axiom, \eqref{G1} is equivalent  to
\begin{equation}  \label{G1B}
\Dim p\land \Dim\neg p \to \Dim(\Dim^* p\land \Dim^*\neg p),
\end{equation}
or in dual form $\Box(\Box^* p\lor \Box^* \neg p)\to \Box p\lor \Box \neg p$, which is the form in which G$_1$ was first defined in \cite{Sheh:d90}. To derive \eqref{G1B} from \eqref{G1}, substitute $p$ for $p_0$ and $\neg p$ for $p_1$ in \eqref{G1}. Conversely, substituting $p_0\land\neg p_1$ for $p$ 
in  \eqref{G1B} leads to a derivation of  \eqref{G1}.

For the semantics of G$_n$, we use the set $R(x)=\{y\in W:xRy\}$ of $R$-successors of $x$ in a frame $(W,R)$. We can view $R(x)$ as a frame in its own right, under the restriction of $R$ to $R(x)$, and consider whether it is path connected, or how many path components it has etc.  $(W,R)$ is called \emph{locally $n$-connected}  if, for all $x\in W$, the frame 
$\F(x)=(R(x),R{\restriction} R(x))$ has at most $n$ path components.
This is equivalent to the definition in section~\ref{ss:kripke frames}.
Note that path components in $\F(x)$ are defined by connecting paths in $(W,R)$ that lie entirely within $R(x)$. 

\begin{fact}\label{fact:Gn = loc nconn}
 A K4 frame validates G$_n$ iff it is locally $n$-connected. 
\end{fact}
For a proof of this see \cite[Theorem 3.7]{LucBry11}.

\subsection{Weak Models}\label{ss:weak models}

We now assume that the set $\Var$ of variables is \emph{finite}. The adjective ``weak'' is sometimes applied to languages with finitely many variables, as well as to models for weak languages and to canonical frames built from them. Weak models may enjoy special properties. For instance,  a proof is given in \cite[Lemma 8]{Sheh:d90} that in a weak
\emph{distinguished}\footnote{A model is distinguished if for any two of its distinct points there is a formula that is true
in the model at one of the points and not the other.} model on a transitive frame, there are only finitely many maximal clusters. 
This was used to show that a weak canonical model for the $\c L_\bo$-system  K4DG$_1$ is locally 1-connected, and from this to obtain the finite model property for that system. The corresponding versions of these results for  K4DG$_n$ with $n\geq 2$ are worked out in \cite{LucBry11}.

We wish to lift these results to the language $\c L_\bo^\dit$ with tangle.  One issue is that the property of  a canonical model being distinguished depends on it satisfying the Truth Lemma: $\M_L,x\models\ph$ iff $\ph\in x$. As we have seen, this fails for tangle logics. Therefore we must continue to work directly with the relation of membership of formulas in points of $W_L$, rather than with their truth in $\M_L$. We will see that it is still possible to recover Shehtman's analysis of maximal clusters in $\F_L$, with the aid of both tangle axioms.

Another issue is that we want to work over K4G$_n$ without assuming the seriality axiom. This requires further adjustments, and care with the distinction between $R$ and $R^*$.

Let $L$ be any tangle logic in our weak language. Put
$\At=\Var\cup\{\Dim\top\}$. For each $s\sub\At$ define the formula
$$
\chi(s)=\bigwedge_{\ph\in s}\ph\land\bigwedge_{\ph\in \At\setminus s}\neg\ph.
$$
For each point $x$ of $W_L$ define $\tau(x)=x\cap\At$. Think of $\At$ as a set of ``atoms'' and $\tau(x)$ as the ``atomic type'' of 
$x$. It is evident that for any $x\in W_L$ and $s\subseteq\At$
we have
\begin{equation} \label{chit}
\chi(s)\in x \text{ iff } s=\tau(x).
\end{equation}
Writing $\chi(x)$ for the formula $\chi(\tau(x))$, we see from \eqref{chit} that $\chi(x)\in x$, and in general $\chi(y)\in x$ iff $\tau(y)=\tau(x)$.

Now fix an inner subframe $\F=(W,R)$ of $\F_L$.  If $C$ is an $R$-cluster in $\F$,
let 
$$
\delta C=\{\tau(x):x\in C\}
$$
be the set of atomic types of members of $C$. We are going to show that maximal clusters in $\F$ are determined by their atomic types. They key to this is:  

\begin{lemma} \label{indisting}
Let $C$ and $C'$ be maximal clusters in $\F$ with $\delta C=\delta C'$. Then for all formulas $\ph$, if $x\in C$ and $x'\in C'$ have $\tau(x)=\tau(x')$, then $\ph\in x$ iff $\ph\in x'$. Thus, $x=x'$.
\end{lemma}
\begin{proof}
Suppose $C$ and $C'$ are maximal with $\delta C=\delta C'$. The key property of maximality that is used is that if $x\in C$ and $xRy$, then $y\in C$, and likewise for $C'$.

The proof  proceeds by induction on the formation of $\ph$. The base case,
when $\ph\in\Var$, is immediate from the fact that  then $\ph\in x$ iff $\ph\in \tau(x)$. The induction cases for the Boolean connectives are straightforward from properties of maximally consistent sets.

Now take the case of a formula $\Dim\ph$ under the induction hypothesis that the result holds for $\ph$, i.e.\ $\ph\in x$ iff $\ph\in x'$ for any $x\in C$ and $x'\in C'$ such that $\tau(x)=\tau(x')$. Take such $x$ and $x'$, and assume $\Dim\ph\in x$. Then $\ph\in y$ for some $y$ such that $xRy$. Then $y\in C$ as $C$ is maximal. Hence $\tau(y)\in \delta C=\delta C'$, so 
$\tau(y)=\tau(y')$ for some $y'\in C'$. Therefore $\ph\in y'$ by the induction hypothesis on $\ph$. But $\Dim\top\in x$ (as $xRy$), so $\Dim\top\in \tau(x)=\tau(x')$. This gives $\Dim\top\in x'$ which ensures that $x'Rz$ for some $z$, with $z\in C'$ as $C'$ is maximal, hence $C'$ is a non-degenerate cluster.\footnote{That is the reason for including $\Dim\top$ in $\At$.}  It follows that  $x'Ry'$, so $\Dim\ph\in x'$ as required. Likewise $\Dim\ph\in x'$ implies $\Dim\ph\in x$, and  the Lemma holds for $\Dim\ph$.

Finally we have the case of a formula $\Di\G$ under the induction hypothesis that the result holds for every $\g\in\G$.  
Suppose  $x\in C$ and $\tau(x)=\tau(x')$ for some $x'\in C'$. Let $\Di\G\in x$. Then by axiom Fix, for each $\g\in\G$ we have 
$ \Dim(\g\land\Di\G)\in x$, implying that $\Dim\g\in x$. Then applying to $\Dim\g$ the analysis of $\Dim\ph$ in the previous paragraph, we conclude that $C'$ is non-degenerate and there is  some $y_\g\in C'$ with  $\g\in y_\g$. Now if $x'R^*z$, then $z\in C'$ so for each $\g\in\G$ we have $zR y_\g$, implying that $\Dim\g\in z$.
This proves that 
$\Box^*(\bigwedge_{\g\in\G}\Dim\g)\in x'$.
But putting $\ph=\top$ in axiom Ind shows that the formula
$$
\Box^*(\top\to \bigwedge_{\g\in\G}\Dim(\g\land\top))\to(\top\to\Di\G)
$$
is an $L$-theorem, From this we can derive that  $\Box^*(\bigwedge_{\g\in\G}\Dim\g)\to\Di\G$ is an $L$-theorem, and hence belongs to $x'$. Therefore $\Di\G\in x'$ as required. Likewise $\Di\G\in x'$ implies $\Di\G\in x$, and so the Lemma holds for $\Di\G$.
\end{proof}

\begin{corollary} \label{equalC}
If $C$ and $C'$ are maximal clusters in $\F$ with $\delta C=\delta C'$, then $C=C'$.
\end{corollary}
\begin{proof}
If $x\in C$, then $\tau(x)\in\delta C=\delta C'$, so there exists $x'\in C'$ with $\tau(x)=\tau(x')$. Lemma \ref{indisting} then implies that 
$x=x'\in C'$, showing $C\sub C'$.  Likewise $C'\sub C$.
\end{proof}

\begin{corollary} \label{finmaxcl}
The set $M$ of all maximal clusters of $\F$ is finite.
\end{corollary}
\begin{proof}
The map $C\mapsto\delta C$ is an injection of $M$ into the double power set $\wp\wp\At$ of the finite set $\At$. This gives an upper bound of $2^{2^{n+1}}$ on the number of maximal clusters, where $n$ is the size of $\Var$.
\end{proof}

Given subsets $X,Y$ of $W$ with $X\sub Y$, we say that $X$ is \emph{definable within $Y$ in $\F$} if there is a formula $\ph$ such that for all $y\in Y$, $y\in X$ iff $\ph\in y$.
We now work towards showing that within each inner subframe $R(x)$ in $\F$, each path component is definable. For each cluster $C$, define the formula
$$
\a(C)=  \bigwedge_{s\in \delta C}\Dim^*\chi(s)\land\bigwedge_{s\in \wp\At\setminus \delta C}\neg\Dim^*\chi(s).
$$
The next result shows that a maximal cluster is definable within the set of all maximal elements of $\F$.

\begin{lemma} \label{maxdef}
If  $C$ is a maximal cluster and $x$ is any maximal element of $\F$, then $x\in C$ iff $\a(C)\in x$.
\end{lemma}
\begin{proof}
Let $x\in C$. If $s\in\delta C$, then $s=\tau(y)$ for some $y$ such that $y\in C$, hence $xR^*y$, and $\chi(s)=\chi(y)\in y$, showing that $\Dim^*\chi(s)\in x$. The converse of this also holds: if $\Dim^*\chi(s)\in x$, then for some $y$, $xR^* y$ and $\chi(s)\in y$. Hence $y\in C$ by maximality of $C$, and $s=\tau(y)$ by \eqref{chit}, so $s\in\delta C$.
Contrapositively then, if $s\notin \delta C$, then $\Dim^*\chi(s)\notin x$, so $\neg\Dim^*\chi(s)\in x$. 
Altogether this shows that all conjuncts of $\a(C)$ are in $x$, so $\a(C)\in x$.

In the opposite direction, suppose $\a(C)\in x$. Let $C'$ be the cluster of $x$. Then we want $C=C'$ to conclude that $x\in C$. Since $x$ is maximal, i.e.\ $C'$ is maximal, it is enough by Corollary \ref{equalC} to show that $\delta C=\delta C'$.

Now if $s\in\delta C$, then $s=\tau(y)$ for some $y\in C$. But $\Dim^*\chi(s)$ is a conjunct of $\a(C)\in x$, so $\Dim^*\chi(s)\in x$.
Hence there exists $z$ with $xR^*z$ and $\chi(s)\in z$. Then $z\in C'$ by maximality of $C'$, and  by \eqref{chit} $s=\chi(z)\in \delta C'$.

Conversely, if $s\in\delta C'$, with  $s=\tau(y)$ for some $y\in C'$, then $xR^* y$ as $x\in C'$, and so 
$\Dim^*\chi(s)\in x$ as $\chi(s)=\chi(y)\in y$. Hence $\neg\Dim^*\chi(s)\notin x$.
But then we must have $s\in\delta C$, for otherwise $\neg\Dim^*\chi(s)$ would be a conjunction of $\a(C)$ and so would belong to $x$.
\end{proof}

It is shown in \cite{Sheh:d90} that any transitive canonical frame (weak or not) has the \emph{Zorn property}:
\begin{center}
$\forall x\,\exists y (xR^*y$ and  $y$ is $R$-maximal).
\end{center}
Note the use of $R^*$: the statement is that either $x$ is $R$-maximal, or it has an $R$-maximal successor. 
The essence of the proof is that the relation
$\{(x,y): xR^\bullet y \text{ or }x= y\}$
is a partial ordering for which every chain has an upper bound, so by Zorn's Lemma $R(x)$ has a maximal element provided that it is non-empty.

The Zorn property is preserved under inner substructures, so it holds for our frame $\F$. One interesting consequence is:

\begin{lemma} \label{finpthcmp}
For each $x\in W$, the frame $\F(x)=(R(x),R{\restriction} R(x))$ has  finitely many path components,
as does $\F$ itself.
\end{lemma}
\begin{proof}
 The following argument works for both $\F$ and $\F(x)$, noting that the $R{\restriction} R(x)$-cluster of an element of $\F(x)$ is the same as its $R$-cluster in $\F$, and that all maximal clusters of $\F(x)$ are maximal in $\F$.

Let $P$ be a path component and $y\in P$. By the Zorn property there is an $R$-maximal $z$ with $yR^* z$. Then  $z\in P$ as $P$ is $R^*$-closed. So  the $R$-cluster of $z$ is a subset of $P$. Since this cluster is maximal, that proves that every path component contains a maximal cluster.

Now distinct path components are disjoint and so cannot contain the same maximal cluster. Since there are finitely many maximal clusters (Corollary \ref{finmaxcl}), there can only be finitely many path components.
\end{proof}

\begin{lemma}  \label{defCRx}
Let $C$ be a maximal cluster in $\F$. Then for all $x\in W$:
\begin{enumerate}[\rm(1)]
\item 
$C\sub R(x)$ iff\/ $\Dim\Box^*\a(C)\in x$.
\item
$C\sub R^*(x)$ iff\/ $\Dim^*\Box^*\a(C)\in x$.
\end{enumerate}
\end{lemma}

\begin{proof}
For (1), first let $C\sub R(x)$. Take any $y\in C$. Then if $yR^*z$ we have $z\in C$ as $C$ is maximal, therefore $\a(C)\in z$ by Lemma \ref{maxdef}. Thus $\Box^*\a(C)\in y$. But $y\in R(x)$, so then $\Dim\Box^*\a(C)\in x$.

Conversely, if $\Dim\Box^*\a(C)\in x$ then for some $y$, $xRy$ and $\Box^*\a(C)\in y$. By the Zorn property, take a maximal $z$ with $yR^*z$. Then $\a(C)\in z$, so $z\in C$ by Lemma \ref{maxdef}. From $xRyR^*z$ we get $xRz$, so $z\in R(x)\cap C$.
Since $R(x)$ is $R^*$-closed, this is enough to force $C\sub R(x)$.

The proof of (2) is similar to (1), replacing $R$ by $R^*$ where required.
\end{proof}

For a given $x\in W$, let $P$ be a path component of the frame $\c F(x)=(R(x),R{\restriction} R(x))$. Let $M(P)$ be the set of all  maximal $R$-clusters $C$ that have  $C\sub P$. Then $M(P)\sub M$, where $M$ is the set of all maximal clusters of $\F$, so $M(P)$ is finite by Corollary \ref{finmaxcl}. Define the formula
$$
\a(P)=\bigvee\{\Dim^*\Box^*\a(C):C\in M(P)\}.
$$
Then $\a(P)$ defines $P$ within $R(x)$:

\begin{lemma} \label{defP}
For all $y\in R(x)$,  $y\in P$ iff\/ $\a(P)\in y$.
\end{lemma}
\begin{proof}
Let $y\in R(x)$.  If $y\in P$, take an $R$-maximal $z$ with $yR^*z$, by the Zorn property. Then $z\in R(x)$, and $z$ is path connected to $y\in P$, so $z\in P$. The  cluster $C_z$ of $z$ is then included in $P$ (if $w\in C_z$ then $zR^*w$ so $w\in P$), and $C_z$ is maximal, so $C_z\in M(P)$. The maximality of $C_z$ together with Lemma \ref{maxdef} then ensure that $\Box^*\a(C_z)\in z$. Hence  $\Dim^*\Box^*\a(C_z)\in y$. But $\Dim^*\Box^*\a(C_z)$ is a disjunct of $\a(P)$, so $\a(P)\in y$.

Conversely, if $\a(P)\in y$, then  $\Dim^*\Box^*\a(C)\in y$ for some $C\in M(P)$. By Lemma \ref{defCRx}(2), $C\sub R^*(y)$. Taking any $z\in C$, since also $C\sub P$ we have $yR^*z\in P$, hence $y\in P$.
\end{proof}

\begin{theorem}\label{canlnc}
Suppose that $L$ includes the scheme $\axG_n$.
Then every inner subframe $\F$ of $\F_L$ is locally $n$-connected.
\end{theorem}
\begin{proof} 
Let $x\in W$. We have to show that $R(x)$ has at most $n$ path components. If it has fewer than $n$ there is nothing to do, so suppose  $R(x)$ has at least  $n$ path components $P_0,\dots,P_{n-1}$. Put 
$P_n=  R(x)\setminus (P_0\cup\cdots\cup P_{n-1})$. We will prove that $P_n=\emptyset$, confirming that there can be no more components.

For each $i< n$, let $\ph_i$ be the formula $\a(P_i)$ that defines $P_i$ within $R(x)$ according to Lemma \ref{defP}.  Let 
 $\ph_n$ be $\neg\bigvee\{\a(P_i):0\leq i< n\}$, so $\ph_n$ defines $P_n$ within $R(x)$. Now for all $i\leq n$ let $\psi_i$ be the formula obtained by uniform substitution of 
$\ph_0,\dots,\ph_n$ for $p_0, \dots,p_n$ in the formula $Q_i$ of \eqref{Qi}. Observe that since the $n+1$ sets 
$P_0,\dots,P_n$ form a partition of $R(x)$, each $y\in R(x)$ contains $\psi_i$ for exactly one $i\leq n$, and indeed $\psi_i$ defines the same subset of $R(x)$ as  $\ph_i$.

Now suppose, for the sake of contradiction, that $P_n\ne\emptyset$.\footnote{In that case $P_n$ is the union of finitely many path components, by Lemma \ref{finpthcmp}, but we do not need that fact.}
Then for each $i\leq n$ we can choose an element $y_i\in P_i$. Then $xRy_i$ and $\psi_i\in y_i$.
It follows that
$\bigwedge_{i\leq n}\Dim \psi_i\in x$. Since all instances of G$_n$ are in $x$, we then get 
$\Dim(\bigwedge_{i\leq n}\Dim^* \neg \psi_i )\in x$. So there is some $y\in R(x)$ such that for each $i\leq n$ there exists a $z_i\in R^*(y)$ such that   $\neg\psi_i\in z_i$, hence $\psi_i\notin z_i$. Now let $P$ be the path component of $y$. If $P=P_i$ for some    $i<n$, then as $y\in P_i$ and  $yR^*z_i$, we get $z_i\in P_i$,  and so $\psi_i\in z_i$ -- which is false. Hence it must be 
that $P$ is disjoint from $P_i$ for all $i<n$, and so is a subset of $P_n$.
But then as $yR^*z_n$ we get $z_n\in P\sub P_n$, and so $\psi_n\in z_n$.
That is also false, and shows that the assumption that $P_n\ne\emptyset$ is false.
\end{proof}

\subsection{Completeness and finite model property for K4G$_n$}\label{sec:fmp Gn}

For the language $\c L_\bo$ without $\Di$, Theorem \ref{canlnc} provides a completeness  theorem for any system extending  K4G$_n$ by showing that any consistent formula $\ph$ is satisfiable in a locally $n$-connected weak canonical model (take a finite $\Var$ that includes all variables of $\ph$ and enough variables to have G$_n$ as a formula in the weak language). But  the ``satisfiable'' part of this depends on the Truth Lemma, which is unavailable in the presence of $\Di$. We will need to apply filtration/reduction to establish completeness itself,  as well as  the finite model property.

Let $L$ be a weak tangle logic that includes G$_n$;  $\F=(W,R)$ an inner subframe of $\F_L$; and $\Phi$ a finite set of formulas that is closed under subformulas.

Recall that $M$ is the  set of all maximal clusters of $\F$, shown to be finite in Corollary \ref{finmaxcl}. 
For each $x\in W$, define  $$M(x)=\{C\in M: C\sub R(x)\}.$$ Then $M(x)$ is finite, being a subset of $M$.
 
Define an equivalence relation $\app$ on $W$ by putting 
\begin{center}
$x\app y$ iff $x\cap\Phi=y\cap\Phi$ and $M(x)=M(y)$. 
\end{center}
We then repeat the earlier standard transitive filtration construction, but using the finer relation $\app$ in place of $\sim$.
Thus we put  $\ab{x}=\{y\in W:x\app y\}$ and  $W_\Phi=\{\ab{x}:x\in W\}$.  The set $W_\Phi$ is finite, because the map $\ab{x}\mapsto (x\cap\Phi,M(x))$ is a well-defined injection of $W_\Phi$ into the finite set $\wp\Phi\times\wp M$.
The surjective function  $f:W\to W_\Phi$ is given by $f(x)=\ab{x}$.

Let $\M_\Phi=(W_\Phi,R_\Phi,h_\Phi)$, where $R_\Phi\sub W_\Phi\times W_\Phi$ is the transitive closure of 
$R_\lambda=\{(\ab{x},\ab{y}):xRy\}$,  $h_\Phi(p)=\{\ab{x}: p\in x\}$ for $p\in\Phi$, and $h_\Phi(p)=\emptyset$  otherwise.

We now verify that the pair $(\M_\Phi,f)$ as just defined satisfies the axioms (r1)--(r5) of a definable reduction of $\F$ via $\Phi$.

\begin{enumerate}[(r1):]
\item 
\emph{$p\in x$ iff $\ab{x}\in h_\Phi(p)$, for all $p\in\Var\cap\Phi$.}

 By definition of $h_\Phi$.

\item
\emph{$\ab{x}=\ab{y}$ implies $x\cap\Phi= y\cap\Phi$.} 

If $\ab{x}=\ab{y}$ then $x\app y$, so $x\cap\Phi= y\cap\Phi$ by definition of $\app$.

\item
\emph{$xRy$ implies $\ab{x}R_\Phi \ab{y}$.}

$xRy$ implies $\ab{x}R_\lambda \ab{y}$ and $R_\lambda\sub R_\Phi$.
\item
\emph{$\ab{x}R_\Phi \ab{y}$ implies  $y\cap \Phi^t\sub x\cap\Phi^t$ and 
$\{\Dim\ph\in\Phi:\Dim^*\ph\in y\}\sub x$.}

The proof is the same as the proof given earlier of (r4) for the standard transitive filtration, but using $\app$ in place of $\sim$ and the fact that $x\app y$ implies $x\cap\Phi= y\cap\Phi$.

\item 
\emph{For each subset $C$ of $W_\Phi$ there is a formula $\ph$ that defines $f^{-1}(C)$ in $W$, i.e.\ $\ph\in y$ iff
 $\ab{y}\in C$.}
 
 To see this, for each $x\in W$ let $\g_x$ be the conjunction of $(x\cap\Phi)\cup\{\neg\psi:\psi\in\Phi\setminus x\}$. Then for any $y\in W$,
 $$
 \g_x\in y \quad\text{iff}\quad x\cap\Phi =  y\cap\Phi.
 $$
 Next, let $\mu_x$ be the conjunction of the finite set of formulas
 $$
 \{\Dim\Box^*\a(C): C\in M(x)\}\cup \{\neg\Dim\Box^*\a(C): C\in M\setminus M(x)\}.
 $$
 Lemma \ref{defCRx} showed that each $C\in M$ has $C\in M(x)$ iff $\Dim\Box^*\a(C)\in x$. From this it follows readily that
for any $y\in W$,
 $$
 \mu_x\in y \quad\text{iff}\quad M(x)=M(y).
 $$
 So putting $\ph_x=\g_x\land\mu_x$, we get that in general
 $$
 \ph_x\in y \quad\text{iff}\quad x\app y \quad\text{iff}\quad \ab{y}\in\{\ab{x}\}.
 $$
 Now if $C=\emptyset$, then $\bot$ defines $f^{-1}(C)$ in $W$. Otherwise if $C=\{\ab{x_1},\dots,\ab{x_n}\}$, then the disjunction
 $
 \ph_{x_1}\lor\cdots\lor\ph_{x_n}
 $
 defines $f^{-1}(C)$ in $W$.% \proofqed
 \end{enumerate}
Consequently, the reduction $\M_t$ of $\M_\Phi$ satisfies the Reduction Lemma. We will show that G$_n$ is valid in the frame of $\M_t$.  But first we show that it is valid in the frame of $\M_\Phi$. Both cases involve some preliminary analysis,  involving linking points of $R_\Phi(\ab{y})$  and $R_t(\ab{y})$ back to points of $R(y)$. This requires further work with maximal elements and clusters.

\begin{lemma} \label{presM}
For all $x,y\in W$, $\ab{x}R_\Phi^*\ab{y}$ implies $M(y)\sub M(x)$.
\end{lemma}
\begin{proof}
If $\ab{x}R_\Phi^*\ab{y}$ there is a finite sequence $x=z_0,\dots,z_k=y$ for some $k\geq 1$ such that for all $i<k$, either $z_i\app z_{i+1}$ or  $z_iR z_{i+1}$.
But $z_i\app z_{i+1}$ implies $M(z_i)=M(z_{i+1})$, and  $z_iR z_{i+1}$ implies $M(z_{i+1})\sub M(z_i)$ by transitivity of $R$. This yields $M(z_{k})\sub M(z_0)$ by induction on $k$.
\end{proof}

\begin{lemma} \label{max}
Suppose $\At\sub\Phi$ and
$a\in W$ is $R$-maximal. Then for all $x\in W$,  $xRa$  iff\/ $\ab{x}R_\Phi\ab{a}$.
\end{lemma}
\begin{proof}
$xRa$ implies $\ab{x}R_\Phi\ab{a}$ by (r3). For the converse, suppose $\ab{x}R_\Phi\ab{a}$
and let  $K$ be the maximal $R$-cluster of $a$. 

If $K$ is non-degenerate then $K\sub R(a)$, so $K\in M(a)$. Then from $\ab{x}R_\Phi\ab{a}$ we get $K\in M(x)$ by  Lemma \ref{presM}, implying $xRa$ as required. 

But if $K$ is degenerate, then $K=\{a\}$ and $R(a)=M(a)=\emptyset$. Also $\Dim\top\notin a$. Since $\ab{x}R_\Phi\ab{a}$, by definition of $R_\Phi$ there are $z,w\in W$ with $\ab{x}R_\Phi^*\ab{z}$ and $zRw\app a$. As $\At\sub\Phi$, from $w\app a$ we get $w\cap\At=a\cap\At$, i.e.\ $\tau(w)=\tau(a)$. In particular $\Dim\top\notin w$, hence $w$ is also $R$-maximal. Therefore $a$ and $w$ are maximal elements with the same atomic type, so $w=a$ by Lemma \ref{indisting}. Thus $zRa$ and so $K\in M(z)$.
Since $\ab{x}R_\Phi^*\ab{z}$ this implies $K\in M(x)$ by Lemma \ref{presM}, giving the required $xRa$ again.
\end{proof}

\begin{lemma} \label{R*max}
For any $y\in W$, let $A$ be the set of all $R$-maximal points in $R(y)$.
Then each  point $v\in R_\Phi(\ab{y})$ has $vR_\Phi^*\ab{a}$ for some $a\in A$. 
\end{lemma}
\begin{proof}
Let $v=\ab{z}\in R_\Phi(\ab{y})$. By the Zorn property there exists an $a$ with $zR^*a$ and $a$ is $R$-maximal.
If $z=a$, then $z$ is $R$-maximal, so as $\ab{y}R_\Phi\ab{z}$ we have $z\in R(y)$ by Lemma \ref{max}. Hence $z\in A$, so in this case we get $\ab{z}R_\Phi^*\ab{a}$ with $a\in A$ by taking $a=z$.

If however $z\ne a$, then $zRa$, hence $\ab{z}R_\Phi\ab{a}$ by (r3).  Also, if $C$ is the $R$-cluster of $a$, then $C\sub R(z)$ and $C$ is maximal, hence $C\in M(z)$. But $\ab{y}R_\Phi\ab{z}$, so Lemma \ref{presM} then implies $C\in M(y)$, therefore  $a\in R(y)$. So in this case  we have $\ab{z}R_\Phi\ab{a}$ with $a\in A$. 
\end{proof}

\begin{theorem} \label{FPhilnc}
If\/ $\At\sub\Phi$, the frame $\F_\Phi=(W_\Phi,R_\Phi)$ is locally $n$-connected.
\end{theorem}
\begin{proof}
For any point $\ab{y}\in W_\Phi$, we have to show that $R_\Phi(\ab{y})$ has at most $n$ path components. But if it had more than $n$, then by picking  points from different components we would get a sequence of more than $n$ points no two of which were path connected. We show that this is impossible, by taking an arbitrary sequence $v_0,\dots,v_n$ of $n+1$ points in $R_\Phi(\ab{y})$, and proving that there must exist distinct $i$ and $j$ such that $v_i$ and $v_j$ are path connected  in $R_\Phi(\ab{y})$.

For each $i\leq n$, by Lemma \ref{R*max} there is an $R$-maximal $a_i\in R(y)$ with $v_iR_\Phi^*\ab{a_i}$. This gives us a sequence $a_0,\dots,a_n$   of  members of $R(y)$. But $R(y)$ has at most $n$ path components, by Theorem \ref{canlnc}. Hence there exist $i\ne j\leq n$ such that there is a connecting  $R$-path 
$a_i=w_0,\dots,w_n=a_{j}$ between $a_i$ and $a_{j}$ that lies in $R(y)$. So for all $i<n$ we have $yRw_i$ and either $w_iRw_{i+1}$ or $w_{i+1}Rw_i$, hence $\ab{y}R_\Phi \ab{w_i}$ and either $\ab{w_i}R_\Phi\ab{w_{i+1}}$ or $\ab{w_{i+1}}R_\Phi\ab{w_i}$.

This shows that $\ab{a_i}$ and $\ab{a_j} $ are path connected in $R_\Phi(\ab{y})$ by the sequence
$\ab{w_0},\dots,\ab{w_n}$. Since $v_iR_\Phi^*\ab{a_i}$ and $v_jR_\Phi^*\ab{a_j}$, it follows that 
$v_i$ and $v_j $ are path connected in $R_\Phi(\ab{y})$, as required.
\end{proof}

From this result we can infer that in the language $\c L_\bo$, for all $n\geq 1$ the finite model property holds  for K4G$_n$ and  KD4G$_n$  over locally $n$-connected K4 and KD4  frames, respectively.
For the proof, 
we take a consistent $\c L_\bo$-formula $\ph$ and let $\Phi$  be the closure under $\c L_\bo$-subformulas of $\At\cup\{\ph\}$. Then $\Phi$ is finite and $\ph$ is satisfiable in the model $\M_\Phi$  (see the remarks about $\M_\Phi$ at the end of section \ref{sec:path conn}). But the frame $\F_\Phi$ of $\M_\Phi$ is locally $n$-connected by the theorem just proved, so validates G$_n$.
Together with the preservation of seriality by $\F_\Phi$, this implies the finite model property results for K4G$_n$ and  KD4G$_n$.

Extending to the language $\c L_{\bo\forall}$, and using that $\F_\Phi$ is path connected in the presence of  axiom C, these 
  finite model property results hold  correspondingly for the four systems K4G$_n$.U,  K4G$_n$.UC,  KD4G$_n$.U,  and KD4G$_n$.UC.

We turn now to the corresponding results for the versions of these systems that include the tangle connective.
\begin{lemma}  \label{critpresRy}
If $y\in W$ is the critical point for some $R_\Phi$-cluster, then  $z\in R(y)$ implies  $\ab{z}\in R_t(\ab{y})$.
\end{lemma}
\begin{proof}
Let $y$ be critical for cluster $C$. If $z\in R(y)$, then $\ab{y}R_\Phi\ab{z}$ (r3), so if $\ab{z}\notin C$ then immediately $\ab{y}R_t\ab{z}$. But if $\ab{z}\in C$, then $\ab{z}\in C^\circ$ and again $\ab{y}R_t\ab{z}$.
\end{proof}

\begin{lemma} \label{pthconpres}
Suppose $\Dim\top\in\Phi$.
Let $y\in W$ be a critical point, and $z,z'\in R(y)$. If $z$ and $z'$ are path connected in $R(y)$, then $\ab{z}$ and $\ab{z'} $ are path connected in $R_t(\ab{y})$.
\end{lemma}
\begin{proof}
Let $z=z_0,\dots, z_n=z'$ be a connecting path between $z$ and $z'$ within $R(y)$.
The criticality of $y$ ensures, by Lemma \ref{critpresRy}, that  $\ab{z_0},\dots, \ab{z_n}$ are all in $R_t(\ab{y})$.
We apply Lemma \ref{repair} to convert this sequence into a connecting $R_t$-path within $R_t(\ab{y})$.

For each $i<n$ we have $z_iRz_{i+1}$ or $z_{i+1}Rz_{i}$, hence $\ab{z_i}R_\Phi\ab{z_{i+1}}$ or $\ab{z_{i+1}}R_\Phi\ab{z_{i}}$ by (r3). So if there is such an $i$ that is ``defective'' in the sense that neither $\ab{z_i}R_t\ab{z_{i+1}}$ nor $\ab{z_{i+1}}R_t\ab{z_{i}}$, then by Lemma \ref{repair}, which applies since $\Dim\top\in\Phi$, there exists a $v_i$ with  $\ab{z_i}R_t v_i$ and  $\ab{z_{i+1}}R_t v_i$. Then $v_i\in R_t(\ab{y})$ by transitivity of $R_t$, as $\ab{z_i}\in R_t(\ab{y})$.
We insert $v_i$ between $\ab{z_i}$ and $\ab{z_{i+1}}$ in the sequence.
Doing this for all defective $i<n$
turns the sequence into a connecting $R_t$-path in $R_t(\ab{y})$ with unchanged  endpoints $\ab{z}$ and $\ab{z'} $.
\end{proof}

\begin{lemma} \label{RtPhi}
Suppose $\Dim\top\in\Phi$ and
$a\in W$ is $R$-maximal. Then for all $x\in W$, $\ab{x}R_t\ab{a}$ iff\/ $\ab{x}R_\Phi\ab{a}$.
\end{lemma}
\begin{proof}
$\ab{x}R_t\ab{a}$ implies $\ab{x}R_\Phi\ab{a}$ by definition of $R_t$. 
For the converse, suppose $\ab{x}R_\Phi\ab{a}$, let $C$ be the $R_\Phi$-cluster of $\ab{x}$,
and let  $K$ be the maximal $R$-cluster of $a$. 

If $\ab{a}\notin C$, then since $\ab{x}R_\Phi\ab{a}$ it is immediate that $\ab{x}R_t\ab{a}$ as required. We are left with the case  $\ab{a}\in C$. Since $\Dim\top\in\Phi$ and $\ab{x}R_\Phi\ab{a}$ we get  $\Dim\top\in x$ by (r4).  As $\ab{x}$ and $\ab{a}$ both belong to $C$, Lemma \ref{import} then gives
 $\Dim\top\in a$. So $R(a)\ne\emptyset$, implying that $R(a)=K$ and $M(a)=\{K\}$.
Moreover, since $\ab{x}R_\Phi\ab{a}$ we see that  $C$ is non-degenerate, so if $y$ is the critical point for $C$ then 
$\ab{y}R_\Phi\ab{a}$, hence $M(a)\sub M(y)$ by Lemma \ref{presM}. Thus $K\in M(y)$, making $yRa$, hence $\ab{a}\in C^\circ$ and so again $\ab{x}R_t\ab{a}$ as required. 
\end{proof}

\begin{theorem}
If\/ $\At\sub\Phi$, the frame $\F_t=(W_\Phi,R_t)$ is locally $n$-connected.
\end{theorem}
\begin{proof}
This refines the proof of Theorem \ref{FPhilnc}. 
If $u\in W_\Phi$, we have to show that $R_t(u)$ has at most $n$ path components. 
Now if $C$ is the $R_\Phi$-cluster  of $u$, then $R_t(u)$ is the union of the nucleus $C^\circ$ and all the $R_\Phi$-clusters coming strictly $R_\Phi$-after $C$. Hence
$R_t(u)=R_t(w)$ for all $w\in C$. In particular, $R_t(u)=R_t(\ab{y})$ where
 $y$ is the critical point of $C$. So we show that  $R_t(\ab{y})$ has at most $n$ path components. We take  an arbitrary sequence $v_0,\dots,v_n$ of $n+1$ points in $R_t(\ab{y})$, and prove that there must exist distinct $i$ and $j$ such that $v_i$ and $v_j$ are path connected  in $R_t(\ab{y})$.

Let $A$ be the set of all $R$-maximal points in $R(y)$.
For each $i\leq n$ we have $v_i\in R_\Phi(\ab{y})$ and so
by Lemma \ref{R*max} there is an $a_i\in A\subseteq R(y)$ such that $v_iR_\Phi^*\ab{a_i}$. Hence 
$v_iR_t^*\ab{a_i}$ by Lemma \ref{RtPhi}.
 This gives us a sequence $a_0,\dots,a_n$   of  members of $R(y)$. But $R(y)$ has at most $n$ path components, by Theorem \ref{canlnc}. Hence there exist $i\ne j\leq n$ such that $a_i$ and $a_j$ are path connected in $R(y)$. Therefore by Lemma \ref{pthconpres},
$\ab{a_i}$ and $\ab{a_j} $ are path connected in $R_t(\ab{y})$. Since $v_iR_t^*\ab{a_i}$ and $v_jR_t^*\ab{a_j}$,
and $v_i,v_j\in R_t(|y|)$,
 it follows that 
$v_i$ and $v_j $ are path connected in $R_t(\ab{y})$. That shows that $R_t(\ab{y})$ does not have more than $n$ path components.
\end{proof}

This result combines with the analysis as in other cases to give the  finite model property for the tangle systems K4G$_nt$,    K4G$_nt$.U,  K4G$_nt$.UC, KD4G$_nt$, KD4G$_nt$.U,  and KD4G$_nt$.UC
for all $n\geq 1$.

\newpage\part{}
In the second part of the paper, we prove topological completeness theorems for
the logics discussed in part~1.
The results in part~1 will of course be used, but much of part~2 can be read
independently --- indeed, nearly all of it, if the reader takes the results of part~1 on trust.

\section{Further basic definitions}

In this section,  the main definitions, notation, and
 basic results needed in Part 2 are developed.

\subsection{Topological spaces}\label{ss:top spaces}

We will assume some familiarity with topology,
but we take some time to reprise the
main concepts and notation.
A \emph{topological space} is a pair $(X,\tau)$,
where $X$ is a set and $\tau\subseteq\wp(X)$ satisfies:
\begin{enumerate}

\item if $\c S\subseteq\tau$ then $\bigcup\c S\in\tau$,

\item if $\c S\subseteq\tau$ is finite then $\bigcap\c S\in\tau$,
on the understanding that $\bigcap\emptyset=X$.

\end{enumerate}
So $\tau$ is a set of subsets of $X$ closed under unions and finite intersections.
By taking $\c S=\emptyset$,
it follows that $\emptyset,X\in\tau$.
The elements of $\tau$ are called \emph{open subsets} of $X$, or just \emph{open sets.}
An \emph{open \nhd} of a point $x\in X$ is an open set containing $x$.
A subset $C\subseteq X$ is called \emph{closed} if $X\setminus C$ is open.
The set of closed subsets of $X$ is closed under intersections and finite unions.
If $O$ is open and $C$ closed then  $O\setminus C$ is open and
$C\setminus O$ is closed.

We use the  signs $\int$, $\cl$, $\did$ to denote the \emph{interior,}
\emph{closure,} and \emph{derivative} operators, respectively.
So for $S\subseteq X$, 
\begin{itemize}
\item $\int S=\bigcup\{O\in\tau:O\subseteq S\}$ --- the largest
open set contained in $S$,

\item   $\cl S=\bigcap\{C\subseteq X:C$ closed, $S\subseteq C\}$
--- the smallest closed set containing $S$;
we have $\cl S=\{x\in X:S\cap O\neq\emptyset
\mbox{ for every open \nhd\ } O\mbox{ of } x\}$,

\item $\did S=\{x\in X:S\cap O\setminus\{x\}\neq\emptyset
\mbox{ for every open \nhd\ } O\mbox{ of } x\}$.

\end{itemize}
Then $\int S\subseteq S\subseteq\cl S\supseteq\did S$.
For all subsets $A,B$ of $X$, we have 
\[
\begin{array}{rcl}
\cl(A\cup B)&=&\cl A\cup\cl B,
\\
\did(A\cup B)&=&\did A\cup\did B, 
\\
\int(A\cap B)&=&\int A\cap\int B.
\end{array}
\]
That is, \emph{closure and $\did$ are additive and interior is multiplicative.}

We follow standard practice and identify (notationally) the space $(X,\tau)$
with $X$.
The reader should note that we do allow empty topological spaces, where $X=\emptyset$.
This is particularly useful when dealing with subspaces.

A \emph{subspace} of $X$ is a topological space of the form
$(Y,\{O\cap Y:O\in\tau\})$, for (possibly empty) $Y\subseteq X$.
It is a subset of $X$, made into a topological space by endowing it with what is called
the \emph{subspace topology.}
It is said to be an \emph{open subspace} if $Y$ is an open subset of $X$.
As with $X$, we identify (notationally) the subspace with its underlying set, $Y$.
We write $\int_Y,\cl_Y$ for the operations of interior and closure in the subspace $Y$.
It can be checked that for every $S\subseteq Y$ we have
$\cl_YS=Y\cap\cl S$, and if $Y$ is an open subspace then $\int_YS=\int S$.

We will be considering various properties that a topological space $X$ may have.
We leave most of them for later, but we mention now
that $X$ is said to be  
\emph{dense in itself} if
no singleton subset is open, 
\emph{connected} if it is not the union of two disjoint non-empty open sets,
and \emph{separable} if it has a countable subset $D$ with $X=\cl D$.
$X$ is \emph{T1} if  every singleton subset $\{x\}$ is closed, and \emph{T$_D$} if the 
derivative $\did\{x\}$ of every singleton is closed, which is equivalent to requiring $\did\did\{x\}\sub\did\{x\}$. The T$_D$ property, introduced in \cite{aull:sepa62}, is strictly weaker than T1.

\subsection{Metric spaces}\label{ss:metric spaces}

A \emph{metric space} is a pair $(X,d)$,
where $X$ is a 
set
and $d:X\times X\to\R$ is a `distance function'
(having nothing to do with the modal operator $\did$) satisfying, for all $x,y,z\in X$,
\begin{enumerate}
\item $d(x,y)\geq0$,

\item $d(x,y)=0$ iff $x=y$,

\item $d(x,y)=d(y,x)$,

\item $d(x,z)\leq d(x,y)+d(y,z)$ (the `triangle inequality').
\end{enumerate}
We assume some experience of working with this definition,
in particular with the triangle inequality.
Examples of metric spaces abound and include the real numbers
$\R$ with the standard distance function $d(x,y)=|x-y|$,
$\R^n$ with Pythagorean distance, etc.
As usual, we  often identify (notationally) $(X,d)$ with $X$.

Let $(X,d)$ be a metric space, and $x\in X$.
For non-empty $S\subseteq X$, define 
\[
d(x,S)=\inf\{d(x,y):y\in S\}.
\]
We leave $d(x,\emptyset)$ undefined.
For a real number $\varepsilon>0$,
we let $N_\varepsilon(x)$ denote the so-called `open ball' 
$\{y\in X:d(x,y)<\varepsilon\}$.  
A metric space $(X,d)$ gives rise to a topological space $(X,\tau_d)$ 
in which
a subset $O\subseteq X$ is declared to be open (i.e., in $\tau_d$)
iff for every $x\in O$, there is some $\varepsilon>0$ such that
$N_\varepsilon(x)\subseteq O$.
In other words, the open sets are the unions of open balls.
We frequently regard a metric space $(X,d)$ equally as a topological space
$(X,\tau_d)$.
So, we will say that a metric space has a given topological property
(such as being dense in itself) if the associated topological space has the property.
As an example,  every metric space is T$_D$, since it has the stronger Hausdorff (or T2) property.

A \emph{subspace} of a metric space $(X,d)$ is a pair
of the form $(Y,d\restriction Y\times Y)$, where $Y\subseteq X$.
It is plainly a metric space, and the topological space
$(Y,\tau_{d\restriction Y\times Y})$
is a subspace of $(X,\tau_d$).

\subsection{Topological semantics}
Given a topological space $X$, 
an \emph{assignment} into $X$ is simply
a map $h:\Var\to\wp(X)$.
A \emph{topological model} is a pair $(X,h)$, where
$X$ is a topological space and $h$ an assignment into~$X$.
We will also be  considering topological models where $\Var$ is replaced by some other set of atoms. Details will be given later.

As with Kripke models, we attribute a topological property
to a topological model if the underlying topological space has the property.

For every topological model $(X,h)$ and every point $x\in X$, we define 
$(X,h),x\models\varphi$, for a
$\Lbig$-formula $\varphi$, by induction on $\varphi$:
\begin{enumerate}
\item $(X,h),x\models p$ iff $x\in h(p)$, for $p\in\Var$.

\item $(X,h),x\models\top$.

\item $(X,h),x\models\neg\varphi$ iff $(X,h),x\not\models\varphi$.

\item $(X,h),x\models\varphi\wedge\psi$ iff $(X,h),x\models\varphi$ and $(X,h),x\models\psi$.

\item $(X,h),x\models\bo\varphi$ iff 
there is an open \nhd\ $O$ of $x$ with $(X,h),y\models\varphi$ for every
$y\in O$.

\item $(X,h),x\models\bod\varphi$ iff 
there is an open \nhd\ $O$ of $x$ with $(X,h),y\models\varphi$ for every
$y\in O\setminus\{x\}$.  We do not require $\varphi$ to hold at $x$ itself.

\item $(X,h),x\models\forall\varphi$ iff 
$(X,h),y\models\varphi$ for every $y\in X$.

\item\label{topsem clause 8} For a non-empty 
finite set $\Delta$ of formulas for which we have 
inductively defined semantics,
write $\sem\delta=\{x\in X:(X,h),x\models\delta\}$,
for each $\delta\in\Delta$.
Then define:
\begin{itemize}
\item $(X,h),x\models\dit\Delta$ iff 
there is some $S\subseteq X$ such that
$x\in S\subseteq\bigcap_{\delta\in\Delta}\cl(\sem\delta\cap S)$,

\item  $(X,h),x\models\didt\Delta$ iff 
there is some $S\subseteq X$ such that
$x\in S\subseteq\bigcap_{\delta\in\Delta}\did(\sem\delta\cap S)$.

\end{itemize}

\item Suppose 
inductively that $\sem\varphi_h=\{x\in X:(X,h),x\models\varphi\}$ is well defined, for every assignment $h$ into $X$.
Define a map $f:\wp(X)\to\wp(X)$  by
\[
f(S)=\sem\varphi_{h[S/q]}\quad\mbox{for }S\subseteq X,
\]
where $h[S/q]$ is defined as in  Kripke semantics (section~\ref{ss:Kripke sem}).
Again,
$f$ is monotonic, and we define $(X,h),x\models\mu q\varphi$ iff $x\in LFP(f)$.
\end{enumerate}
The definition makes sense but has no content if $X$ is empty: there are
no points $x\in X$ to evaluate at.
Writing $\sem\varphi_h=\{x\in X:(X,h),x\models\varphi\}$,
we have $\sem{\bo\varphi}_h=\int(\sem\varphi_h)$,
$\sem{\di\varphi}_h=\cl(\sem\varphi_h)$, and
$\sem{\did\varphi}_h=\did(\sem\varphi_h)$ for each $\varphi,h$.
Again, 
$\sem{\nu q\varphi}=GFP(f)$, where $\varphi,f$ are as in the last clause.

\begin{remark}\label{rmk:sem of tangle}\rm
Again we briefly discuss the semantics of $\dit$ and $\didt$
(see clause~\ref{topsem clause 8} above).
 With $\varphi\equiv\psi$ redefined to mean that 
$(X,h),x\models\varphi\leftrightarrow\psi$ for every topological model $(X,h)$ and $x\in X$,
the equivalences in \eqref{e:mu and tangle} above continue to hold,
and indeed they motivate clause~\ref{topsem clause 8}.
However, there is a perhaps more intuitive 
meaning for $\dit$ and $\didt$ in terms of \emph{games,} which are used extensively in the mu-calculus.
Let players \pa, \pe\ play a game of length $\omega$ on $X$.
Initially, the position is $x$.
In each round, if the current position is $y\in X$,
player \pa\ chooses an open \nhd\ $O$ of $y$ and a formula $\delta\in\Delta$.
Player \pe\ must select a point $z\in O$ at which $\delta$ is true
(and with $z\neq y$ in the case of $\didt$).
If she cannot, player \pa\ wins.
That is the end of the round, and the next round commences from position $z$.
Player \pe\ wins if she survives every round.
It can be checked that $(X,h),x\models\dit\Delta$
(respectively, $(X,h),x\models\didt\Delta$) iff \pe\ has a \ws\ in this game
(respectively, the game where she must additionally choose $z\neq y$).
\end{remark}

\subsection{Topological semantics in open subspaces}

Let $X$ be a topological space and $Y$ a subspace of $X$.
Each assignment $h:\Var\to\wp(X)$ into $X$ induces
an assignment $h_Y$ into $Y$, via
$h_Y(p)=Y\cap h(p)$, for each $p\in\Var$.
Thus,  we can evaluate formulas at points in $Y$
in both $(X,h)$ and $(Y,h_Y)$.
Because the semantics of the connectives $\bo,\bod,\dit,\didt$ 
depend on only arbitrarily small open \nhd s of the evaluation point,
it is easily seen that
if $Y$ is an \emph{open} subspace of $X$,
 we get the same result for every formula not involving $\forall$.
That is, the following analogue of lemma~\ref{lem:gen submodels} holds:

\begin{lemma}\label{lem:open subspaces}
Whenever $Y$ is an open subspace of $X$, we have
$(X,h),y\models\varphi$ iff
$(Y,h_Y),y\models\varphi$,
for every $y\in Y$ and $\varphi\in\c L^{\mu\dit\didt}_{\bo\bod}$.
\end{lemma}
(This holds vacuously if $Y$ is empty.)

\subsection{Satisfiability, validity, equivalence}\label{ss:validity 2}

Let $X$ be a topological space.
A set $\Gamma$ of $\Lbig$-formulas is said to be 
\emph{satisfiable in $X$}
if there exist an  assignment $h$ into $X$ and a point $x\in X$
such that $(X,h),x\models\gamma$ for every $\gamma\in\Gamma$.

Let $\varphi$ be an $\Lbig$-formula.
We say that $\varphi$ is \emph{satisfiable in $X$}
if the set $\{\varphi\}$ is so satisfiable.
We say that $\varphi$ is \emph{valid in $X$,}
or that \emph{$X$ \emph{validates} $\varphi$,}
if $\neg\varphi$ is not satisfiable in $X$.
We also say that $\varphi$ is \emph{equivalent} to a formula $\psi$
in $X$
if $\varphi\leftrightarrow\psi$ is valid in $X$.

In any space $X$, the `4' schema: $\bo \varphi\to\bo\bo \varphi$ is valid under 
the interpretation $\sem{\bo\varphi}=\int\sem\varphi$. But the schema 
$\bod \varphi\to\bod\bod \varphi$, or equivalently
$\did\did \varphi\to \did \varphi$, is valid under 
the interpretation $\sem{\did\varphi}=\did\sem\varphi$ if, and only if, $X$ is a T$_D$ space. This is because in any space the derivatives of all subsets are closed iff the derivatives of all singletons are closed (see \cite[Theorem 5.1]{aull:sepa62}).

\subsection{Logics}

Let $\c K$ be a class of topological spaces.
In the context of a given language $\c L\subseteq\Lbig$,
the \emph{($\c L$)-logic of $\c K$} is the set of all $\c L$-formulas 
that are valid in every member of $\c K$.
Exactly as for Kripke semantics, a Hilbert system $H$ for $\c L$ 
with set of theorems  $T$ is said to be 
\begin{itemize}
\item \emph{sound over $\c K$}
if  $T$ is a subset of the logic of $\c K$
(all $H$-theorems are valid in $\c K$),

\item  \emph{weakly complete}, or simply \emph{complete, over $\c K$}
if $T$ contains the logic of $\c K$
(all $\c K$-valid formulas are $H$-theorems),

\item \emph{strongly complete over $\c K$} if
every countable $H$-consistent set $\Gamma$ of $\c L$-formulas
is satisfiable in some \str\ in $\c K$.
\end{itemize}
For example the $\c L_\bod$-system K4 is sound and complete over the class of all T$_D$-spaces, a result due to Esakia (see \cite{Esa04:apal}).

The logic of a single space $X$ is defined to be
the logic of the class $\{X\}$; similar definitions are used for the other terms here. 

We say that a topological space $X$ \emph{validates $H$}
if $H$ is sound over~$X$.
To establish this, it is enough to check that each axiom of $H$ is valid in $X$,
and that each rule of $H$ preserves $X$-validity.

It can be checked that $H$ is weakly complete
over $\c K$ iff every \emph{finite} $H$-consistent
set of formulas is satisfiable in some space in $\c K$.
Hence, every strongly complete Hilbert system is also weakly complete.
The main aim of this part of the paper is to provide Hilbert systems
that are (where possible) sound and strongly complete over various topological spaces,
with respect to various sublanguages of 
$\Lbig$.

\section{Translations}\label{sec:translations}

The language $\Lbig$
has some redundancy.
We can  express $\bo$ with $\bod$, and $\dit$ with $\didt$
(but not vice versa).
We can also express $\dit,\didt$ with $\mu$ --- and often vice versa,
using results of Dawar and Otto \cite{DO09}.

Later, we will need translations that work in both 
 topological spaces and (possibly restricted) Kripke models.
In this  section, we will explore translations --- but only to the extent needed for later work.
We will again assume that $\Var$ is infinite.

\subsection{Translating $\did$ and $\didt$ to $\mu$}

This is the simplest case.  We have already seen the idea, in
the equivalence of $\dit$- and $\dit$-formulas to $\nu$-formulas given 
in~\eqref{e:mu and tangle} in section~\ref{ss:Kripke sem}.
\begin{definition}
For each $\Lbig$-formula
$\varphi$, we define a $\c L^{\mu}_{\bo\bod\forall}$-formula
$\varphi^\mu$ as follows:
\begin{enumerate}
\item $p^\mu=p$ for $p\in\Var$.

\item $-^\mu$ commutes with the boolean connectives, 
 $\bo$, $\bod$, $\forall$, and $\mu$ (cf.~definition~\ref{def:phi*}).

\item $(\dit\Delta)^\mu=\nu q\bigwedge_{\delta\in\Delta}\di(\delta^\mu\wedge q)$, 
where $q\in\Var$ does not occur in any $\delta^\mu$ ($\delta\in\Delta$).

\item $(\didt\Delta)^\mu=\nu q\bigwedge_{\delta\in\Delta}\did(\delta^\mu\wedge q)$,
where $q\in\Var$ does not occur in any $\delta^\mu$ ($\delta\in\Delta$).
\end{enumerate}
\end{definition}
These formulas can be checked to be well formed.
The translation simply replaces 
$\dit$ by an expression using $\mu$ and $\bo$,
and similarly for $\didt$.
So if $\varphi\in\c L_{\bo}^{\dit}$
then $\varphi^\mu\in\c L_{\bo}^{\mu}$,
if $\varphi\in\c L_{\bod}^{\didt}$
then $\varphi^\mu\in\c L_{\bod}^{\mu}$, etc.

This translation is faithful in all relevant semantics:

\begin{lemma}\label{lem:mu trans}
Let $\varphi$ be any $\Lbig$-formula.
Then
$\varphi$ is equivalent to $\varphi^\mu$ in 
every transitive Kripke frame and  in every topological space.
(See sections~\ref{ss:validity} and~\ref{ss:validity 2} for the definition of equivalence.)
\end{lemma}

\begin{proof}
An easy induction on $\varphi$.
We consider only the case $\dit\Delta$ (for finite $\Delta\neq\emptyset$), 
and only in Kripke semantics
(the case $\didt\Delta$ is of course identical).
Assume the lemma for each $\delta\in\Delta$.
Take any transitive Kripke model
$\c M=(W,R,h)$ and any $w\in W$.
Inductively, 
$\c M,w\models(\dit\Delta)^\mu$
iff $\c M,w\models\nu q\bigwedge_{\delta\in\Delta}\di(\delta\wedge q)$.
By the post-fixed point characterisation of greatest fixed points,
this holds iff $(*)$ there is $S\subseteq W$
with $w\in S$ and such that for every $s\in S$ and $\delta\in\Delta$,
 there is $t\in S$ with $sRt$ and $\c M,t\models\delta$.
 
Assuming $(*)$, it is easy to choose a sequence $w=s_0Rs_1Rs_2\ldots$ in $S$
by induction so that $\{n<\omega:\c M,s_n\models\delta\}$ is infinite for every $\delta\in\Delta$.
It follows that 
$\c M,w\models\dit\Delta$.
Conversely, if $\c M,w\models\dit\Delta$ then
there are worlds $w=w_0Rw_1Rw_2\ldots$ in $W$
with $\{n<\omega:\c M,w_n\models\delta\}$ infinite for every $\delta\in\Delta$.
Let $S=\{w_n:n<\omega\}$.
Then $w\in S$, and 
for each $w_n\in S$ and $\delta\in\Delta$, there is 
$m>n$ with  $\c M,w_m\models\delta$.
Then $w_m\in S$, and by transitivity of $R$ we have $w_nRw_m$.
So $(*)$ holds.
\end{proof}

\subsection{Translating $\bo$ to $\bod$ and $\dit$  to $\didt$}
Just replacing $\bo$ by $\bod$ and $\dit$ by $\didt$ in a formula $\varphi\in\Lbig$
yields an $\c L^{\mu\didt}_{\bod\forall}$-formula equivalent to $\varphi$
in all Kripke frames.  But the two are not equivalent in topological spaces, so we seek a better translation that works in both semantics.

\begin{definition}
For each $\Lbig$-formula
$\varphi$, we define a $\c L^{\mu\didt}_{\bod\forall}$-formula
$\varphi^d$ as follows:
\begin{enumerate}
\item $p^d=p$ for $p\in\Var$.

\item $-^d$ commutes with the boolean connectives, $\bod$,  $\didt$,
$\forall$, and $\mu$.

\item $(\bo\varphi)^d=\varphi^d\wedge\bod\varphi^d$.

\item $(\dit\Delta)^d=(\bigwedge\Delta^d)\vee
\did(\bigwedge\Delta^d)\vee(\didt\Delta^d)$, where $\Delta^d=\{\delta^d:\delta\in\Delta\}$.
\end{enumerate}
\end{definition}
Again, $\varphi^d$ is always well formed.
The translation $-^d$ is pretty good:
\begin{lemma}\label{lem:-d in refl frames}
Each $\Lbig$-formula $\varphi$
is equivalent to $\varphi^d$ in every reflexive Kripke frame.
\end{lemma}

\begin{proof}
An easy induction on $\varphi$. To show, e.g., that
$\bo\varphi$ implies $(\bo\varphi)^d$, we need reflexivity.
We also note that $\bigwedge\Delta$ and $\did\bigwedge\Delta$
both imply $\dit\Delta$ in reflexive Kripke models.
\end{proof}

\begin{lemma}\label{lem:trans equivalent top}
Each $\Lbig$-formula $\varphi$
is equivalent to $\varphi^d$ in a topological space $X$ if, and only if, $X$ is T$_D$.
\end{lemma}

\begin{proof}
Let $X$ be a T$_D$ topological space.
We prove by induction on $\varphi$ that 
each $\Lbig$-formula $\varphi$ is equivalent
to $\varphi^d$ in $X$.
We consider only two cases: $\bo\varphi$ and $\dit\Delta$.
Inductively assume the result for $\varphi$ and  each formula in the finite
set $\Delta$ of formulas,
let $h$ be an assignment into $X$, and let $x\in X$.
In the proof, we write `$x\models{}$' as short for `$(X,h),x\models$',
and for a formula $\varphi$, we write
$\sem\varphi=\{y\in X:y\models\varphi\}$.

We prove that $x\models\bo\varphi\leftrightarrow(\bo\varphi)^d$.
We have $x\models\bo\varphi$ iff
for some open \nhd\ $O$ of $x$, we have
$(X,h),y\models \varphi$ for every $y\in O$.
This is plainly iff $x\models\varphi\wedge\bod\varphi$.
Inductively, this is iff
$x\models\varphi^d\wedge\bod\varphi^d$ --- i.e., iff
$x\models(\bo\varphi)^d$.

Now we prove that $x\models\dit\Delta\leftrightarrow(\dit\Delta)^d$.
Recall that 
\[(\dit\Delta)^d=(\bigwedge\Delta^d)\vee
\did(\bigwedge\Delta^d)\vee(\didt\Delta^d).
\]
First we prove that $x\models(\dit\Delta)^d\to\dit\Delta$.
Suppose that $x\models(\dit\Delta)^d$.
To show that $x\models\dit\Delta$, we need to find $S\subseteq X$
with $x\in S\subseteq\bigcap_{\delta\in\Delta}\cl(\sem\delta\cap S)$.
If $x\models\bigwedge\Delta^d$, take  $S=\{x\}$.
If $x\models\did\bigwedge\Delta^d$,
take $S=\{x\}\cup\sem{\bigwedge\Delta^d}$.
And if $x\models\didt\Delta^d$,
there is $S\subseteq X$ with $x\in S\subseteq\bigcap_{\delta\in\Delta}\did(\sem\delta\cap S)$; then $x\in S\subseteq\bigcap_{\delta\in\Delta}\cl(\sem\delta\cap S)$ 
as required.

It remains to prove that $x\models\dit\Delta\to(\dit\Delta)^d$.
So suppose that $x\models\dit\Delta$.
If $x\models(\bigwedge\Delta^d)\vee\did(\bigwedge\Delta^d)$, we are done.

So suppose not.
Thus,  there is an open \nhd\
$U$ of $x$ with $y\models\neg\bigwedge\Delta^d$ for every $y\in U$.
So for every $y\in U$, there is $\delta_y\in\Delta$
with $y\models\neg\delta_y^d$.

We prove that $x\models\didt\Delta^d$.

Since $x\models\dit\Delta$,
there is $S\subseteq X$ with 
$x\in S\subseteq\bigcap_{\delta\in\Delta}\cl(\sem\delta\cap S)$.

\claim Put $S'=U\cap S$. Then
$x\in S'\subseteq\bigcap_{\delta\in\Delta}\did(\sem{\delta^d}\cap S')$.

\pfclaim 
Plainly, $x\in S'$.
For the other half, 
let $y\in S'$ and $\delta\in\Delta$ be arbitrary;
we show that $y\in\did(\sem{\delta^d}\cap S')$.
So let $O$ be any open \nhd\ of $y$. 
As $X$ is T$_D$, $\did\{y\}$ is closed, so since it does not contain $y$,
 $O\cap U\setminus \did\{y\}$ is an open \nhd\ of $y$ too.
As $y\in S'\subseteq S\subseteq\cl(\sem{\delta_y}\cap S)$,
there is some $z\in  O\cap U\cap S\setminus \did\{y\}$ with
$z\models\delta_y$.
But $y\models\neg\delta_y^d$, so inductively,
$y\models\neg\delta_y$.
It follows that $z\neq y$.

Now we have $z\notin\{y\}\cup \did\{y\}=\cl\{y\}$, so
$O\cap U\setminus\cl\{y\}$ is an open \nhd\ of $z$.
Since $z\in S\subseteq\cl(\sem\delta\cap S)$,
there is some $t\in O\cap U\cap S\setminus\cl\{y\}=O\cap S'\setminus\cl\{y\}$ with
 $t\models\delta$. Then $t\ne y$.
Since $O$  was  arbitrary, this shows that $y\in\did(\sem\delta\cap S')$.
Since inductively, 
$\sem\delta=\sem{\delta^d}$, this proves the claim.

By definition of the semantics, the claim immediately yields $x\models\didt\Delta^d$ as required.
This completes the induction and the proof that each $\varphi$
is equivalent to $\varphi^d$.
(The reader may like to construct an alternative proof using the games described in remark~\ref{rmk:sem of tangle}.)

Conversely, to show that the T$_D$ hypothesis is necessary, we first prove
\begin{quote}
\begin{lemma}  \label{dxclosed}
In any space $X$, for any $x\in X$, $\cl\did\{x\}\setminus \did\{x\}\sub\{x\}$.
Hence
$\did\{x\}$ is closed iff $x\notin \cl\did\{x\}$.
\end{lemma}

\begin{proof}
For the first part, since  $\did\{x\}\sub \cl\{x\}$ and the latter is closed,  
$\cl\did\{x\}\sub \cl\{x\}=\did\{x\}\cup \{x\}$. This implies $\cl\did\{x\}\setminus \did\{x\}\sub\{x\}$.

For the second part, $\did\{x\}$ is closed iff $\cl\did\{x\}\setminus \did\{x\}=\emptyset$. 
By the first part, this holds iff $x\notin\cl\did\{x\}\setminus \did\{x\}$.
But $x\notin\did\{x\}$, so $x\notin\cl\did\{x\}\setminus \did\{x\}$ iff 
$x\notin \cl\did\{x\}$.
\end{proof}
\end{quote}
Now suppose the space $X$  is not T$_D$. Then there is some point $x$ of $X$ with $\did\{x\}$ not closed. By  Lemma \ref{dxclosed},
$x\in \cl\did\{x\}$. Hence $\cl\{x\}\sub\cl\did\{x\}$.
Let $p\in\Var$ 
and  $h:\Var\to\wp X$ satisfy
$h(p)=\{x\}$ for some (arbitrary) $x$.
Then $(X,h),x\models\dit\{p,\did p\}$, but $(X,h),x\not\models(\dit\{p,\did p\})^d$,
i.e.\ $(X,h),x\not\models (p\wedge \did p)\vee\did(p\wedge\did p)\vee\didt\{p,\did p\}$, giving a case of $\ph$ not being equivalent to $\ph^d$.
That $x\not\models (p\wedge \did p)\vee\did(p\wedge\did p)$ follows because
$\sem{p\wedge \did p}=\{x\}\cap\did\{x\}=\emptyset$. That $x\not\models\didt\{p,\did p\}$ follows as no punctured \nhd\ $O\setminus\{x\}$ contains a point of $\sem{p}=\{x\}$.
To see that $x\models\dit\{p,\did p\}$, let $S=\cl\{x\}$.
Then $S$ is included in both 
$\cl(\sem{p}\cap S)=\cl\{x\}=S$ and
$\cl(\sem{\did p}\cap S)=\cl(\did\{x\})$ (because $\cl\{x\}\sub\cl\did\{x\}$ as noted above).
Since $x\in S$, it follows that $x\models\dit\{p,\did p\}$ .
\end{proof}

\subsection{Translating $\mu$ to $\dit$}\label{ss:DO09}

We use this translation only to prove strong completeness
for $\c L_\bo^\mu$ in theorem~\ref{thm:str compl boxes}(\ref{stro compl boxes part2}).
Fortunately, most of the hard work involved has already been done by others.
We will need only the fact below, but its proof was a  major enterprise.

\begin{fact}[Dawar--Otto, \hbox{\cite[theorem 4.57(5)]{DO09}}]\label{fact:DO}
For each formula $\varphi$ of $\c L^\mu_{\bo}$,
there is a  formula $\varphi^t$ of $\c L^{\dit}_{\bo}$ that is
equivalent to $\varphi$ in every finite transitive Kripke frame.
\end{fact}

To lift this to topological spaces, we will use 
the proof theory from section~\ref{ss:hs mu}.

\begin{corollary}\label{cor:t transl top equiv}
Each $\c L_\bo^\mu$-formula $\varphi$ is equivalent to $\varphi^t$ in
every topological space.
\end{corollary}

\begin{proof}
By fact~\ref{fact:DO} and lemma~\ref{lem:mu trans}, 
$\varphi\leftrightarrow(\varphi^t)^\mu$ is
an $\c L_\bo^\mu$-formula
valid in every finite transitive Kripke frame.
By theorem~\ref{thm:S4mu sc},
$\axS4\mu\vdash\varphi\leftrightarrow(\varphi^t)^\mu$.

Now it is easy to check that 
$\axS4\mu$ is sound over every topological space.
(The S4 axioms are sound by definition of the topological
semantics of $\bo$, and the fixed point  axiom and rule are sound 
by the semantics of $\mu$.)
Hence, $\varphi\leftrightarrow(\varphi^t)^\mu$ is valid in every 
topological space.
But by lemma~\ref{lem:mu trans}, $(\varphi^t)^\mu$ is equivalent
to $\varphi^t$ in every topological space.
We conclude that 
$\varphi$ is equivalent to $\varphi^t$ in every 
topological space, as required.
\end{proof}
By the corollary and lemma~\ref{lem:mu trans},
$\c L_\bo^\mu$ and $\c L_\bo^\dit$ uniformly have the same expressive power
in every topological space.

Since $\bo,\bod$ and $\dit,\didt$ are indistinguishable in Kripke semantics,
a similar analysis would give a translation
from $\c L_\bod^\mu$ to $\c L_\bod^\didt$ valid
in every topological space.
(For this purpose,
the T axiom $\bo\varphi\to\varphi$ would be dropped in section~\ref{ss:hs mu},
and the translation in definition~\ref{def:phi*} 
adapted to represent transitive closure.)
The translation would show that $\c L_\bod^\mu$ and $\c L_\bod^\didt$ 
are equally expressive over all T$_D$ topological spaces.
We could use it to lift weak completeness
for $\c L_\bod^\mu$ to strong completeness.
Unfortunately,  we do not have a weak completeness result for $\c L_\bod^\mu$ to lift.

\section{More topology}

The  finite model property theorems of Part~1 will be instrumental in
our completeness theorems for (some) topological spaces.
Not surprisingly, we will also need some simple and standard topological definitions and results, 
together with some more substantial ones.
The first one is very simple.

\begin{lemma}\label{lem:inf}
Let $X$ be a dense-in-itself T$_D$ topological space.
Then every non-empty open subset of $X$ is infinite.
\end{lemma}

\begin{proof}
It suffices to show that every non-empty open subset $O$ has a non-empty open \emph{proper} subset $O'$, since infinitely many iterations of that fact will produce an infinite sequence of distinct points in $O$. Take any $x\in O$. Then $x$ belongs to $ O\setminus\did\{x\}$, which is open as $\did\{x\}$ is closed in the T$_D$-space. Since $X$ is dense-in-itself, there must then be some $y\ne x$ with $y\in O\setminus\did\{x\}$. As $y\ne x$ and $y\notin \did\{x\}$,  $y$ has an open \nhd\ $U$ with $x\notin U$. Put $O'=O\cap U$ to get that $O'$ is open, non-empty as it contains $y$, and a proper subset of $O$ as it does not contain $x$.
\end{proof}

\subsection{The $\did$ operator on sets}

Let $X$ be a topological space.
For a set $S\subseteq X$,
recall that 
$\did S=
\{x\in X:S\cap O\setminus\{x\}\neq\emptyset$ for every open \nhd\ $O$ of $x\}$, the set of strict limit points of $S$.
The $\did$ operator has the following basic properties.

\begin{lemma}\label{lem:did cl}
Let $S,T\subseteq X$.
\begin{enumerate}
 \item $\cl S=S\cup\did S$.

\item\label{didcl additive}  $\did$ is additive: $\did(S\cup T)=\did S\cup\did T$.

\item\label{didcl part5} If $X$ is
dense in itself,
then (i) $\int S\subseteq\did S$,
and (ii) if $S$ is open then $\did S=\cl S$.
\end{enumerate}
\end{lemma}

\begin{proof}
Easy.
\end{proof}

\subsection{Regular open sets}

Let $X$ be a topological space.
A \ro\ subset of $X$ is one 
equal to the interior of its closure.
We will mainly be interested in \ro\ subsets of  open subspaces of $X$,
so we give definitions directly for such situations.

\begin{definition}\label{def:ro subset}
Let $U$ be an open subset of $X$.
A subset $S$ of $X$ is said to be a \emph{\ro\ subset of $U$}
if  $S=\int(U\cap\cl S)$.
\end{definition}
As `$\int$' is multiplicative and $U$ is open, it is equivalent to say that $S=U\cap\int\cl S$, and we
sometimes prefer this formulation.
In such a case, $S\subseteq U$ and $S$ is open.
So $S=\int_U\cl_US$: $S$ is a \ro\ subset of the subspace $U$ of $X$.
It is worth noting that if $S\subseteq U$ is arbitrary then $\int_U\cl_U S$ is 
a \ro\ subset of $U$.

It is known
(see, e.g., \cite[chapter 10]{GiHa:ba})
that for every open subset $U$ of $X$,
the set $RO(U)$ of \ro\ subsets of $U$
is closed under the
operations $+,\cdot,-,0,1$ defined by
\begin{itemize}
\item $S+S'=U\cap\int\cl(S\cup S')$
\item $S\cdot S'=S\cap S'$
\item $-S=U\setminus\cl S$
\item $0=\emptyset$ and $1=U$,
\end{itemize}
and $(RO(U),+,\cdot,-,0,1)$ is a (complete) \ba.
We will also use the notation $RO(U)$ to denote this \ba.
The standard boolean ordering
$\leq$ on $RO(U)$ coincides with
set inclusion, because for $S,T\in RO(U)$ we have
 $S\leq T$ iff $S\cdot T=S$,
 iff $S\cap T=S$, iff $S\subseteq T$.
We will need the following general lemma.

\begin{lemma}\label{lem:ro1}
Let $V\subseteq U$ be  open subsets of $X$,
and $S,S'$ be \ro\ subsets of $U$.
\begin{enumerate}
\item\label{lem:ro:part1} If $T=U\setminus\cl S$, 
then $T$ is also a \ro\ subset of $U$, with
$S=U\setminus\cl T$ and $U\setminus S\subseteq\cl T$.

\item\label{lem:ro:part2} If $U\cap\cl S\cap\cl S'=\emptyset$,
then $S+S'=S\cup S'$.

\item\label{lem:ro:part3} If $S\subseteq V$,
then $S$ is a \ro\ subset of $V$.

\item\label{lem:ro:part4} Every \ro\ subset of $S$ is a \ro\ subset of $U$.

\end{enumerate}
\end{lemma}

\begin{proof}
\begin{enumerate}
\item The first two points follow from \ba\ considerations, and
can easily be shown directly.
The third point, $U\setminus S\subseteq\cl T$, follows from
$U\setminus\cl T=S$.

\item
Since $S,S'\leq S+S'$ and $\leq$ coincides with $\subseteq$,
we obtain $S,S'\subseteq S+S'$ and so $S\cup S'\subseteq S+S'$.
Conversely,
it is easy to check\footnote{Indeed, $\bo\di(p\vee q)\to\bo\di p\vee\bo\di q\vee(\di p\wedge\di q)$
is valid in $\axS4$ frames, so provable in $\axS4$. 
Since $\axS4$ is sound over $X$, the formula is valid in $X$.}
that
\[
\int\cl(S\cup S')\subseteq \int\cl S\cup\int\cl S'\cup(\cl S\cap\cl S').
\]
Since $U\cap\cl S\cap\cl S'=\emptyset$,
\[
S+S'=U\cap\int\cl(S\cup S')\subseteq (U\cap\int\cl S)\cup(U\cap\int\cl S')=S\cup S',
\]
as required.

\item 
$V\cap\int\cl S=(V\cap U)\cap \int\cl S=V\cap (U\cap \int\cl S)=V\cap S=S$.

\item Let $T$ be a \ro\ subset of $S$. Clearly, $\int\cl T\subseteq\int\cl S$.
So $U\cap\int\cl T=U\cap(\int\cl S\cap\int\cl T)=(U\cap\int\cl S)\cap\int\cl T=S\cap\int\cl T=T$.
\end{enumerate}\unskip
\end{proof}

\subsection{Normal spaces}

\begin{definition}\label{def:normal}
A topological space $X$ is said to be
\emph{Hausdorff} (or T2) if  for every two distinct points $x_0,x_1\in X$,
there are disjoint open sets $O_0,O_1$ with $x_0\in O_0$ and $x_1\in O_1$,
and  \emph{normal} (or T4) if
it is Hausdorff 
and for every two disjoint closed subsets $C_0,C_1$ of $X$,
there are disjoint open sets $O_0,O_1$ with $C_0\subseteq O_0$ and $C_1\subseteq O_1$.
\end{definition}
Equivalently, $X$ is normal iff
it is Hausdorff
and if $C\subseteq O\subseteq X$, $C$ closed, and $O$ open, then there is open $P$
with $C\subseteq P\subseteq\cl P\subseteq O$.

\begin{lemma}\label{lem:norm++}
Let $C_0,C_1$
be disjoint closed subsets of a normal topological space $X$. Then there are 
\ro\ subsets
$O_0,O_1$ of $X$ with disjoint closures,  such that
$C_0\subseteq O_0$ and $C_1\subseteq O_1$.
\end{lemma}

\begin{proof}
By normality, there are disjoint open sets 
$O^-_0\supseteq C_0$ and $U\supseteq C_1$.
Then $O_0^-\subseteq X\setminus U$, a closed set.
So $O_0=\int\cl O_0^-$ is a \ro\ subset of $X$ disjoint from $U$.
We have $C_0\subseteq O_0^-\subseteq O_0\subseteq\cl O_0\subseteq X\setminus U$, so $\cl O_0$ and $C_1$ are disjoint closed sets.
By normality again, there are disjoint open sets $V\supseteq \cl O_0$ and 
$O_1^-\supseteq C_1$.
Let $O_1=\int\cl O_1^-$, a \ro\ subset of $X$ disjoint from $V$.
Then $C_1\subseteq O_1^-\subseteq O_1\subseteq\cl O_1\subseteq X\setminus V$, so $\cl O_0\cap\cl O_1=\emptyset$.
Now $O_0,O_1$ are as required.
\end{proof}

The following is well known (see, e.g., \cite[III, 6.1]{RS:mm}), but is so important for us that we include a quick proof.

\begin{lemma}\label{lem:ms}
Every metric space is normal.
\end{lemma}

\begin{proof}
Let $X$ be a metric space.
It is easy to check that $X$ is Hausdorff, and we leave this to the reader.
Let $C,D$ be disjoint closed subsets of $X$.
By symmetry, it is enough to
show that there is open $O\supseteq C$ with $\cl(O)\cap D=\emptyset$.
If $C=\emptyset$, take $O=\emptyset$.  If $D=\emptyset$ take $O=X$.
So we can suppose $C,D\neq\emptyset$, and thus define 
\[
O=\{x\in X:d(x,C)<d(x,D)/2\}
\]
(recall from section~\ref{ss:metric spaces} that $d(x,S)=\inf\{d(x,s):s\in S\}$
for non-empty $S\subseteq X$).
Then $C\subseteq O$, because if $x\in C$ then $d(x,C)=0$,
while $x\notin D$, so $d(x,D)>0$ as $D$ is closed.
It is easily seen that $O$ is  open
and $\cl(O)\subseteq \{x\in X:d(x,C)\leq d(x,D)/2\}$,
so it is enough to show that this latter set is disjoint from $D$.
If $x$ is in both, 
then $d(x,C)\leq d(x,D)/2=0$ so $x\in C$ as $C$ is closed.  
This contradicts the assumption that $C\cap D=\emptyset$.
\end{proof}

\subsection{Tarski's theorem and relatives}

\def\setG{\mathbb{G}}
\def\setB{\mathbb{B}}

The primary topological results needed later 
(for representing finite Kripke frames in proposition~\ref{prop:di rep})
are 
provided by the next theorem.
A recent related result is 
\cite[proposition 11]{KudShe14}.

\begin{theorem}\label{thm:Tarski}\label{thm:Tarski++}
Let $X$ be a \nice\ space.
\begin{enumerate}

\item\label{Tarski thm 1} Let $\setT,\setU$ be open subsets of $X$, with $\emptyset\neq\setT\subseteq\setU$.
Let $k<\omega$.
Then there are pairwise disjoint non-empty subsets $\Vec\setI k\subseteq\setT$
satisfying 
\[
\did\setI_i=\cl(\setT)\setminus\setU\quad\mbox{for each }i\leq k.
\]

\item \label{Tarski thm 2}Let $\setG$ be a non-empty open subset of $X$,
and let $r,s<\omega$.  Then $\setG$ can be partitioned into non-empty open
subsets $\vec \setG r$ and other non-empty sets $\Vec \setB s$ 
such that, letting
\[
D=\cl(\setG)\setminus \bigcup_{1\leq l\leq r}\setG_l,
\]
we have $\cl(\setG_i)\setminus \setG_i=D$ for each $i=1,\ldots,r$,
and $\did\setB_j=D$ for each $j=0,\ldots,s$.
\end{enumerate}

\end{theorem}

Part~\ref{Tarski thm 2} above is essentially known.
Paraphrasing slightly, Tarski \cite[satz 3.10]{Tar38}
proved the following.
Let $X$ be a dense-in-itself normal topological space 
 with a countable
basis of open sets (see below).
Then for every $r<\omega$,
every non-empty open subset
$\setG$ of $X$ can be partitioned into non-empty
open sets $\vec\setG r$ and a non-empty set $\setB_0$
such that
$\cl(\setG)\setminus\setG\subseteq\cl\setB_0
\subseteq\cl\setG_1\cap\ldots\cap\cl\setG_r$.
Here and below, the empty intersection (when $r=0$)
is taken to be $X$.
This statement is equivalent to the statement in part~\ref{Tarski thm 2} of  theorem~\ref{thm:Tarski++} above in the case $s=0$ and with $\did \setB_j$ replaced by $\cl\setB_j$.

A topological space $(X,\tau)$ has a countable
basis of open sets iff 
there is countable $\tau_0\subseteq\tau$
such that $\tau$ is the smallest topology on $X$ containing $\tau_0$.
Given this and normality,
Urysohn's theorem \cite{Urysohn25:met} yields that $\tau=\tau_d$ for some metric $d$ on $X$.
Any metric space is normal, and has a countable basis of open sets iff it is 
separable (see section~\ref{ss:top spaces}).
So   Tarski's stipulation on $X$ boils down to stipulating that $X$ is a separable \nice\ space.

Removing  the restriction to $s=0$ but with the same hypotheses on $X$,
McKinsey and Tarski \cite[theorem 3.5]{McKiT44} proved that
for every $r,s<\omega$, every non-empty open set
$\setG$ can be partitioned into non-empty open sets
$\vec\setG r$ and non-empty sets $\Vec\setB s$
with 
$\cl(\setG)\setminus\setG\subseteq\cl\setB_0=\cdots=\cl\setB_s\subseteq\cl\setG_1\cap\ldots\cap\cl\setG_r$.
This statement is equivalent to the statement
of theorem~\ref{thm:Tarski++}(\ref{Tarski thm 2}) above,
with $\did \setB_j$ replaced by $\cl\setB_j$.
It was used in \cite{McKiT44} to prove 
(in our terminology) that the $\c L_\bo$-logic of $X$ is S4.

Removing the assumption of separability,
Rasiowa and Sikorski \cite[III, 7.1]{RS:mm}
proved  theorem~\ref{thm:Tarski++}(\ref{Tarski thm 2})
as formulated above, but with $\did \setB_j$ replaced by $\cl\setB_j$.
Our use of $\did\setB_j$ is only a formal strengthening of 
\cite[III, 7.1]{RS:mm},
since the same effect can be achieved by first obtaining disjoint
sets $\setB_j^i$ with $\cl\setB_j^i=D$ for $j=0,\ldots,s$ and $i=0,1$,
and then defining $\setB_j=\setB_j^0\cup\setB_j^1$ for each $j$. 
As $\setB_j^0\cap\setB_j^1=\emptyset$, using lemma~\ref{lem:did cl} we have
\[
D\subseteq (D\setminus\setB_j^0)\cup(D \setminus\setB_j^1)
=(\cl\setB_j^0\setminus\setB_j^0)\cup(\cl\setB_j^1\setminus\setB_j^1)
\subseteq\underbrace{\did\setB_j^0\cup\did\setB_j^1}_{\did\setB_j}
\subseteq
\cl\setB_j^0\cup\cl\setB_j^1=D,
\]
so $\did\setB_j=D$ as required.
Given this, the reader may ask why we give a proof of part~2 at all.
The answer is that we wish to make clear the affinity between the two parts of the theorem,
as well as make our paper more self contained
and explicit as to the topological arguments needed in our completeness proof.

\begin{proof}
We will get to the theorem shortly, but first, fix $k<\omega$.
We define a game, $\c G_k$,
 to build pairwise disjoint subsets $\Vec\setI k$ of $X$.
The game has two players,
\pa\ (male) and 
\pe\ (female), and $\omega$ rounds, numbered $0,1,2,\ldots.$
At the start of round $n$ (for each $n<\omega$), pairwise disjoint sets 
$\Vec{I^n}k\subseteq X$ are in play,
satisfying
\begin{equation}\label{e:did=0}
\did I_i^n=\emptyset\quad\mbox{ for each }i\leq k.
\end{equation}
Observe that each $I_i^n$ is closed,
because by lemma~\ref{lem:did cl}, $\cl I^n_i=I^n_i\cup\did I^n_i=I^n_i$.
Also, 
\begin{equation}\label{e:UIjn empty int}
\int\Big(\bigcup_{j\leq k}I^n_j\Big)=\emptyset.
\end{equation}
For if $U\subseteq\bigcup_{j\leq k}I^n_j$ is open, then
by lemma~\ref{lem:did cl} and \eqref{e:did=0},
\[
U\subseteq\cl U=\did U\subseteq\did\bigcup_{j\leq k}I_j^n
=\bigcup_{j\leq k}\did I_j^n=\emptyset.
\]
The game starts off with all of the sets empty: 
$I_0^0=\cdots=I_{k}^0=\emptyset$.
Round $n$ is played as follows.
Player \pa\ moves first, by playing a triple 
$
(\varepsilon_n,i_n,O_n),
$
of his choice,
where 
$\varepsilon_n>0$ is a real number, $i_n\leq k$,
and $O_n$ is a non-empty open subset of $X$.
Let
\begin{equation}\label{e:Pn}
P_n=O_n\setminus\bigcup_{j\leq k}I_j^n.
\end{equation}
Then $P_n\neq\emptyset$:
for otherwise, $\emptyset\neq O_n\subseteq\bigcup_{j\leq k}I_j^n$,
contradicting \eqref{e:UIjn empty int}.
Player \pe\ responds to \pa's move by using Zorn's lemma to choose a maximal subset 
$Z_n\subseteq P_n$
such that $d(x,y)\geq\varepsilon_n$ for each distinct $x,y\in Z_n$.
Observe that 
\begin{enumerate}
\renewcommand{\theenumi}{Z\arabic{enumi}}
\renewcommand{\labelenumi}{Z\arabic{enumi}.}

\item \label{cond Z1} $\did Z_n=\emptyset$ (because for all $x\in X$, 
the set $N_{\varepsilon_n/2}(x)\cap Z_n$ 
has  at most one element).
Just as with $I^n_i$ above, it follows that $Z_n$ is closed. 

\item\label{cond Z2}  $Z_n$ is non-empty (because 
$P_n$ is non-empty and any singleton subset of $P_n$ satisfies the $\varepsilon_n$-condition).
 
\item\label{cond Z3} $d(x,Z_n)<\varepsilon_n$ for every 
$x\in P_n$ (else $x$ can be added to $Z_n$, contradicting its maximality).
Recall again that $d(x,Z_n)=\inf\{d(x,z):z\in Z_n\}$, which is defined
because $Z_n$ is non-empty.
\end{enumerate}
Player \pe\ then extends $I^n_{i_n}$ by $Z_n$, leaving the other
sets $I^n_i$ unchanged.  Formally, she defines
\[
\begin{array}{rcll}
I_{i_n}^{n+1}&=&I_{i_n}^n\cup Z_n,
\\[4pt]
I_i^{n+1}&=&I_i^n&\mbox{for each }i\leq k\mbox{ with }i\neq i_n.
\end{array}
\] 
This completes the round, and the sets
$\Vec{I^{n+1}}k$ are passed to the start of round $n+1$.
Note that \eqref{e:did=0} 
holds for these sets, since
$\did I_{i_n}^{n+1}=\did I_{i_n}^n\cup\did Z_n=\emptyset$
by lemma~\ref{lem:did cl}, \eqref{e:did=0} for $I_{i_n}^n$, and \ref{cond Z1} above.
Also, by~\eqref{e:Pn}, $Z_n$ is disjoint from each $I^n_i$, so the 
$I^{n+1}_i$ $(i\leq k)$ are pairwise disjoint.

At the end of the game, we define
$\setI_i=\bigcup_{n<\omega}I_i^n$ for each $i\leq k$.
Plainly,  $\Vec\setI k$ are pairwise disjoint.

We say that \pa\ \emph{plays well} in $\c G_k$
if his choices of $\varepsilon_n$
tend to zero, the set $\{n<\omega:i_n=i\}$ is infinite for each $i\leq k$,
and his choices of $O_n$ form a descending chain: $O_0\supseteq O_1\supseteq\cdots$.

It is clear by condition \ref{cond Z2} above that if \pa\ plays well then 
$\Vec\setI k$ are all non-empty.

\claim In any play (match?) of the game in which \pa\ plays well,
for each $i\leq k$ we have
\[
\did\setI_i=\bigcap_{n<\omega}\cl O_n.
\]

\pfclaim 
Let $n<\omega$.
Define $I_i^{>n}=\setI_i\setminus I_i^n$.
This is the set of points that \pe\ added to $\setI_i$ in or after round $n$.
By the game rules and because \pa\ played well,
$I_i^{>n}\subseteq\bigcup_{n\leq m<\omega}Z_m
\subseteq\bigcup_{n\leq m<\omega}O_m=O_n$.
Obviously, $\setI_i=I_i^n\cup I_i^{>n}$.
So by lemma~\ref{lem:did cl} and \eqref{e:did=0},
\[
\did\setI_i=\did(I_i^n\cup I_i^{>n})=\did I_i^n\cup\did I_i^{>n}
=\did I_i^{>n}\subseteq\did O_n\subseteq\cl O_n.
\]
This holds for all $n$,
so $\did\setI_i\subseteq\bigcap_{n<\omega}\cl O_n$.

Conversely, let $x\in\bigcap_{n<\omega}\cl O_n$.
Let a real number $\varepsilon>0$ be given.
Since \pa\ plays well, we can pick a round, say $n$, such that
\pa\ chose $\varepsilon_n\leq\varepsilon$ and
$i_n=i$, and such that if $x\in\setI_i$ then already $x\in I_i^n$.
Since $x\in\cl O_n$,
the set $N_\varepsilon(x)\cap O_n$ is non-empty, 
and plainly it is  open.
As before,  \eqref{e:UIjn empty int} implies that
$N_\varepsilon(x)\cap O_n\setminus\bigcup_{j\leq k}I^n_j$
is non-empty as well.
Fix a point $y$ in this set.
Then $y\in P_n$ and $d(x,y)<\varepsilon$.

In round $n$, player \pe\ picks $Z_n\subseteq P_n$
satisfying conditions \ref{cond Z1}--\ref{cond Z3} above.
Observe that $x\notin Z_n$,
because otherwise, $x\in Z_n\subseteq\setI_i$ (since $i_n=i$),
so by assumption on $n$ we have
$x\in I_i^n$, so by~\eqref{e:Pn},
$x\notin P_n\supseteq Z_n$, a contradiction.
Since $y\in P_n,$ by \ref{cond Z3} we have $d(y,Z_n)<\varepsilon_n$.
Since $d(x,y)<\varepsilon$,
we have $d(x,Z_n)<\varepsilon+\varepsilon_n\leq 2\varepsilon$.
So there is $z\in Z_n\subseteq\setI_i$ with $z\neq x$
(since $x\notin Z_n$) and $d(x,z)<2\varepsilon$.
This holds for all $\varepsilon>0$,
and it follows that $x\in\did\setI_i$, proving the claim.

\medskip

Now we prove part 1 of the theorem.
Suppose first that $\cl(\setT)\setminus\setU=\emptyset$.
Noting that $\setT$ is infinite (by lemma~\ref{lem:inf}),
we can take $\Vec\setI k$ to be disjoint singleton
subsets of $\setT$. Plainly, 
all requirements are met.

So suppose that $\cl(\setT)\setminus\setU\neq\emptyset$.
Let \pa\ and \pe\ play the game $\c G_k$.
We suppose that \pa\ plays well, and also so that for each $n<\omega$,
\[
O_n=\setT\cap\bigcup_{x\in \cl(\setT)\setminus\setU}N_{\varepsilon_n}(x).
\]
Note that $O_n$ is open, and 
non-empty because $\cl(\setT)\setminus\setU\neq\emptyset$, so \pa\ can legally play it.
Then $\Vec\setI k$  are pairwise disjoint,
and non-empty since \pa\ plays well.
We have $Z_n\subseteq O_n\subseteq\setT$ for each $n$,
so  $\Vec\setI k$
are subsets of $\setT$.
By the claim, $\did\setI_i=\bigcap_{n<\omega}\cl O_n$ for each $i\leq k$,
so it suffices to show that $\bigcap_{n<\omega}\cl O_n=\cl(\setT)\setminus\setU$.

Certainly, each 
$x\in\cl(\setT)\setminus\setU$ lies in $\cl O_n$ for each $n$,
because for every $\varepsilon>0$, 
\[
O_n\cap N_\varepsilon(x)
\supseteq 
\Big(\setT\cap\bigcup_{y\in \cl(\setT)\setminus\setU}N_{\varepsilon_n}(y)\Big)
\cap N_{\min(\varepsilon,\varepsilon_n)}(x)
=\setT\cap N_{\min(\varepsilon,\varepsilon_n)}(x)\neq\emptyset.
\]
So $\cl(\setT)\setminus\setU\subseteq\bigcap_{n<\omega}\cl O_n$.
Conversely, first note that
$O_0\subseteq\setT$,
so $\bigcap_{n<\omega}\cl O_n\subseteq\cl O_0\subseteq\cl\setT$.
It remains to show that $\setU\cap\bigcap_{n<\omega}\cl O_n=\emptyset$.
Suppose for contradiction that there is some $x\in\setU\cap\bigcap_{n<\omega}\cl O_n$. As $\setU$ is open, we can choose $\delta>0$ with $N_\delta(x)\subseteq\setU$.
As \pa\ played well, we can pick $n<\omega$ such that $\varepsilon_n\leq\delta$.
Then $x\in\cl O_n$,
so $d(x,O_n)=0$.
By definition of $O_n$, for each $y\in O_n$
we have $d(y,\cl(\setT)\setminus\setU)<\varepsilon_n$.
So $d(x,\cl(\setT)\setminus\setU)<\varepsilon_n$ as well.
As $\varepsilon_n\leq\delta$ and  $N_\delta(x)\subseteq\setU$,
this is a contradiction.
We conclude that indeed $\setU\cap\bigcap_{n<\omega}\cl O_n=\emptyset$,
so $\bigcap_{n<\omega}\cl O_n\subseteq\cl(\setT)\setminus\setU$, as required.
We have proved part~1 of the theorem.

\medskip

To prove part 2, let \pa\ and \pe\ play $\c G_{s+r}$.
As we will see, \pa\ will play so that $\Vec\setI{s+r}\subseteq\setG$.
In the end, $\vec\setB s$ will be $\vec\setI s$,
$\vec\setG r$ will be `fattened' versions
of $\setI_{s+1},\ldots,\setI_{s+r}$, and $\setB_0$ will be the rest
of $\setG$ (we will have $\setB_0\supseteq\setI_0$).
For the fattening, at the start of round $n$ (for each $n<\omega$), 
for each $j=s+1,\ldots,s+r$, \pa\ defines an auxiliary open set
$G^n_j$ such that
\begin{eqnarray}
I^n_{j}\subseteq G^n_{j}
\label{e:G sets 1}
\\
G^0_j\subseteq G^{1}_j\subseteq\cdots
\label{e:G sets 2}
\\
I^n_0,\ldots,I^n_s,\cl G^n_{s+1},\ldots,\cl G^n_{s+r}
\mbox{ are pairwise disjoint subsets of }\setG.
\label{e:G sets 3}
\end{eqnarray}
The sets $G^n_j$ are for \pa's own private use and are not formally part of the game.
(If $r=0$, there are no $j$ in range and he does nothing.)
At the start of round~0, he simply puts $G^0_{s+1}=\cdots=G^0_{s+r}=\emptyset$.
Suppose we are at the start of round~$n$, for arbitrary $n<\omega$,
and that \pa\ has defined open $G^n_{j}\supseteq I^n_{j}$
($s+1\leq j\leq s+r$)  satisfying
\eqref{e:G sets 1}--\eqref{e:G sets 3}.
In round~$n$ he plays $(\varepsilon_n,i_n,O_n)$, where
$i_0=0$,
\begin{equation}\label{e:pa's O_n in Tarski}
O_n=\setG\setminus\bigcup_{s+1\leq j\leq s+r}\cl G^n_j,
\end{equation}
and the $\varepsilon_n,i_n$ are chosen so that overall, he plays well.
By \eqref{e:G sets 2}, 
$O_0\supseteq O_1\supseteq\cdots$,
as required for him to play well.
(We remark that if $r=0$ then $O_n=\setG$ for all $n$.)

We check that this is always a legal move for \pa.
Certainly, $O_n$ is open.  We show that it is always non-empty.
For $n=0$ we plainly have $O_0=\setG\neq\emptyset$.
In round $0$, \pa\ plays $i_0=0$,
and \pe\ defines $I^1_0=Z_0\neq\emptyset$ by condition \ref{cond Z2} above.
Since the $I^n_0$ form a chain,
$I^n_0\supseteq I^1_0\neq\emptyset$ for all $n>0$, and by~\eqref{e:G sets 3}
and~\eqref{e:pa's O_n in Tarski},
$I^n_0\subseteq O_n$.  So $O_n\neq\emptyset$ for all $n$.

Player \pe\ continues round $n$
by selecting $Z_n\subseteq P_n$ and defining
$I^{n+1}_{i_n}=I^n_{i_n}\cup Z_n$ according to the rules.

It is now time for \pa\ to define $G^{n+1}_j$ for $j=s+1,\ldots,s+r$.
If $i_n\leq s$,
he leaves the sets unchanged,  defining $G^{n+1}_j=G^n_j$ for all $j$.
Trivially, conditions~\eqref{e:G sets 1}--\eqref{e:G sets 2} continue to hold.
We check \eqref{e:G sets 3}.
First, $Z_n\subseteq P_n$,
so $I^{n+1}_{i_n}$ is disjoint from $I^{n+1}_j$ for $i_n\neq j\leq s$.
Second, if $s+1\leq j\leq s+r$
then $I^{n+1}_{i_n}=I^n_{i_n}\cup Z_n\subseteq I^n_{i_n}\cup O_n$;
by~\eqref{e:G sets 3}, $I^n_{i_n}$ is disjoint from $\cl G^n_j=\cl G^{n+1}_j$,
and by~\eqref{e:pa's O_n in Tarski}, $O_n$ is disjoint from $\cl G^{n+1}_j$
as well.

If instead, $i_n>s$, then \pa\
defines $G^{n+1}_j=G^n_j$ for  $j\neq i_n$,
and uses normality of $X$ to choose an open set $G^{n+1}_{i_n}$
satisfying
\begin{equation}\label{e:A's set}
\overbrace{\cl(G_{i_n}^n)\cup Z_n}^{\rm closed}\subseteq G^{n+1}_{i_n}\subseteq \cl(G^{n+1}_{i_n})\subseteq 
\overbrace{\setG\setminus
\big(\bigcup_{j\leq s}I^n_j\;\;\cup\bigcup_{\genfrac{}{}{0pt}{1}{s+1\leq j\leq s+r}{j\neq i_n}}\cl(G^n_{j})\big)}^{\rm open}.
\end{equation}
We need to check some things here.
First, by condition \ref{cond Z1} above, $Z_n$ is closed and
so the left-hand side of \eqref{e:A's set} is closed.
Similarly, we saw just after \eqref{e:did=0}
that each $I^n_j$ is closed,
so the right-hand side of \eqref{e:A's set} is open.
Second,
it follows from \eqref{e:G sets 3}
that $\cl(G_{i_n}^n)$ is contained in the right-hand side of \eqref{e:A's set}. 
Also $Z_n\subseteq P_n\subseteq O_n$,
and it follows from \eqref{e:Pn} and \eqref{e:pa's O_n in Tarski}
 that
$Z_n$ is contained in the right-hand side of \eqref{e:A's set} as well.
So $G^{n+1}_{i_n}$ can be found as stated.

We also need to check
\eqref{e:G sets 1}--\eqref{e:G sets 3} for the $G^{n+1}_j$.
Condition~\eqref{e:G sets 1} holds because 
$I^{n+1}_{i_n}=I^n_{i_n}\cup Z_n\subseteq G^n_{i_n}\cup Z_n\subseteq G^{n+1}_{i_n}$,
and for $j\neq i_n$ we have
$G^{n+1}_j=G^n_j\supseteq I^n_j=I^{n+1}_j$.
Conditions~\eqref{e:G sets 2} and~\eqref{e:G sets 3} are clear from the definitions and 
\eqref{e:A's set}.

As promised, at the end of play we define
\[
\begin{array}{rcll}
\setG_i&=&\displaystyle\bigcup_{n<\omega}G^n_{s+i}&\mbox{for  }1\leq i\leq r,
\\[14pt]
\setB_j&=&\setI_j&\mbox{for  }1\leq j\leq s,
\\[4pt]
\setB_0&=&\displaystyle\setG\setminus\Big(\bigcup_{1\leq i\leq r}\setG_i
\cup\bigcup_{1\leq j\leq s}\setB_j
\Big)
\\[14pt]
D&=&\displaystyle\cl(\setG)\setminus\bigcup_{1\leq l\leq r}\setG_l.
\end{array}
\]
Note that $\setI_{s+i}\subseteq\setG_i$ for $1\leq i\leq r$
by \eqref{e:G sets 1},
and $\setI_j\subseteq\setB_j$ for $j\leq s$ by the definitions.
Because \pa\ played well, the $\setG_j$ are non-empty (and plainly open)
and the $\setB_j$ are non-empty.
It follows from \eqref{e:G sets 3} that together they partition $\setG$.

For the final piece of the theorem, there are two preliminaries.
First, we observe that each set $\setG_i$
($1\leq i\leq r$) has a nice property.
Each time \pa\ plays $i_n= s+i$ in some round~$n$,
by \eqref{e:G sets 2}, \eqref{e:A's set}, and the definition of $\setG_i$,
for every $m\leq n$ we have
$\cl G^m_{s+i}\subseteq\cl G^n_{s+i}\subseteq G^{n+1}_{s+i}\subseteq\setG_i$.
Since \pa\ played $i_n=s+i$ infinitely often, it follows that
\begin{equation}\label{e:closure inside}
\cl G^m_{s+i}\subseteq\setG_i\quad\mbox{for each }
m<\omega\mbox{ and } 1\leq i\leq r.
\end{equation}
Second, we use this to show that 
\begin{equation}\label{e:value of D}
D=\bigcap_{n<\omega}\cl O_n.
\end{equation}
Note that if $C\subseteq S\subseteq X$ and $C$ is closed,
then $S= C\cup(S\setminus C)\subseteq C\cup\cl(S\setminus C)$;
the right-hand side is closed, so
$\cl S\subseteq C\cup\cl(S\setminus C)$, whence
 $\cl(S)\setminus C\subseteq\cl(S\setminus C)$.
Now, for each $n<\omega$
we have
\[
\begin{array}{rcll}
D&=&\cl(\setG)\setminus\bigcup_{1\leq i\leq r}\setG_i
&\mbox{by definition}
\\
&\subseteq&\cl(\setG)\setminus\bigcup_{1\leq i\leq r}\cl G^n_{s+i}
&\mbox{by \eqref{e:closure inside}}
\\
&\subseteq&\cl\big(\setG\setminus\bigcup_{1\leq i\leq r}\cl G^n_{s+i}\big)
\quad
&
\mbox{by the observation above}
\\
&=&\cl O_n&\mbox{by \eqref{e:pa's O_n in Tarski}.}
\end{array}
\]
So $D\subseteq\bigcap_{n<\omega}\cl O_n$.
Conversely,
we certainly have $\bigcap_{n<\omega}\cl O_n\subseteq\cl O_0=\cl\setG$
 since $O_0=\setG$.
Now fix $i$ with $1\leq i\leq r$.
By \eqref{e:pa's O_n in Tarski}, 
for each $n<\omega$ we have $G^n_{s+i}\cap O_n=\emptyset$,
so as $G^n_{s+i}$ is open,
$G^n_{s+i}\cap\cl O_n=\emptyset$.
It follows that 
\[
\setG_i\cap\bigcap_{n<\omega}\cl O_n=(\bigcup_{n<\omega}G^n_{s+i})\cap\bigcap_{n<\omega}\cl O_n=\emptyset.
\]
This holds for each $i$,
so  $\bigcap_{n<\omega}\cl O_n\subseteq\cl(\setG)\setminus\bigcup_{1\leq i\leq r}\setG_i=D$,
proving \eqref{e:value of D}.

Now we can finish easily.
For each $0\leq j\leq s$, 
we plainly have
$\setB_j\subseteq\setG\setminus\bigcup_{1\leq l\leq r}\setG_l
\subseteq D$.
Since $D$ is closed,
$\did\setB_j\subseteq\cl\setB_j\subseteq D$.
Conversely, by \eqref{e:value of D} and the claim,
$D=\bigcap_{n<\omega}\cl O_n=\did\setI_j\subseteq\did\setB_j$.

Similarly, take $i$ with $1\leq i\leq r$.
Since 
the $\setG_l$ ($1\leq l\leq r$) are pairwise disjoint open subsets of $\setG$, we have
$\cl\setG_i\subseteq\cl(\setG)\setminus\bigcup_{l\neq i}\setG_l$
and hence 
$\cl(\setG_i)\setminus\setG_i\subseteq
\cl(\setG)\setminus\bigcup_{1\leq l\leq r}\setG_l=D$.
Conversely, by \eqref{e:value of D}, the claim, and lemma~\ref{lem:did cl} we have
$D=\bigcap_{n<\omega}\cl O_n=\did\setI_{s+i}\subseteq\did\setG_i\subseteq\cl\setG_i$.
By definition, $D\cap\setG_i=\emptyset$.
So $D\subseteq\cl(\setG_i)\setminus\setG_i$, as required.
\end{proof}

\begin{corollary}\label{cor:Tarski++opens}
Let $\setU$ be an open subspace of a \nice\ space $X$, and
suppose that $\setS_0$, $\setS_1$ are open subsets of $\setU$
such that  $\setU\cap\cl \setS_0\cap\cl \setS_1=\emptyset$
and $\setT=\setU\setminus\cl(\setS_0\cup \setS_1)\neq\emptyset$.
Then  there are \ro\ subsets $\setU_0,\setU_1$ of $\setU$
such that $\setU\cap\cl\setU_0\cap\cl\setU_1=\emptyset$, and for each $i=0,1$:
\begin{enumerate}
\item $\setU\cap\cl \setS_i\subseteq\setU_i$,

\item writing $\setT_i=\setU_i\setminus\cl \setS_i$,
 we have $\setT_i\neq\emptyset$ and
 $\cl(\setT)\setminus\setU\subseteq\cl\setT_i$.
\end{enumerate}
\end{corollary}

\begin{proof}
Since $\setT$ is a non-empty open subset of $\setU$,
we can use theorem~\ref{thm:Tarski++} to choose disjoint non-empty subsets
$\setI_0,\setI_1\subseteq\setT$
such that $\did\setI_0=\did\setI_1=\cl(\setT)\setminus\setU$.

We now work in the subspace $\setU$. 
Recall that  $\cl_{\setU}$ denotes the closure operator in the subspace topology on $\setU$,
so $\cl_{\setU} K=\setU\cap\cl K$ for subsets $K\subseteq\setU$.
The sets
\[
\cl\nolimits_{\setU} \setS_0,
\;
\cl\nolimits_{\setU}\setS_1,
\;
\setI_0,
\;
\setI_1
\]
are 
pairwise disjoint (by assumptions) and closed in $\setU$.
(Each $\setI_i$ is closed in $\setU$ because
by lemma~\ref{lem:did cl},
$\cl_{\setU}\setI_i=\setU\cap\cl\setI_i=\setU\cap(\setI_i\cup\did\setI_i)
=\setU\cap(\setI_i\cup(\cl(\setT)\setminus\setU))=\setU\cap\setI_i=\setI_i$.)
Hence, 
$\setI_0\cup\cl\nolimits_{\setU}\setS_0$ and $\setI_1\cup\cl\nolimits_{\setU}\setS_1$
are disjoint closed subsets of $\setU$.
The subspace $\setU$ is a metric space in its own right, and
so, by lemma~\ref{lem:ms}, normal.
Using lemma~\ref{lem:norm++} in $\setU$,
we can find \ro\ subsets $\setU_0,\setU_1$ of $\setU$ with 
\begin{equation}\label{cor:pf:cond1}
\setI_i\cup\cl\nolimits_{\setU}\setS_i\subseteq\setU_i\subseteq\setU
\quad\mbox{for } i=0,1,
\end{equation}
and 
$\cl\nolimits_{\setU}\setU_0\cap\cl\nolimits_{\setU}\setU_1=\emptyset$.
Working back in $X$ again, this says that
\begin{equation}\label{cor:pf:cond2}
\setU\cap\cl\setU_0\cap\cl\setU_1=\emptyset.
\end{equation}
Now for each $i=0,1$, write $\setT_i=\setU_i\setminus\cl \setS_i$.
By definition, $\setI_i\subseteq\setU_i$.
 Also, $\setI_i\cap(\setU\cap\cl \setS_i)=\emptyset$,
and since $\setI_i\subseteq\setU$, this gives $\setI_i\cap\cl \setS_i=\emptyset$.
Hence, $\setI_i\subseteq\setT_i$, so $\setT_i\neq\emptyset$.
We now obtain
\begin{equation}\label{cor:pf:cond3}
\cl(\setT)\setminus\setU=\did\setI_i\subseteq\cl\setI_i\subseteq\cl\setT_i.
\end{equation}
Lines \eqref{cor:pf:cond1}, \eqref{cor:pf:cond2}, and \eqref{cor:pf:cond3}, 
together with
$\setT_i\neq\emptyset$, establish the corollary.
\end{proof}

\section{Representations of frames over topological spaces}

Our next aim is to use the results of the preceding section to 
construct a `\rep' from an arbitrary \nice\
space to any given finite connected locally connected KD4 Kripke frame.
The notion of \rep\ is chosen so as to preserve $\c L_{\bod\forall}^\mu$-formulas,
and this will allow us to prove completeness theorems in the next two sections.

Until the end of section~\ref{ss:full rby}, we fix a topological space $X$
and a finite Kripke frame $\c F=(W,R)$.
We will frequently regard the elements of $W$
as propositional atoms.

\subsection{Representations}

The following definition 
seems to originate with Shehtman:
see equation (71) in \cite[\S5, p.25]{Sheh:d90}.

\begin{definition}
A~map $\rho:X\to W$ is said to be a \emph{\rep\ of $\c F$ over $X$}
if for every $x\in X$ and $w\in W$ we have
\[
(X,\rho^{-1}),x\models\did w\iff R(\rho(x),w).
\]
\end{definition}

Here, $\rho^{-1}$ assigns an atom $w\in W$ to the 
possibly empty subset $\{x\in X:\rho(x)=w\}$ of $X$.
The condition says that
for every $x\in X$,
the set of points  of $W$ with preimages  under $\rho$ 
in every open \nhd\ of $x$ but distinct from $x$ itself is precisely $R(\rho(x))$.
Equivalently, $\did\rho^{-1}(w)=\rho^{-1}(R^{-1}(w))$ for every $w\in W$,
where $R^{-1}$ is the converse relation of $R$.

Note that  $\rho$ need not be surjective.
Indeed, the empty map is vacuously a \rep\ of $\c F$ over the empty space  ---  and we definitely do allow empty \rep s.

It can be checked that if $\rho:X\to W$ is a \rep\ then $R\restriction\rng\rho$ is transitive.
Endow $W$ with the topology
generated by  $\{R(w):w\in W\}$ (so the open sets are those $A\subseteq W$
such that $a\in A$ implies $R(a)\subseteq A$).
Then every \rep\ of $\c F$ over $X$ is an interior map from $X$ to $W$: that is,
a map that is both continuous and open.  (Many other topological completeness proofs use interior maps.)
The converse, however, does not hold in general.  See \cite{GEG05,LucBry11} for 
more information.

Although Shehtman  uses the term `d-p-morphism'  (when $\rho$ is surjective),
here we will call $\rho$ a `\rep' because it is closely related to the \rep s
of algebras of relations seen in algebraic logic. 
Indeed, if $\rho$ is a surjective \rep\ of $(W,R)$ over $X$
 then $\rho^{-1}$ induces an embedding from $\wp(W)$ into $\wp(X)$
that preserves the algebraic structure with which these power sets can be naturally endowed.

\subsection{Representations over subspaces}
Our main interest is in \rep s over $X$ itself,
but \rep s over subspaces are also useful in proofs.
Given a subspace $U$ of $X$,
a map $\rho:U\to W$ induces a well defined assignment
$\rho^{-1}:W\to\wp(X)$ by $\rho^{-1}(w)=\{x\in X:x\in U$ and $\rho(x)=w\}$, for $w\in W$.
Put simply, preimages under $\rho$ of elements of $W$ are obviously subsets of $U$, but they are also subsets of 
$X$, and so $\rho^{-1}$ can be regarded equally as an assignment into $U$ or $X$, as appropriate. The following easy lemma gives some connections between the two views.

\begin{lemma}\label{lem:bod reps invar}
Let  $U$ be a subspace of $X$ and let $\rho:U\to W$ be a map.
Let $x\in U$ and $w\in W$ be arbitrary.
\begin{enumerate}
\item If $(U,\rho^{-1}),x\models\did w$ then $(X,\rho^{-1}),x\models\did w$.

\item If $U$ is open in $X$, then $(U,\rho^{-1}),x\models\did w$ iff $(X,\rho^{-1}),x\models\did w$.
\end{enumerate}
\end{lemma}

\begin{proof}
For the first part, assume that $(U,\rho^{-1}),x\models\did w$
 and let $O$ be any open \nhd\ of $x$ in $X$.
Then $O\cap U$ is an open \nhd\ of $x$ in $U$, so by assumption,
there is  $y\in O\cap U\setminus\{x\}$ with $(U,\rho^{-1}),y\models w$.
Then $y\in O\setminus\{x\}$ and $(X,\rho^{-1}),y\models w$.
Hence, $(X,\rho^{-1}),x\models\did w$.

For the second part, assume that $(X,\rho^{-1}),x\models\did w$.
Let $N$ be an arbitrary open \nhd\ of $x$ in $U$, so that $N=O\cap U$ for some open \nhd\ $O$ of $x$ in $X$.
As $U$ is assumed open in $X$, we see that $N$ is also open in $X$,
so by assumption, there is $y\in N\setminus\{x\}$ with $(X,\rho^{-1}),y\models w$.
Plainly, $(U,\rho^{-1}),y\models w$.
This shows that $(U,\rho^{-1}),x\models\did w$, and
the converse follows from the first part.
\end{proof}
By part 2 of the lemma, 
if $\rho$ is a \rep\ of $\c F$ over  an open subspace $U$ of $X$, then
$(X,\rho^{-1}),x\models\did w$ iff $R(\rho(x),w)$ for every $x\in U$ and $w\in W$.
So we can work in $(X,\rho^{-1})$ instead of $(U,\rho^{-1})$.
To avoid too much jumping around between subspaces, we will do this below, often without mention.
Part~\ref{lem:simple:3} of the next lemma makes it a little more explicit.
The lemma gives some general information on how \rep s of different generated subframes of 
$\c F$ over different subspaces of $X$ are related.

\begin{lemma}\label{lem:reps simple}
Let $\c G=(W',R')$ be a generated subframe of $\c F$.
Let $T$, $U$, and $U_i$ $(i\in I)$ be open 
subspaces of $X$, with $T\subseteq U=\bigcup_{i\in I}U_i$.
Finally, let $\rho:U\to W'$ be a map.
Then:
\begin{enumerate}
\item\label{lem:simple:1}  $\rho$ is a \rep\ of $\c F$ over $U$ iff it is a \rep\ of $\c G$ over $U$.

\item\label{lem:simple:2} $\rho$ is a \rep\ of $\c F$ over $U$ iff  for each $i\in I$, the restriction
 $\rho\restriction U_i$ is a \rep\ of $\c F$ over $U_i$.

\item\label{lem:simple:3} If $\rho\restriction T$ is a \rep\ of $\c F$ over $T$,
then $(X,\rho^{-1}),x\models\did w$ iff $R(\rho(x),w)$, for each $x\in T$ and $w\in W$.
\end{enumerate}

\end{lemma}

\begin{proof}
Simple.
\end{proof}

\subsection{Representations preserve formulas}

Here, we will show that surjective \rep s
 preserve all formulas of $\c L_{\bod\forall}^{\mu}$.
Since \rep s are like p-morphisms, albeit between different kinds of \str,
this is entirely expected
and the proof is essentially quite standard
--- see \cite[lemma 20]{Sheh:d90}
and \cite[corollary 2.9]{GEG05}, for example.
We do need, however, that $\c F$ is finite.
We will be able to handle larger sublanguages of $\Lbig$
by using the translations of section~\ref{sec:translations}.

Let us explain the setting.
Suppose we are given a \rep\ $\rho:X\to W$ of $\c F$ over $X$.
Recall that $\Var$ is our fixed base set of propositional variables, or atoms.
For each assignment $h:\Var\to\wp(W)$ of atoms in $\Var$ into $W$,
the map $\rho^{-1}\circ h:\Var\to\wp(X)$
is an assignment of atoms into $X$, given of course 
by 
\[(\rho^{-1}\circ h)(p)=\{x\in X:\rho(x)\in h(p)\}, \quad\mbox{ for each } 
p\in \Var.
\]
So $\rho$, or rather $\rho^{-1}$,
gives us a way to transform an assignment into $\c F$
to one into $X$, and then to evaluate a formula
in the resulting model on $X$. 
The following definition encapsulates when we get the same result as in the original model
on $\c F$:

\begin{definition}
Let $\rho:X\to W$ be a map, and let $\varphi$ be a formula of $\c L_{\bo\bod}^{\mu\dit\didt}$. 
We say that \emph{$\rho$ preserves $\varphi$} if
for  every assignment 
$h:\Var\to\wp(W)$ and every $x\in X$,
\begin{equation}\label{e:rho pres phi}
(X,\rho^{-1}\circ h),x\models\varphi\quad\mbox{iff}\quad(W,R,h),\rho(x)\models\varphi.
\end{equation}
\end{definition}

We are now ready for our main preservation result.

\begin{proposition}\label{prop:fmla pres}
Let  $\rho:X\to W$ be
a surjective \rep\ of $\c F$ over $X$.
Then $\rho$ preserves every formula of $\c L_{\bod\forall}^\mu$.
\end{proposition}

\begin{proof}
The proof is by induction on $\varphi$.
The atomic and boolean cases are easy and left to the reader.
Let $\varphi$ be a formula, and
inductively assume \eqref{e:rho pres phi} for every assignment 
$h:\Var\to\wp(W)$ and every $x\in X$.
It is sufficient to consider the cases $\did\varphi,$ $\forall\varphi$,
and $\mu q\varphi$.

First, consider $\did\varphi$.
Fix $h,x$.
Suppose  that $(W,R,h),\rho(x)\models\did\varphi$.
Choose $w\in R(\rho(x))$ with
$(W,R,h),w\models\varphi$.
As $\rho$ is a \rep,
$(X,\rho^{-1}),x\models\did w$.
So for every open \nhd\ $O$ of $x$,
there is $y\in O\setminus\{x\}$
with $\rho(y)=w$.
Since $(W,R,h),w\models\varphi$,
for any such $y$
we inductively have $(X,\rho^{-1}\circ h),y\models\varphi$.
It follows that $(X,\rho^{-1}\circ h),x\models\did\varphi$.

Conversely, suppose that $(X,\rho^{-1}\circ h),x\models\did\varphi$.
Let $\sem\varphi=\{y\in X:(X,\rho^{-1}\circ h),y\models\varphi\}$.
As $\c F$ is finite and $\did$ is additive (lemma~\ref{lem:did cl}(\ref{didcl additive})), we have
\begin{align*}
x\in \did \sem\varphi
&=\did(\sem\varphi\cap X)
=\did\Big(\sem\varphi\cap\bigcup_{w\in W}\rho^{-1}(w)\Big)
\\
&=
\did\Big(\bigcup_{w\in W}\big(\sem\varphi\cap\rho^{-1}(w)\big)\Big)
=\bigcup_{w\in W}\did(\sem\varphi\cap\rho^{-1}(w)).
\end{align*}
So we can take $w\in W$
with $x\in\did(\sem\varphi\cap\rho^{-1}(w))$.
Then $(X,\rho^{-1}),x\models\did w$, so as $\rho$ is a \rep,
$R(\rho(x),w)$.
Moreover, $\sem\varphi\cap\rho^{-1}(w)\neq\emptyset$.
Take any $y\in\sem\varphi\cap\rho^{-1}(w)$.
Then $(X,\rho^{-1}\circ h),y\models\varphi$ and $\rho(y)=w$.
Inductively, $(W,R,h),w\models\varphi$.
By Kripke semantics,
$(W,R,h),\rho(x)\models\did\varphi$, as required.

\medskip

Next, consider $\forall\varphi$.
Then $(X,\rho^{-1}\circ h),x\models\forall\varphi$
iff $(X,\rho^{-1}\circ h),y\models\varphi$ for all $y\in X$,
iff $(W,R,h),\rho(y)\models\varphi$  for all $y\in X$ (by the inductive hypothesis \eqref{e:rho pres phi}),
iff $(W,R,h),w\models\varphi$ for all $w\in W$ (since $\rho$ is surjective),
iff $(W,R,h),\rho(x)\models\forall\varphi$.
\medskip

Finally consider the case $\mu q\varphi$, assumed well formed. 
Fix arbitrary $h:\Var\to\wp(W)$.
We  define an assignment $h^\alpha:\Var\to\wp(W)$ for each ordinal $\alpha$.
For each atom $p\neq q$, we set $h^\alpha(p)=h(p)$.  We define $h^\alpha(q)$
by induction on $\alpha$ as follows:
\begin{itemize}
\item $h^0(q)=\emptyset$,

\item $h^{\alpha+1}(q)=\{w\in W:(W,R,h^\alpha),w\models\varphi\}$,

\item $h^\delta(q)=\bigcup_{\alpha<\delta}h^\alpha(q)$ for limit ordinals $\delta$.

\end{itemize}
Of course, $W$ is finite, but we need all ordinals
for the argument below. 
Let $\eta=\rho^{-1}\circ h:\Var\to\wp(X)$.
Define an assignment $\eta^\alpha:\Var\to\wp(X)$ in the same way as for $h^\alpha$:
let $\eta^\alpha(p)=\eta(p)$ for all atoms $p\neq q$ and all $\alpha$,
 and
 \begin{itemize}
\item $\eta^0(q)=\emptyset$,

\item $\eta^{\alpha+1}(q)=\{x\in X:(X,\eta^\alpha),x\models\varphi\}$,

\item $\eta^\delta(q)=\bigcup_{\alpha<\delta}\eta^\alpha(q)$ for limit ordinals
 $\delta$.

\end{itemize}
\claim\  $\eta^\alpha(q)=\rho^{-1}(h^\alpha(q))$ for each ordinal $\alpha$.

\pfclaim By induction on $\alpha$. 
For $\alpha=0$ this is saying that $\rho^{-1}(\emptyset)=\emptyset$, which is true.
Assume the result for $\alpha$ inductively.
So $\eta^\alpha=\rho^{-1}\circ h^\alpha$.  We now obtain
\[
\begin{array}{rcll}
\eta^{\alpha+1}(q)
&=&\{x\in X:(X,\eta^\alpha),x\models\varphi\}&\mbox{ by definition of }\eta^{\alpha+1}
\\
&=&\{x\in X:(X,\rho^{-1}\circ h^\alpha),x\models\varphi\}&\mbox{ since }\eta^\alpha=\rho^{-1}\circ h^\alpha
\\
&=&\{x\in X:(W,R, h^\alpha),\rho(x)\models\varphi\}&\mbox{ by inductive hypothesis }\eqref{e:rho pres phi}
\\
&=&\{x\in X:\rho(x)\in h^{\alpha+1}(q)\}&\mbox{ by definition of }h^{\alpha+1}
\\
&=&\rho^{-1}(h^{\alpha+1}(q)).
\end{array}
\]
For limit $\delta$ we have
\[
\rho^{-1}(h^\delta(q))
=\rho^{-1}(\bigcup_{\alpha<\delta}h^\alpha(q))
=\bigcup_{\alpha<\delta}\rho^{-1}(h^\alpha(q))
=_{IH}\bigcup_{\alpha<\delta}
\eta^\alpha(q)=\eta^\delta(q).
\]
This completes the induction on $\alpha$, and proves the claim.

\medskip
By semantics of $\mu$, we have
$(X,\eta),x\models\mu q\varphi$ iff
$x\in\bigcup_{\alpha\in \rm On}\eta^\alpha(q)$,
iff $x\in \bigcup_\alpha \rho^{-1}(h^\alpha(q))$ by the claim,
iff $\rho(x)\in\bigcup_\alpha h^\alpha(q)$,
iff $(W,R,h),\rho(x)\models\mu q\varphi$.
This completes the induction and proves the proposition. 
\end{proof}

\subsection{Basic \rep s}

Certain very primitive \rep s  called \emph{basic \rep s} will play an important role later, because they
can easily be extended to more interesting \rep s.

\begin{definition}
Let $S,U$ be  open subspaces of $X$, with $S\subseteq U$,
and let $\sigma:S\to W$ be a \rep\ of $\c F$ over $S$.
We say that $\sigma$ is \emph{$U$-basic} if
for every $x\in U$ and $w,v\in W$, if $(X,\sigma^{-1}),x\models\di w\wedge\di v$
then $Rwv$.
\end{definition}
Note that we use $\di$ and not $\did$ here.

\begin{remark}\label{rmk:basic reps}\rm
In the setting of this definition:
\begin{enumerate}
\item Vacuously, if $\sigma$ is empty then it is $U$-basic.
\item 
More generally, but equally trivially,
 if $\rng\sigma$ is contained in 
a nondegenerate cluster $C$ in $\c F$,
then $\sigma$ is $U$-basic.  For,
$(X,\sigma^{-1}),x\models\di w\wedge\di v$
implies that $w,v\in\rng\sigma\subseteq C$, and so $Rwv$ as $C$ is a nondegenerate cluster.
\end{enumerate}
We remark (but will not formally use) that $\sigma$
is $U$-basic iff
$\rng\sigma$ is a (possibly empty) union of $R$-maximal clusters in $\c F$
whose preimages under $\sigma$ have pairwise disjoint closures within $U$.
Moreover,  each
such preimage is a \ro\ subset of $S$.
\end{remark}

\subsection{Full \rep s}

In induction proofs, one often needs a stronger inductive hypothesis
than formally required for the final result. 
This will be the case in proposition~\ref{prop:di rep} below, and 
the  notion of \emph{$T$-full} \rep\
will be used to formulate it.

\begin{definition}\label{def:Tfull rep}
Let  $T\subseteq U$ be open subspaces of~$X$. 
A \rep\ $\rho:U\to W$ of $\c F$ over $U$
is said to be \emph{$T$-full} if:
\begin{enumerate}
\item for every  $x\in \cl(T)\setminus U$
and $w\in W$, we have
$(X,\rho^{-1}),x\models\did w$,

\item if $T$ is non-empty then $\rho:U\to W$ is surjective.
\end{enumerate}
\end{definition}
\noindent Every \rep\ is vacuously $\emptyset$-full.

\subsection{Full representability}\label{ss:full rby}

\begin{definition}\label{def:fully rb did case}
We say that $\c F$ is \emph{fully \rb\ (over $X$)}
if whenever
\begin{enumerate}
\item $U\subseteq X$ is open,

\item $S$ is a \ro\ subset of $U$,

\item $\sigma:S\to W$ is a $U$-basic \rep\ of $\c F$ over $S$,

\item $T=U\setminus\cl S$,
\end{enumerate}
then $\sigma$ extends to a  $T$-full \rep\ $\rho:U\to W$ of $\c F$ over $U$.
\end{definition}

\noindent
Notice that in the \ba\ $RO(U)$ of \ro\ subsets of $U$,
we have $T=-S$, so $\{S,T\}$ is a partition of 1.
That is, $S,T\in RO(U)$, $S\cdot T=0$, and $S+T=1$.

In proposition~\ref{prop:di rep} below, we will fulfil our main aim, to 
prove (surjective) representability of every finite connected locally connected $\axK\axD4$-frame.
We are going to do it by induction on the size of the frame;
we appear to need a stronger inductive hypothesis, namely full representability,
than is needed for the conclusion;
$T$-fullness and extending $\sigma$ are mainly to do with this,
but
the $\sigma$ part
is also helpful in the proof of strong completeness in theorem~\ref{thm:strong did} later.
Note that if $\c F$ is fully \rb\ over $X$, and $X\neq\emptyset$,
then by taking $U=X$ and $S=\sigma=\emptyset$, we see that there exists a
surjective \rep\
of $\c F$ over $X$. So we do obtain our desired conclusion from
the stronger hypothesis of full representability.

\subsection{Main proposition}\label{ss:main rep prop}

The following proposition has relatives 
in the literature: see, e.g., \cite[theorem 3.7]{McKiT44}, \cite[proposition 22]{Sheh:d90}, \cite[lemma 4.4]{LucBry11}, and
\cite[lemma 16]{KudShe14}.
It actually holds for any dense-in-itself
topological space $X$ for which theorem~\ref{thm:Tarski} and
corollary~\ref{cor:Tarski++opens} can be proved.

\begin{proposition}\label{prop:di rep}
Suppose that $X$ is a \nice\ space.
Then every finite connected locally connected $\axK\axD4$ frame $\c F=(W,R)$ is fully \rb\
over $X$.
\end{proposition}

\begin{proof}
The proof is by induction on the number of worlds in $\c F$.
Let $\c F=(W,R)$ be a finite connected locally connected $\axK\axD4$ frame, and assume the result inductively for all smaller frames.
Note that $R$ is transitive.
Recall that we write 
\begin{itemize}
\item $R^\circ=\{(w,v)\in W^2: Rwv\wedge Rvw\},$

\item $R^\bullet=\{(w,v)\in W^2:Rwv\wedge\neg Rvw\},
$
\end{itemize}
and for $w\in W$, 
\begin{itemize}
\item  $\c F(w)$ for the subframe $(R(w),R\restriction R(w))$ of $\c F$ with domain $R(w)$,

\item  $\c F^*(w)$
for the subframe $(R^*(w),R\restriction R^*(w))=(R(w)\cup\{w\}, R\restriction R(w)\cup\{w\})$
of $\c F$ generated by $w$.
\end{itemize}

Let $U\subseteq X$ be open, let $S$ be a \ro\ subset of $U$,
and let $\sigma:S\to W$ be a $U$-basic \rep\ of $\c F$ over $S$.
Write 
\[
T=U\setminus\cl S.
\]
We need to extend $\sigma$ to a  $T$-full \rep\ $\rho:U\to W$ of $\c F$ over $U$.

If $T=\emptyset$, then $U\subseteq\cl S$, so $S=\int(U\cap\cl S)=\int U=U$.
Thus, $\sigma:S\to W$ is already a \rep\ of $\c F$ over $U$,
and it is vacuously $T$-full. So we can take $\rho=\sigma$.
We are done. 

So assume from now on that $T\neq\emptyset$.
There are three cases.

\paragraph{\boldmath Case 1: $\c F=\c F^*(w_0)$ for some reflexive ${w_0}\in W$}
Choose such a ${w_0}$ (it may not be unique).
Then $R(w_0)=W$ and ${w_0}\in R^\circ({w_0})$ since $w_0$ is reflexive.
So $R^\circ(w_0)\neq\emptyset$.
Since $T$ is clearly a non-empty open set,
we can use theorem~\ref{thm:Tarski}(\ref{Tarski thm 2}) to partition $T$
into non-empty open sets $G_{v^\bullet}$ (${v^\bullet}\in R^\bullet({w_0})$)
and other non-empty sets $B_{v^\circ}$ $({v^\circ}\in R^\circ({w_0})$)
such that for each ${v^\bullet}\in R^\bullet({w_0})$
and ${v^\circ}\in R^\circ({w_0})$ we have
\begin{equation}\label{e:tarski cond1}
\cl (G_{v^\bullet})\setminus G_{v^\bullet}=\did B_{v^\circ}=\cl( T)\setminus\bigcup_{v\in R^\bullet({w_0})}G_{v}=D,\mbox{ say.}
\end{equation}

For each ${v^\bullet}\in R^\bullet({w_0})$, the frame
$\c F^*(v^\bullet)$  is connected (as it is rooted) and locally connected $\axK\axD4$
(as it is a generated subframe of $\c F$).
Since ${w_0}$ is a world of  $\c F$ but not of $\c F^*(v^\bullet)$, the frame $\c F^*(v^\bullet)$ is smaller than $\c F$.
By the inductive hypothesis, $\c F^*(v^\bullet)$ is fully \rb\ over $X$.
So, taking the  \ro\ subset `$S$' of $G_{v^\bullet}$ to be $\emptyset$ and `$T$'
to be $G_{v^\bullet}\setminus\cl\emptyset=G_{v^\bullet}$,
we can find a $G_{v^\bullet}$-full \rep\ $\rho_{v^\bullet}$ of $\c F^*(v^\bullet)$ over~$G_{v^\bullet}$.

Define $\rho:U\to W$ by:
\[
\rho(x)=
\begin{cases}
\rho_{v^\bullet}(x),&\mbox{if }x\in G_{v^\bullet}\mbox{ for some (unique) }{v^\bullet}\in R^\bullet({w_0}),
\\
{v^\circ},&\mbox{if }x\in B_{v^\circ}\mbox{ for some (unique) }{v^\circ}\in R^\circ({w_0}),
\\
\sigma(x),&\mbox{if }x\in S,
\\
{w_0},&\mbox{otherwise,}
\end{cases}
\]
for each $x\in U$.
The map $\rho$ is well defined because the $G_{v^\bullet}$, the $B_{v^\circ}$, and $S$ are pairwise disjoint,
and plainly it is total and extends $\sigma$.

We aim to show that $\rho$ is a $T$-full \rep\ of $\c F$ over $U$.
The following claim will help.

\claim Let $x\in D$ (see~\eqref{e:tarski cond1}).
Then $(X,\rho^{-1}),x\models\did w$ for every $w\in W$.

\pfclaim
Let $x\in D$ and $w\in W$ be given.
There are two cases.
The first is when $w\in R^\bullet({w_0})$.
Now~\eqref{e:tarski cond1} gives $x\in\cl G_w\setminus G_w$.
As $\rho_w$ is a $G_w$-full \rep\ of $\c F^*(w)$, a frame of which $w$ is a world, we have $(X,\rho_w^{-1}),x\models\did w$,
and hence $(X,\rho^{-1}),x\models\did w$ 
(since $\rho_w\subseteq\rho$).

The second case is when $w\notin R^\bullet({w_0})$. 
Since $w\in W=R(w_0)=R^\bullet({w_0})\cup R^\circ({w_0})$,
we have  $w\in R^\circ({w_0})$.
By~\eqref{e:tarski cond1}, $x\in\did B_w$ (since $x\in D$).
Since $\rho\restriction B_w$ has constant value $w$,
we obtain again that $(X,\rho^{-1}),x\models\did w$.
This proves the claim.

\medskip

We now check that $\rho$ is a  \rep\ of $\c F$ over $U$.
Let $x\in U$ and $w\in W$.
We require
$(X,\rho^{-1}),x\models\did w$ iff $R(\rho(x),w)$.
There are four cases.
\begin{enumerate}
\item Suppose that $x\in G_{v^\bullet}$ for some ${v^\bullet}\in R^\bullet({w_0})$.
Since $G_{v^\bullet}$ is open and $\rho\restriction G_{v^\bullet}=\rho_{v^\bullet}$, a \rep\ over $G_{v^\bullet}$
of the generated subframe $\c F^*(v^\bullet)$ of $\c F$,
lemma~\ref{lem:reps simple} yields
$(X,\rho^{-1}),x\models\did w$ iff  $R(\rho(x),w)$.

\item Suppose that  $x\in B_{v^\circ}$ for some ${v^\circ}\in R^\circ({w_0})$.
Then $\rho(x)={v^\circ}$.
As ${v^\circ}\in R^\circ({w_0})$, we have $R{v^\circ}{w_0}$.
By transitivity of $R$, we have $R(\rho(x),w)$ for every $w\in W$.
So we need to prove that $(X,\rho^{-1}),x\models\did w$ for every $w\in W$.
But $x\in B_{v^\circ}\subseteq D$ by
definition of $D$~\eqref{e:tarski cond1}, so this follows from the claim.

\item If $x\in S$, then since $S$ is open and $\rho\restriction S=\sigma$, a 
\rep\ of $\c F$ over $S$, the result  follows from lemma~\ref{lem:reps simple} again.

\item Suppose finally that $x\in U\setminus(S\cup T)$.
Then $\rho(x)={w_0}$.
Since $R({w_0},w)$ for all $w\in W$, we require that $(X,\rho^{-1}),x\models\did w$ for all $w\in W$ as well.

Now as $S$ is a \ro\ subset of $U$,
by lemma~\ref{lem:ro1} we obtain
$U\setminus S= \cl T$.
Hence, $x\in \cl T\setminus T\subseteq D$ by~\eqref{e:tarski cond1}.
As in case 2, the claim now gives
$(X,\rho^{-1}),x\models\did w$ for all $w\in W$. 
\end{enumerate}

So $\rho$ is indeed a \rep\ of $\c F$ over $U$.
We check that it is $T$-full.
First let $x\in\cl T\setminus U$. 
Then  $x\in \cl T\setminus T\subseteq D$ by~\eqref{e:tarski cond1}.
By the claim, $(X,\rho^{-1}),x\models\did w$ for every $w\in W$, as required.

We also need that $\rho$ is surjective.
Take any $x\in B_{w_0}$.
Then $x\in D$ by definition of $D$ in \eqref{e:tarski cond1}.
By the claim, $(X,\rho^{-1}),x\models\did w$,
and so $\rho^{-1}(w)\neq\emptyset$,
 for every $w\in W$.
Hence, $\rho$ is surjective.

\paragraph{\boldmath Case 2: $\c F=\c F^*(w_0)$ for some irreflexive ${w_0}\in W$}
Choose such a ${w_0}$ (it is unique this time).
Then $W$ is the disjoint union of $\{w_0\}$ and $R(w_0)$.
Using theorem~\ref{thm:Tarski}(\ref{Tarski thm 1}),
select non-empty $I\subseteq T$ with 
\begin{equation}\label{e:I set}
\did I=\cl T\setminus U.
\end{equation}
Write 
\[
\begin{array}{rcl}
U'&=&U\setminus I,
\\
T'&=&T\setminus I.
\end{array}
\]
We aim to use the inductive hypothesis on 
these sets and $\sigma:S\to\c F({w_0})$, so we 
check the necessary conditions.

\Claim1
$U'$ is open, $S$ is a \ro\ subset of $U'$, and $T'=U'\setminus\cl S$.

\pfclaim 
First, $U'$ is open.
For, by lemma~\ref{lem:did cl} and~\eqref{e:I set}, 
\[
U\setminus\cl I=U\setminus(I\cup\did I)=U\setminus(I\cup(\cl(T)\setminus U))=U\setminus I=U',
\]
and the left-hand side is open.

We are given that $S$ is a \ro\ subset of $U$.
Since $S\subseteq U$ and  $I\subseteq T=U\setminus\cl S$,
we have $S\subseteq U\setminus I=U'$.
By lemma~\ref{lem:ro1}(3), $S$ is a \ro\ subset of $U'$.

Finally,
$U'\setminus\cl S=(U\setminus I)\setminus\cl S=(U\setminus\cl S)\setminus I=T\setminus I=T'$.
This proves the claim.  \medskip

\Claim2  $\sigma$ is a $U'$-basic \rep\ of $\c F({w_0})$ over $S$.

\pfclaim
First we show that $\sigma:S\to R({w_0})$.
We know that $\sigma:S\to W=\{{w_0}\}\cup R({w_0})$.
Assume for contradiction that there is some $x\in S$ with $\sigma(x)={w_0}$.
Then plainly, $x\in U$ and $(X,\sigma^{-1}),x\models\di{w_0}$.
As $\sigma$ is a $U$-basic
\rep\ of $\c F$ over $S$, we obtain $R{w_0}{w_0}$, contradicting the choice of $w_0$ as irreflexive.
So indeed, $\rng\sigma\subseteq W\setminus\{{w_0}\}= R({w_0})$.
Since $\sigma$ is a \rep\ of $\c F$ over $S$, by
lemma~\ref{lem:reps simple}  it is also a \rep\  (over $S$) of 
the generated subframe $\c F({w_0})$ of $\c F$.
It is trivially $U'$-basic, since if $x\in U'$, $w,v\in R({w_0})$, and $(X,\sigma^{-1}),x\models\di w\wedge\di v$,
then $x\in U$ and $w,v\in W$ as well, so $Rwv$ since $\sigma$ is $U$-basic. This proves the claim.

\medskip

In summary, $U'$ is open,  $S$ is a \ro\ subset of $U'$,
$\sigma$ is a $U'$-basic \rep\ of $\c F({w_0})$  over $S$,
and $T'=U'\setminus\cl S$.

Now $\c F({w_0})$ is smaller than $\c F$
(since ${w_0}\notin R({w_0})$),
connected (since $\c F$ is locally connected), and locally connected $\axK\axD4$
(since it is a generated subframe of~$\c F$).
By the inductive hypothesis, $\c F({w_0})$ is fully \rb\ over $X$.

So $\sigma$ extends to a 
$T'$-full \rep\ $\rho':U'\to R({w_0})$ of $\c F({w_0})$ over $U'$.
By $T'$-fullness,
\begin{equation}\label{e:2 cases in 1}
(X,\rho'^{-1}),x\models\did v\mbox{ for every } v\in R({w_0}) \mbox{ and }
x\in\cl T'\setminus U'.
\end{equation}
We extend $\rho'$ to a map $\rho:U\to W$ by defining
\[
\rho(x)=
\begin{cases}
\rho'(x),&\mbox{if }x\in U',
\\
w_0,&\mbox{if }x\in I,
\end{cases}
\]
for $x\in U$.
This is plainly well defined and total.
Since $\rho$ extends $\rho'$, it also extends $\sigma$.
We will show that $\rho$ is a $T$-full \rep\ of $\c F$ over $U$.
To do it, we need
another claim.

\Claim3
$\cl T\setminus U\subseteq\cl I\subseteq\cl T'\setminus U'$.

\pfclaim 
By \eqref{e:I set} and lemma~\ref{lem:did cl}, we have
 $\cl T\setminus U=\did I\subseteq\cl I$.
 
Using openness of $T=T'\cup I$, the assumption that $X$ is dense in itself, and
lemma~\ref{lem:did cl}(\ref{didcl part5},\ref{didcl additive}),
we have
$I\subseteq T\subseteq\cl T=\did T=\did T'\cup\did I$.
But by \eqref{e:I set},  $I\cap\did I\subseteq U\cap\cl T\setminus U=\emptyset$.
So in fact, $I\subseteq\did T'\subseteq\cl T'$.
Hence, $\cl I\subseteq\cl T'$.
Since $I\cap U'=\emptyset$ and  $U'$ is open (claim~1), we have $\cl I\cap U'=\emptyset$.
So $\cl I\subseteq\cl T'\setminus U'$,
proving the claim.

\medskip

\Claim4 $\rho$ is a \rep\ of $\c F$ over $U$.

\pfclaim Let $x\in U$.  We require
$(X,\rho^{-1}),x\models\did w$ iff $R(\rho(x),w)$,
for each $w\in W$.

There are two cases here.
The first is when $x\in I$.
Then $\rho(x)={w_0}$, so we require first
that $(X,\rho^{-1}),x\models\did w$ for each $w\in R({w_0})$.
So pick any $w\in R({w_0})$.
By claim~3, $x\in I\subseteq\cl I\subseteq\cl T'\setminus U'$,
so by~\eqref{e:2 cases in 1},
$(X,\rho'^{-1}),x\models\did w$.
As $\rho'\subseteq\rho$,  the result follows.

We also require that
$(X,\rho^{-1}),x\not\models\did w$ for each $w\in W\setminus R({w_0})$ ---
that is,
$(X,\rho^{-1}),x\not\models\did {w_0}$.
But as $x\in U$, we have $x\notin\cl T\setminus U=\did I$ by~\eqref{e:I set}.
Since $\rho^{-1}({w_0})=I$, we do indeed have $(X,\rho^{-1}),x\not\models\did {w_0}$.

The second case is when $x\notin I$.
In this case, $x\in U'$, an open set, and $\rho\restriction U'=\rho'$,
 a \rep\  over $U'$ of the generated subframe $\c F({w_0})$ of $\c F$.
By lemma~\ref{lem:reps simple}, $(X,\rho^{-1}),x\models\did w$ iff $R(\rho(x),w)$
for every $w\in W$,
as required. The claim is proved.

\Claim5 $\rho$ is $T$-full.

\pfclaim
Let $x\in\cl T\setminus U$ and $w\in W$.
We require $(X,\rho^{-1}),x\models\did w$.

Suppose first that $w={w_0}$.
By~\eqref{e:I set}, $x\in\did I$. Since
$I=\rho^{-1}({w_0})$, we obtain $(X,\rho^{-1}),x\models\did {w_0}$.
Suppose instead that $w\in R({w_0})$.
By claim~3, $x\in\cl T'\setminus U'$.
So by \eqref{e:2 cases in 1},
$(X,\rho'^{-1}),x\models\did w$.
As $\rho'\subseteq\rho$,
we obtain $(X,\rho^{-1}),x\models\did w$
as required. 

We must also show that $\rho(U)=W$.
Well, $I\neq\emptyset$, and it follows from claim~3 that  $T'\neq\emptyset$ as well.
As $\rho'$ is $T'$-full, $\rho'(U')=R({w_0})$.
So 
\[
\rho(U)=\rho'(U')\cup\rho(I)=R({w_0})\cup\{{w_0}\}=W,
\] as required.
This proves the claim and completes case 2 of 
proposition~\ref{prop:di rep}. Only case 3 remains, but this is the hardest case.

\paragraph{Case 3: otherwise}
As $\c F$ is finite and connected,
we can choose 
worlds 
$a_0,b_0,a_1,b_1,\ldots,\allowbreak b_{n-1},\allowbreak a_{n}\in W$,
 for some least possible $n<\omega$,
such that $Ra_ib_i$ and $Ra_{i+1}b_i$ for each $i<n$,
each $b_i$ is $R$-maximal (so that $R^\bullet(b_i)=\emptyset$),
 and $W=\bigcup_{i\leq n}R^*(a_i)$.
By the case assumption,  $n\geq1$.

Write $\c F^*({a_0})$ as  $\c F_0=(W_0,R_0)$, say.
Let $\c F_1=(W_1,R_1)$ be the smallest generated subframe of $\c F$ 
containing $a_1,\ldots,a_{n}$.
We have $W_{0}\cup W_1=W$ and $b_0\in W_0\cap W_1$.
Plainly,  $\c F_0$ and $\c F_1$
are connected generated subframes of $\c F$.
Therefore, they are locally connected $\axK\axD4$ frames.
By minimality of $n$, they are proper subframes of $\c F$.
By the inductive hypothesis, $\c F_0$ and $\c F_1$ are  fully \rb\ over $X$.
Our plan is to combine suitable \rep s of them to give a \rep\ of $\c F$ over $U$.

Recall that $S$ is a \ro\ subset of $U$ and $\sigma:S\to W$ is a $U$-basic \rep\ of $\c F$. 
We use $W_0,W_1$ to split $S$ (and, later, $\sigma$) in two.
Let
\[
\begin{array}{rcl}
S_0&=&\sigma^{-1}(W_0) =\{x\in S:\sigma(x)\in W_0\},
\\
S_1&=&S\setminus S_0.
\end{array}
\]
So $\sigma(S_0)\subseteq W_0$ and
$\sigma(S_1)\subseteq W\setminus W_0\subseteq W_1$.
Also, $S_0=S\setminus S_1$.

\Claim1
 $S_0$ and $S_1$ are \ro\ subsets of $U$, and
$U\cap\cl(S_0)\cap\cl(S_1)=\emptyset$.

\pfclaim We prove the last point first.
Suppose for contradiction that there is some $x\in U\cap\cl(S_0)\cap\cl(S_1)$.
As $x\in\cl S_0$, we have $(X,\sigma^{-1}),x\models\di\bigvee_{w\in W_0}w$.
As $\di$ is additive, it follows that there is some $w_0\in W_0$ such that
$(X,\sigma^{-1}),x\models\di w_0$.
Similarly, as $x\in\cl S_1$ and $\sigma(S_1)\subseteq W\setminus W_0$, 
there is some $w_1\in W\setminus W_0$ with
$(X,\sigma^{-1}),x\models\di w_1$.
As $\sigma$ is a $U$-basic \rep, we obtain $Rw_0w_1$.
Since $\c F_0$ is a generated subframe of $\c F$,
this implies that $w_1\in W_0$, a contradiction.
So $U\cap\cl(S_0)\cap\cl(S_1)=\emptyset$ as required.

Now let $i<2$. We show that $S_i$ is \ro\ in $U$.
First
note that $S_i$ is open.
To see this, observe that
\[
\begin{array}{rcll}
S_i&\subseteq& S\cap U\cap\cl S_i
&\mbox{as }S_i\subseteq S\subseteq U\mbox{ by definition and assumption}
\\
&\subseteq& S\cap U\setminus\cl S_{1-i}
&\mbox{by the first part}
\\
&=& S\setminus\cl S_{1-i}
&\mbox{as }S\subseteq U\mbox{ by assumption}
\\
&\subseteq&
S\setminus S_{1-i}
&\mbox{as }S_{1-i}\subseteq\cl S_{1-i}
\\
&=&S_i
&\mbox{by definition of }S_i.
\end{array}
\]
Hence, $S_i=S\setminus\cl S_{1-i}$,  an open set.

It follows that $\cl(S_i)\cap S_{1-i}=\emptyset$,
so $S_i\subseteq S\cap\cl S_i\subseteq S\setminus S_{1-i}=S_i$.
Thus, $S\cap\cl S_i=S_i$,
and so $\int(S\cap\cl S_i)=\int S_i=S_i$ as $S_i$ is open.
So $S_i$ is \ro\ in $S$, and as $S$ is \ro\ in $U$,
lemma~\ref{lem:ro1}(\ref{lem:ro:part4}) yields that $S_i$ is \ro\ in $U$.
The claim is proved.

\medskip

The claim and the assumption at the outset that
$T\neq\emptyset$ are more than enough to apply corollary~\ref{cor:Tarski++opens},
to obtain open subsets $U_i,T_i$ of $U$,
for $i=0,1$, satisfying the following conditions:
\begin{enumerate}
\renewcommand{\theenumi}{C\arabic{enumi}}
\renewcommand{\labelenumi}{C\arabic{enumi}.}

\item\label{ROsets cond1} $U\cap\cl U_0\cap\cl U_1=\emptyset$,

\item\label{ROsets cond2} $U\cap\cl S_i\subseteq U_i$,

\item\label{ROsets cond3} $T_i=U_i\setminus\cl S_i\neq\emptyset$,

\item\label{ROsets cond4} $\cl(T)\setminus U\subseteq\cl(T_i)$,

\item\label{point:Ui ro}
$U_i$ is a \ro\ subset of $U$.
\end{enumerate}
We now work in the \ba\ $RO(U)$ of \ro\ subsets of $U$.
By~\ref{point:Ui ro}, we have $U_0,U_1\in RO(U)$.
We define further elements of $RO(U)$:
\begin{enumerate}
\renewcommand{\theenumi}{C\arabic{enumi}}
\renewcommand{\labelenumi}{C\arabic{enumi}.}
\setcounter{enumi}{5}
\item $M=-(U_0+U_1)$,

\item  $V_i=M+U_i$ for $i=0,1$.
\end{enumerate}
The main property of these sets is as follows.

\Claim2 $\{M,S_0,S_1,T_0,T_1\}$ is a partition of $1$ in the \ba\ $RO(U)$.
That is, the five elements are pairwise disjoint \ro\ subsets of $U$, with
\begin{equation}\label{partition of U}
U=
\setbox0\hbox{$\overbrace{S_0+T_0\strut}^{U_0}+\overbrace{M+S_1+T_1\strut}^{V_1}$}\wd0=0pt\box0\hskip-2pt
\underbrace{\phantom{\strut S_0+T_0+M}}_{V_0}\phantom{+}\underbrace{\phantom{S_1+T_1\strut}}_{U_1}.
\end{equation}

\pfclaim
Let $i<2$.
By claim 1 and condition~\ref{point:Ui ro} above, $S_i,U_i\in RO(U)$.
By this and condition~\ref{ROsets cond3}, 
\begin{equation}\label{e:char Ti}
T_i=U_i\setminus\cl S_i=U_i\cap U\setminus\cl S_i
=U_i\cdot-S_i\in RO(U).
\end{equation}
So $S_i\cdot T_i=\emptyset$
and, since $S_i\subseteq U_i$ by condition~\ref{ROsets cond2},
also
$U_i=U_i\cdot S_i+U_i\cdot-S_i=S_i+T_i$.
Condition~\ref{ROsets cond1} above gives $U_0\cdot U_1=\emptyset$.
By definition, $M=-(U_0+U_1)$, so $M\in RO(U)$ and $M$ is disjoint from
$T_i,S_i$. Also, $U=U_0+U_1+M=S_0+T_0+S_1+T_1+M$.
It is now plain that $M+S_i+T_i=M+U_i=V_i$.
This proves the claim.

\medskip

We aim to apply the inductive hypothesis 
to $V_i,M+S_i,T_i,\c F_i$, for each $i=0,1$.
We will need a $V_i$-basic \rep\ of $\c F_i$ over $M+S_i$,
and the next claim helps us get one.

\Claim3
For each $i<2$ we have $U\cap \cl M\cap \cl S_i=\emptyset$, and
$M+S_i=M\cup S_i$ in $RO(U)$.

\pfclaim By definition, $M=-(U_0+U_1)=U\setminus\cl(U_0+U_1)\subseteq U\setminus U_i$.
Since $U_i$ is open,
$\cl M\cap U_i=\emptyset$.
But $U\cap\cl S_i\subseteq U_i$ by condition~\ref{ROsets cond2} above,
so $U\cap \cl M\cap \cl S_i=\emptyset$.
By lemma~\ref{lem:ro1}, $M+S_i=M\cup S_i$.
This proves the claim.

\medskip
\noindent
So all we need is to find suitable \rep s over $M$ and $S_i$ and take their union.

Clearly, $\c F^*(b_0)$ is a
subframe of $\c F_0$, and so a \emph{proper} subframe of $\c F$.
It is obviously connected (since rooted), and a generated subframe of $\c F$, so a locally connected $\axK\axD4$ frame.
By the inductive hypothesis, it is fully \rb\ over $X$.
So we can find an ($M$-full) \rep\  $\mu:M\to R(b_0)$ of $\c F^*(b_0)$ over $M$.

For each $i<2$ let
\[
\sigma_i=(\sigma\restriction S_i)
\;:\;S_i\to W_i.
\]

\Claim4 For each $i<2$, $\mu\cup\sigma_i: M\cup S_i\to W_i$
is a well defined $V_i$-basic \rep\ of $\c F_i$ over $M\cup S_i$.

\pfclaim
Since $\c F^*(b_0)$ is a generated subframe  of $\c F_i$,
it follows from lemma~\ref{lem:reps simple}(\ref{lem:simple:1}) that
$\mu$ is a \rep\ of $\c F_i$ over $M$.
Similarly, $\sigma_i$ is a \rep\ of $\c F_i$ over $S_i$.
Since  $M$ and $S_i$ are disjoint open sets,
$\mu\cup\sigma_i:M\cup S_i\to W_i$ is well defined and,
by lemma~\ref{lem:reps simple}(\ref{lem:simple:2}),
a \rep\ of $\c F_i$ over $M\cup S_i$.

To prove that it is $V_i$-basic, let
$x\in V_i$ and $v,w\in W_i$ be given,
and suppose that $(X,(\mu\cup\sigma_i)^{-1}),x\models\di w\wedge\di v$.
We require $Rwv$.

Plainly, $x\in\cl(M\cup S_i)=\cl M\cup\cl S_i$, and $x\in V_i\subseteq U$.
But  $U\cap\cl M\cap\cl S_i=\emptyset$ by claim~3.
So there are two possibilities.

The first one is that $x\notin\cl M$.
In this case, we must have $(X,\sigma_i^{-1}),x\models\di w\wedge\di v$.
As $\sigma_i\subseteq\sigma$, we  also have $(X,\sigma^{-1}),x\models\di w\wedge\di v$.
As $\sigma$ is $U$-basic, we obtain $Rwv$.

The other possibility is that  $x\notin\cl S_i$.
So $(X,\mu^{-1}),x\models\di w\wedge\di v$.
Since $\mu$ is a \rep\ of $\c F^*(b_0)$, we have $w,v\in R(b_0)$.
But $b_0$ is $R$-maximal, so $R^\bullet(b_0)=\emptyset$.
Hence, $w\in R^\circ(b_0)$, so $Rwb_0$, and since $Rb_0v$, we deduce $Rwv$ by transitivity. (Essentially we are using that $\c F^*(b_0)$ is a non-degenerate cluster.)
This proves the claim.

\medskip

In summary, for each $i<2$ we have:
\begin{itemize}
\item $V_i$ is  open (by claim 2)

\item $M+S_i,V_i\in RO(U)$ and $M+S_i\subseteq V_i$, so  by lemma~\ref{lem:ro1},
$M+S_i$ is a \ro\ subset of $V_i$

\item working in $RO(U)$, we have
$V_i=(M+S_i)+T_i$ and $(M_i+S_i)\cdot T_i=\emptyset$ by claim~2.
So $T_i=V_i\cdot-(M+S_i)=V_i\cap U\setminus\cl(M+S_i)=V_i\setminus\cl(M+S_i)$.

\item $M+S_i=M\cup S_i$ (by claim 3), and  $\mu\cup\sigma_i:M\cup S_i\to W_i$ is a $V_i$-basic \rep\ of $\c F_i$
over $M+S_i$ (by claim  4)
\end{itemize}
So for each $i<2$, recalling that  $\c F_i$ is fully \rb,
we see that 
$\mu\cup\sigma_i:M\cup S_i\to W_i$ extends to a $T_i$-full
 \rep\ $\rho_i:V_i\to W_i$ of $\c F_i$ over $V_i$.
We have
 \begin{equation}\label{e:T_ifull gives}
(X,\rho_i^{-1}),x\models\did w\quad\mbox{for every }
w\in W_i\mbox{ and }x\in \cl T_i\setminus V_i.
\end{equation}

Finally define
\begin{equation}
\rho=\rho_0\cup\rho_1:U\to W.
\end{equation}
We check first that $\rho$ is well defined and total.
Working in $RO(U)$ again, we have $\dom\rho_0\cap\dom\rho_1=V_0\cap V_1=V_0\cdot V_1=M$ 
by \eqref{partition of U}.
But $\rho_0\restriction M=\mu=\rho_1\restriction M$.
So $\rho$ is well defined.
Also, $V_i=-U_{1-i}=U\setminus\cl U_{1-i}$ (for $i=0,1$) by \eqref{partition of U},
 and $U\cap\cl U_0\cap\cl U_1=\emptyset$
by condition~\ref{ROsets cond1} above,
so
\begin{equation}\label{D0+D1=U}
\dom\rho=V_0\cup V_1=(U\setminus\cl U_1)\cup(U\setminus\cl U_0)=U\setminus(\cl U_1\cap\cl U_0)=U.
\end{equation}
Hence, $\rho$ is total.
Plainly, $\rho$ extends $\sigma$,
since $\rho=\rho_0\cup\rho_1
\supseteq(\mu\cup\sigma_0)\cup(\mu\cup\sigma_1)=\mu\cup\sigma$.

\Claim5 $\rho$ is a \rep\ of
$\c F$ over $U$.

\pfclaim 
Let $i<2$.
Then $\rho\restriction V_i=\rho_i$, a \rep\ of $\c F_{i}$ over $V_i$.
By lemma~\ref{lem:reps simple}(\ref{lem:simple:1}),
this is also a \rep\ of $\c F$ over $V_i$,
which is an open set by claim 2.
By~\eqref{D0+D1=U},  $U=V_0\cup V_1$, so by
lemma~\ref{lem:reps simple}(\ref{lem:simple:2}), 
$\rho$ is a \rep\ of $\c F$ over $U$, proving the claim.

\medskip

\Claim6 $\rho$ is $T$-full.

\pfclaim
Let $x\in\cl T\setminus U$.
We require $(X,\rho^{-1}),x\models\did w$ for every $w\in W$.

For each $i<2$, as $\cl T\setminus U\subseteq\cl T_i$ by condition~\ref{ROsets cond4} above,
and $x\notin U\supseteq V_i$,
we have
$x\in\cl T_i\setminus V_i$.
Since $\rho_i\subseteq\rho$,
it follows from~\eqref{e:T_ifull gives} that
$(X,\rho^{-1}),x\models\did w$ for every $w\in W_i$.
This holds for each $i=0,1$.
Since $W_0\cup W_1=W$, 
we have $(X,\rho^{-1}),x\models\did w$ for every $w\in W$.

Finally, we show that $\rho(U)=W$.
Since 
each $\rho_i$ is a $T_i$-full \rep\ of $\c F_i$ over $V_i$,
and $T_i\neq\emptyset$ by condition~\ref{ROsets cond3},
by~\eqref{D0+D1=U} we obtain $\rho(U)=\rho(V_0)\cup\rho(V_1)=\rho_0(V_0)\cup\rho_1(V_1)=W_0\cup W_1=W$.
This proves the claim, and with it, proposition~\ref{prop:di rep}.
\end{proof}

\begin{remark}\rm We end with some technical remarks on the definition of
`fully \rb' (definition~\ref{def:fully rb did case})
and its relation to the proof just completed.
They are not needed later, and the reader can of course skip them if desired.

It is very helpful throughout the proof
that $U$ is open --- see, e.g., lemma~\ref{lem:reps simple}.
However, we cannot assume in definition~\ref{def:fully rb did case} that $U$ is \ro\ in $X$.
For if we did, then in case~2 of the proof,
we have
$\cl I\subseteq\cl T'\subseteq\cl U'$ by claim~3 and $T'\subseteq U'$, so
$U'\neq U=\int\cl U=\int(\cl U'\cup\cl I)=\int\cl U'$.
Therefore, $U'$ is not \ro\ in $X$, and we can not apply
the inductive hypothesis to it.
We use that $X$ is dense in itself to show that $I\subseteq\cl T'$.

At least according to the construction we gave,
$S$ should be open.
In case 1,
if $S$ is not open then there is $x\in S\setminus\int S\subseteq \cl (U\setminus S)$,
and a little thought shows that $(X,\rho^{-1}),x\models\did {w_0}$ for any such $x$. 
For $\rho$ to be a \rep, we would need $R(\rho(x),w_0)$.
Since $\rho\supseteq\sigma$ and $x\in S$, this says that $R(\sigma(x),w_0)$, which we have no reason to suppose is true.

The problem if $S$ is not \ro\ in $U$ is that,
again in case~1, we used that $U\setminus S=\cl T$.
If this were to fail, there may be 
points $x\in U\setminus (S\cup\cl T)$ (so $x\in U\cap\int\cl S$).
We have to define $\rho$ on these $x$, and defining $\rho(x)={w_0}$ as in the proof
may not give a \rep.
However, as $\sigma$ is $U$-basic, it is possible to define $\rho(x)$ using $\sigma$ instead.
This effectively extends $\sigma$ to $U\cap\int\cl S$.
So we can assume without loss of generality that
$S$ is \ro\ in $U$.  It is therefore easier to do so and avoid the problem completely.

We could just suppose in definition~\ref{def:fully rb did case}
that $S$ is \ro\ in $X$,
but we cannot suppose this of $U$, and we have to work in $RO(U)$,
so there is little gain in doing so.

We need that $\sigma$ is $U$-basic
in order that in case~3, the subsets $S_0,S_1$ have disjoint closures in $U$.
This in turn is needed to apply normality in the proof of 
corollary~\ref{cor:Tarski++opens}.

We cannot assume instead in definition~\ref{def:fully rb did case}
 that $\sigma$ is $X$-basic,
because in case 3, we cannot guarantee that $\mu\cup\sigma_i$ is $X$-basic.
This is because we do not know that $M\cap\cl S_i=\emptyset$,
but only that $U\cap M\cap\cl S_i=\emptyset$.
We could solve this problem by assuming further that $\cl S\subseteq U$ (which implies that $S$ is \ro\ in $X$), but this weakens the proposition sufficiently to
cause trouble in theorem~\ref{thm:strong did} later, where we would need to ensure that $\cl S_n\cup\cl S_{n+1}\subseteq U_n$ for each $n$.

Finally, we mention that actually $\rho(T)=W$ when $T\neq\emptyset$
 --- not only $\rho$ but also $\rho\restriction T$
is surjective.
We might try to drop the second, surjectivity part of definition~\ref{def:Tfull rep} and simply
prove it from the first part, as in cases 1 and 2 of the proof,
but it is not clear how to do this in case 3.
\end{remark}

\section{Weak completeness}

We are now ready to prove our first tranche of main results,
showing that Hilbert systems for various sublanguages
of $\Lbig$ are sometimes sound and always complete over any
non-empty \nice\ space.
Several of the proofs use the translations 
$-^d$ and $-^\mu$ of section~\ref{sec:translations}.
We establish only weak completeness.
We will discuss strong completeness later, in section~\ref{sec:strong c}.

Here and later, we include `$t$' in the name of a Hilbert system
to indicate that it includes the tangle axioms {\bf Fix} and {\bf Ind}
of section~\ref{ss:tangle logics}.
Recall that by lemma~\ref{lem:ms}, metric spaces, regarded as topological spaces,  are Huasdorff and hence T$_D$.

\subsection{Weak completeness for $\c L_\bo^\mu$ and $\c L_\bo^\dit$}\label{ss:weak comp bo case}
The pioneering
result in this field was the theorem of \cite{McKiT44} that the 
$\c L_\bo$-logic of every separable \nice\ space is S4\null.
The assumption of separability was removed in \cite{RS:mm}.
We begin by generalising this theorem,
establishing (weak) completeness results for $\c L_\bo^\mu$ and $\c L_\bo^\dit$
over any \nice\ space.  
We will go on to prove strong completeness in theorem~\ref{thm:str compl boxes}.

\begin{theorem}\label{thm:compl S4mu, S4t}
Let $X$ be a non-empty \nice\ space.
\begin{enumerate}
\item The Hilbert system $\axS4\mu$ is sound and complete over $X$
for $\c L_\bo^\mu$-formulas.

\item The Hilbert system $\axS4t$ is sound and complete over $X$
for $\c L_\bo^\dit$-formulas.
\end{enumerate}

\end{theorem}

\begin{proof}
For part 1, soundness is easy to check and indeed we have already mentioned it
in corollary~\ref{cor:t transl top equiv}.
For completeness, let $\varphi$ be an $\c L_\bo^\mu$-formula
that is not a theorem of $\logicbomu$.
By  theorem~\ref{thm:S4mu sc},
we can find a finite  S4 frame $\c F=(W,R)$, 
an assignment $h$ into $\c F$, and a world $w\in W$
with $(W,R,h),w\models\neg\varphi$.  
By replacing $\c F$ by $\c F(w)$, 
we can suppose that $w$ is a root of $\c F$
--- this can be justified in a standard way using lemma~\ref{lem:gen submodels}.
Since $\c F$ is rooted, it is clearly connected.
Since it is reflexive and transitive,
it is a locally connected $\axK\axD4$ frame.
So by proposition~\ref{prop:di rep}, $\c F$ is fully \rb\ over $X$.
So, taking $U=X$ and $S=\sigma=\emptyset$ in the definition of `fully \rb' 
(definition~\ref{def:fully rb did case}),
we may choose an $X$-full, hence surjective, \rep\ $\rho$ of $\c F$ over $X$.
Choose  $x\in X$ with $\rho(x)=w$.
Then
\[
\begin{array}{rcll}
(W,R,h),w\models\varphi
&\mbox{iff}&
(W,R,h),w\models\varphi^d&\mbox{by lemma~\ref{lem:-d in refl frames}, 
since }\c F\mbox{ is reflexive,}
\\
&\mbox{iff}&(X,\rho^{-1}\circ h),x\models\varphi^d
&\mbox{by proposition~\ref{prop:fmla pres}, since $\varphi^d\in\c L_{\bod\forall}^\mu$,}
\\
&\mbox{iff}&(X,\rho^{-1}\circ h),x\models\varphi
&\mbox{by lemma~\ref{lem:trans equivalent top}, since $X$ is T$_D$.}
\end{array}
\]
We obtain $(X,\rho^{-1}\circ h),x\models\neg\varphi$.
Thus, $\varphi$ is not valid over $X$, proving completeness.

The proof of part 2 is similar. For the soundness of the tangle axioms 
see \cite[Theorem 6.1]{FD:apal12}.
For completeness, the differences are:
$\varphi$ is assumed to be an $\c L_\bo^\dit$-formula
that is not a theorem of $\axS4t$;
we use 
the results of section~\ref{sec:fmp S4} in place of theorem~\ref{thm:S4mu sc} to obtain a finite
S4 Kripke model satisfying $\neg\varphi$ at a root;
and having obtained a surjective \rep\ $\rho$ of $\c F$ over $X$
and $x\in X$ with $\rho(x)=w$, we use the additional translation
$-^\mu$ from section~\ref{sec:translations}, as follows.
Note that $\varphi\in\c L_\bo^\dit$,
$\varphi^d\in\c L_\bod^\didt$,
and $(\varphi^d)^\mu\in \c L_\bod^\mu\subseteq\c L_{\bod\forall}^\mu$.
\[
\begin{array}{rcll}
(W,R,h),w\models\varphi
&\mbox{iff}&
(W,R,h),w\models\varphi^d&\mbox{by lemma~\ref{lem:-d in refl frames}, 
since }\c F\mbox{ is reflexive,}
\\
&\mbox{iff}&(W,R,h),w\models(\varphi^d)^\mu
&\mbox{by lemma~\ref{lem:mu trans}, since $\c F$ is transitive,}
\\
&\mbox{iff}&(X,\rho^{-1}\circ h),x\models(\varphi^d)^\mu
&\mbox{by proposition~\ref{prop:fmla pres},
 since $(\varphi^d)^\mu\in\c L_{\bod\forall}^\mu$,}
\\
&\mbox{iff}&(X,\rho^{-1}\circ h),x\models\varphi^d
&\mbox{by lemma~\ref{lem:mu trans} again,}
\\
&\mbox{iff}&(X,\rho^{-1}\circ h),x\models\varphi
&\mbox{by lemma~\ref{lem:trans equivalent top}, since $X$ is T$_D$.}
\end{array}
\]\unskip
\end{proof}

\subsection{Weak completeness for $\c L_{\bo\forall}$ and $\c L_{\bo\forall}^\dit$}\label{ss:weak comp bo forall}

Completeness for languages with $\forall$ follows the same lines,
although soundness requires that the space be connected.

\begin{theorem}\label{thm:compl S4.UC, S4t.UC}
Let $X$ be a non-empty \nice\ space.
\begin{enumerate}
\item The Hilbert system $\axS4.\axU\axC$ is complete over $X$
for $\c L_{\bo\forall}$-formulas, and sound if $X$ is 
connected.%
\footnote{In \cite[theorem 18]{Sheh:everywhere99},
Shehtman states this result when  $X$ is additionally assumed separable.
However, \cite[footnote 7]{KudShe14}
states that \cite{Sheh:everywhere99} ``contains a stronger claim: 
[the $\c L_{\bo\forall}$-logic of $X$ is $\axS4.\axU\axC$] for any connected dense-in-itself separable metric $X$. However, recently we found a gap in the proof of Lemma 17 from that paper. Now we state the main result only for the case 
$X=\R^n$; a proof can be obtained by applying the methods of the present Chapter, but we are planning to publish it separately.''}

\item The Hilbert system $\axS4t.\axU\axC$ is complete over $X$
for $\c L_{\bo\forall}^\dit$-formulas, and sound if $X$ is connected.
\end{enumerate}
\end{theorem}

\begin{proof}
For part 1, soundness when $X$ is connected is again clear:
connectedness is needed so that the $C$ axiom is valid in $X$.
For completeness, even when $X$ is not connected,
suppose that $\varphi\in\c L_{\bo\forall}$ is not a theorem of $\axS4.\axU\axC$. 
By the results of section~\ref{sec:path conn},
or by \cite[theorem 10]{Sheh:everywhere99},
 $\axS4.\axU\axC$ has the finite model property, so
we can find a finite connected S4 frame $\c F=(W,R)$, 
an assignment $h$ into $\c F$, and a world $w\in W$
such that $(W,R,h),w\models\neg\varphi$.
The proof that $\varphi$ is not valid in $X$ is now exactly as in
theorem~\ref{thm:compl S4mu, S4t}.

Part 2 is proved similarly, using the results of section~\ref{sec:path conn}
to obtain a finite model.
\end{proof}

We have no results for $\c L_{\bo\forall}^\mu$ because
we are not aware of any completeness theorem for this language
with respect to finite connected S4 frames.
If one is proved in future, we could take advantage of it.

\subsection{Weak completeness for $\c L_{\bod}$ and $\c L_{\bod}^\didt$}\label{ss:weak comp bod and tangle}

In one way this is even easier, as we do not need the translation $\varphi^d$.
But again, soundness requires a condition on the space.

\begin{theorem}\label{thm:compl KD4G1(t)}
Let $X$ be a non-empty \nice\ space.
\begin{enumerate}
\item The Hilbert system $\axK\axD4\axG_1$ is complete over $X$
for $\c L_{\bod}$-formulas, and sound if $\axG_1$ is valid in $X$.

\item The Hilbert system $\axK\axD4\axG_1t$ is complete over $X$
for $\c L_{\bod}^\didt$-formulas, and sound if $\axG_1$ is valid in $X$.
\end{enumerate}
\end{theorem}

\begin{proof}
For part 1, soundness is clear.
For completeness, even when $X$ does not validate $\axG_1$,
suppose that $\varphi\in\c L_\bod$ 
is not a theorem of $\axK\axD4\axG_1$.
As we mentioned in section~\ref{ss:weak models},
$\axK\axD4\axG_1$ has the finite  model property  \cite[theorem 15]{Sheh:d90}, so
we can find a finite  $\axK\axD4\axG_1$ frame $\c F=(W,R)$,
 an assignment $h$ into $\c F$, and a world $w\in W$
such that $(W,R,h),w\models\neg\varphi$.
As usual, by replacing $\c F$ by $\c F(w)$, we can suppose that $\c F$
is  connected.
It is also locally connected
because it validates $\axG_1$ (see fact~\ref{fact:Gn = loc nconn}).
Using proposition~\ref{prop:di rep}, let $\rho$ be a 
surjective \rep\ of $\c F$ over $X$.
Let $x\in X$ satisfy $\rho(x)=w$.
Then $(X,\rho^{-1}\circ h),x\models\neg\varphi$
by proposition~\ref{prop:fmla pres}.
So $\varphi$ is not valid in $X$.

The proof of part 2 is similar, except that 
we use the results of section~\ref{sec:fmp Gn} to obtain a finite model,
and in order to apply proposition~\ref{prop:fmla pres},
we first use the translation $-^\mu$ to turn 
$\varphi\in\c L_\bod^\didt$ into an $\c L_\bod^\mu$-formula $\varphi^\mu$
equivalent to $\varphi$ in transitive frames and in $X$.
\end{proof}

\begin{remark}\label{rmk:shehtman}\rm
Theorem~\ref{thm:compl KD4G1(t)}(1) is related to earlier work of Shehtman
\cite{Sheh:d90}.
In \cite[theorem 23, p.39]{Sheh:d90},
the following is proved for the language $\c L_\bod$:
\begin{quote}
\begin{enumerate}
\item[(i)] 
Let  $X$ be a topological space  having an open set homeomorphic to some $\R^n$, $n>0$.  Then $L(D(X))\subseteq D4\axG_1$
[the $\c L_\bod$-logic of $X$ is contained in $\axK\axD4\axG_1$].

\item [(ii)] If additionally $X$ satisfies
conditions of lemma 2 then  $L(D(X))= D4\axG_1$.
\end{enumerate}
\end{quote}
Lemma 2 \cite[p.3]{Sheh:d90} states the following.
\begin{quote}
Let $X$ be a topological space satisfying the following condition:
for any open $U$ and any $x\in U$ there is open $V\subseteq U$
such that $x\in V$ and $(V\setminus\{x\})$ is connected
[as a subspace of $X$].
Then $X\models \axG_1$. 
\end{quote}
Shehtman's  results (i), (ii) above follow from theorem~\ref{thm:compl KD4G1(t)}(1).
We remark that the converse of his lemma~2 fails in general --- a
counterexample is 
given by the subspace
$X=
\R^2\setminus\{(1/n,y):n\mbox{ a positive integer,}\;y\in\R\}
$
of $\R^2$.
\cite[theorems 3.12, 3.14]{LucBry11} give a characterisation of
when a topological space validates $\axG_n$, for $n\geq1$.

Shehtman \cite[p.43]{Sheh:d90} also states two open problems:

\begin{enumerate}
\item To describe all [$\c L_\bod$-]logics [of] dense-in-itself metric spaces $X$.
In particular, is $[K]D4\axG_1$ the greatest of them?

\item Is theorem 23(ii) extended to the infinite dimensional case?
In particular, does it hold for Hilbert space $\ell_2$ (with the weak or with the
strong topology)?
\end{enumerate}
Theorem~\ref{thm:compl KD4G1(t)}(1) appears to resolve 
problem 2 and the second part of problem 1, both positively.

Shehtman also proved in \cite[theorem 29]{Sheh:d90} that
the $\c L_\bod$-logic of every  zero-dimensional separable \nice\  space
is $\axK\axD4$. This does not follow from theorem~\ref{thm:compl KD4G1(t)}.
\end{remark}

\subsection{Weak completeness for $\c L_{\bod\forall}$ and $\c L_{\bod\forall}^\didt$}\label{ss:weak comp bod forall and tangle}

The following is now purely routine.
\begin{theorem}\label{thm:compl KD4G1(t).UC}
Let $X$ be a non-empty \nice\ space.
\begin{enumerate}
\item The Hilbert system $\axK\axD4\axG_1.\axU\axC$ is complete over $X$
for $\c L_{\bod\forall}$-formulas, and sound if $X$ 
is connected and validates $\axG_1$.

\item The Hilbert system $\axK\axD4\axG_1t.\axU\axC$ is complete over $X$
for $\c L_{\bod\forall}^\didt$-formulas, and sound if $X$ 
is connected and validates $\axG_1$.

\end{enumerate}
\end{theorem}

\begin{proof}
The finite model property for $\axK\axD4\axG_1.\axU\axC$
 and $\axK\axD4\axG_1t.\axU\axC$
follows from the results of section~\ref{sec:fmp Gn}.
There are no other new elements in the proof, so we leave it to the reader.
\end{proof}

\section{Strong completeness}\label{sec:strong compl}

Here, we will prove that
$\axK\axD4\axG_1t$ is strongly complete over
any non-empty \nice\ space $X$:
 any countable $\axK\axD4\axG_1t$-consistent set
of  $\c L_\bod^\didt$-formulas
is satisfiable over $X$.
The analogous results for $\c L_\bo^\mu$ and the weaker languages $\c L_\bod$ and $\c L_\bo^\dit$ will follow.
The analogous result for $\c L_\bo$ also follows, but this is a known result,  proved recently by
Kremer \cite{Kremer2010:stro}. We will then show that 
strong completeness frequently fails for languages with $\forall$.

\subsection{The problem}

Let us outline a na\"\i ve approach to the problem.
It does not work, but it will illustrate the difficulty we face and motivate the formal proof later.

Let $\Gamma$ be a countable $\axK\axD4\axG_1t$-consistent set
of $\c L_\bod^\didt$-formulas.
For simplicity, assume that $\Gamma$ is maximal consistent.
Write $\Gamma$ as the union of an increasing chain $\Gamma_0\subseteq\Gamma_1\subseteq\cdots$
of finite sets.
Fix $x\in X$.
By weak completeness (theorem~\ref{thm:compl KD4G1(t)}), each $\Gamma_n$ ($n<\omega$) is satisfiable at $x$,
so we can find an assignment $g_n$ on $X$ with $(X,g_n),x\models\Gamma_n$.
Suppose we could build a new assignment $g$ that 
\emph{behaves like $g_n$ for larger and larger $n$,} as we approach $x$.
Then we might hope that $(X,g),x\models\Gamma_n$ for all $n$, and so
$(X,g),x\models\Gamma$.

To define such a $g$, we
choose a countable sequence $X=S_0\supseteq S_1\supseteq\cdots$
of open \nhd s of $x$, such that 

\begin{enumerate}
\renewcommand{\theenumi}{S\arabic{enumi}}
\renewcommand{\labelenumi}{S\arabic{enumi}.}
\item every open \nhd\ of $x$ contains some $S_n$
(that is, the $S_n$ form a `base of open \nhd s' of $x$).
\end{enumerate}
$X$ is a metric space, so we can do this.
Since we can make the $S_n$ as small as we like, and the $\Gamma_n$ are finite sets,
we can suppose that for each $n<\omega$:
\begin{enumerate}
\renewcommand{\theenumi}{S\arabic{enumi}}
\renewcommand{\labelenumi}{S\arabic{enumi}.}
\setcounter{enumi}{1}
\item \label{naive suppose 1}
for each $\bod\varphi\in\Gamma_n$, we have $(X,g_n),y\models\varphi$ 
for every $y\in S_n\setminus\{x\}$,

\item \label{naive suppose 2} for each $\did\varphi\in\Gamma_n$, there is $y\in S_n\setminus\cl S_{n+1}$
with $(X,g_n),y\models\varphi$.
\end{enumerate}
We can now define a new assignment $g$ by `using $g_n$ within $S_n$', for each 
$n<\omega$.
More precisely, we let
$$
g(p)\cap(S_n\setminus S_{n+1})=g_n(p)\cap(S_n\setminus S_{n+1})
$$
for each atom $p$ and each $n<\omega$.
We also need to define $g$ at $x$ itself, but we can use $\Gamma$
to determine truth values of atoms there.

Now we try to prove that $\varphi\in\Gamma$ iff $(X,g),x\models\varphi$
for all formulas $\varphi$, by induction
on $\varphi$.
The atomic and boolean cases are easy.
Consider the case $\did\varphi$.

If $\did\varphi\in\Gamma$, then $\did\varphi\in\Gamma_n$ for all large enough $n$,
so  by~\ref{naive suppose 2}, there is $y\in S_n\setminus\cl S_{n+1}$ with
$(X,g_n),y\models\varphi$.
As $S_n\setminus\cl S_{n+1}$ is open and $g_n$ agrees with $g$ on it,
it follows that $(X,g),y\models\varphi$.
This holds for cofinitely many $n$, so $(X,g),x\models\did\varphi$.

Conversely, if $(X,g),x\models\did\varphi$,
then for infinitely many $n$, there is $y\in S_n\setminus S_{n+1}$
with $(X,g),y\models\varphi$.
\emph{If we could find such a $y\in S_n\setminus\cl S_{n+1}$,}
then as above, $(X,g_n),y\models\varphi$,
and it would follow by~\ref{naive suppose 1} and maximality of $\Gamma$
that $\did\varphi\in\Gamma$.

But it may be that we can only find such $y\in\cl S_{n+1}$.
The truth of $\varphi$ at such $y$
may not be preserved when we change from $g$ to $g_n$,
because it may depend on points in $S_{n+1}$,
and at such points, $g$ agrees with $g_{n+1}$, not $g_n$.
(We cannot just make $S_{n+1}$ smaller to take the witnesses $y$ out of $\cl S_{n+1}$,
because $g$ will then change,
and we may no longer have $(X,g),y\models\varphi$.)

So we would like to arrange a \emph{smooth transition} between $g_n$
and $g_{n+1}$, avoiding unpleasant discontinuities.
It would be sufficient if there is some closed $T_{n+1}\subseteq S_{n+1}$
such that
$g_n$ and $g_{n+1}$ agree on the `buffer zone' $S_{n+1}\setminus T_{n+1}$.
Much of the formal proof below is aimed at achieving
something like this for atoms occurring in $\Gamma_n$ --- see claim~3 especially.

\subsection{Strong completeness for $\c L_\bod^\didt$}

\begin{theorem}[strong completeness]\label{thm:strong did}
Let $X$ be a non-empty \nice\ space. 
Then the Hilbert system $\axK\axD4\axG_1t$ is strongly complete over $X$
for $\c L_{\bod}^\didt$-formulas, and sound if $\axG_1$ is valid in $X$.
\end{theorem}

\begin{proof}
For soundness, see theorem~\ref{thm:compl KD4G1(t)}.
For strong completeness, 
let $\Gamma$ be a countable $\axK\axD4\axG_1t$-consistent set of $\c L_\bod^\didt$-formulas.
We show that $\Gamma$ is satisfiable over~$X$.
We can suppose without loss of generality that $\Gamma$ is maximal consistent.
Since $\Gamma$ is countable, we can write it as $\Gamma=\bigcup_{n<\omega}\Gamma_n$, where $\Gamma_0\subseteq\Gamma_1\subseteq\cdots$ is a chain of finite sets.
Let $L_n$ be the finite set of atoms occurring in formulas in $\Gamma_n$, 
for each $n<\omega$. So $L_0\subseteq L_1\subseteq\cdots$.
For each $n<\omega$, as $\Gamma_n$ is $\axK\axD4\axG_1t$-consistent,
by the results of section~\ref{sec:fmp Gn} there is a finite Kripke model $\c M_n=(W_n,R_n,h_n)$
whose frame
$(W_n,R_n)$ validates $\axK\axD4\axG_1$, and a world $w_n\in W_n$
with 
\[
\c M_n,w_n\models\Gamma_n.
\]
We can assume without loss of generality that
the $W_n$ $(n<\omega)$ are pairwise disjoint.
For each $n$,
fix an arbitrary $e_n\in W_n$ with $R_nw_ne_n$ and
such that $e_n$ is $R_n$-maximal --- that is, $R^\bullet_n(e_n)=\emptyset$.

For $i\leq j<\omega$ and $w\in W_j$ write
\[
\begin{array}{rcll}
\tp_i(w)&=&\{p\in L_i:\c M_j,w\models p\}&\in\wp L_i
\\
\tau{}^j_i&=&\{\tp_i(w):w\in R_j(e_j)\}&\in\wp\wp L_i
\end{array}
\]
So $\tp_i(w)$ is the `atomic type' of $w$ in $\c M_j$ with respect to the finite set $L_i$ of atoms.
We do not need to write $\tp_i^j(w)$ since the $W_n$ are pairwise disjoint
so $j$ is determined by $w$.
And
$\tau{}^j_i$ is the set of such types that occur as types of points in the cluster $R_j(e_j)$.

\Claim1 We can suppose without loss of generality that $\tau{}^j_i=\tau{}^i_i$
whenever $i\leq j<\omega$.
 
\pfclaim Essentially K\"onig's tree lemma.
We will define by induction infinite sets $\omega=I_{-1}\supseteq I_0\supseteq I_1\supseteq\cdots$.
We let $i_n=\min I_n$, and we will arrange that
 $0=i_{-1}<i_0<i_1<\cdots$ and $i_n\geq n$ for all $n$.
Let $n<\omega$ and suppose that
we are  given $I_{n-1}$ and $i_{n-1}=\min I_{n-1}\geq n-1$ inductively.
Using that $\wp\wp L_{n} $ is finite, choose infinite $I_{n}\subseteq I_{n-1}\setminus\{i_{n-1}\}$
such that $\tau{}^i_n\in\wp\wp L_n$ is constant for all $i\in I_{n}$.
The term $\tau{}^i_n$ is defined for all $i\in I_n$, because $i\geq\min I_n>i_{n-1}\geq n-1$ and so  $i\geq n$.
Of course define $i_n=\min I_n$. Then $i_n>i_{n-1}$ and $i_n\geq n$ as required.
This completes the definition.
Now replace $\c M_n,w_n,e_n$ by $\c M_{i_n},w_{i_n},e_{i_n}$ for each $n<\omega$.
Do not change $\Gamma_n$ or $L_n$.
Since $n\leq i_n$, we have $\Gamma_n\subseteq\Gamma_{i_n}$,
and consequently we still have $\c M_n,w_n\models\Gamma_n$ for each $n$.
And if $r\leq s<\omega$ we have $i_r,i_s\in I_r$, so
$\tau{}^{i_r}_r=\tau{}^{i_s}_r$,
and consequently after replacement, $\tau{}^r_r=\tau{}^{s}_r$.
This proves the claim.

\medskip

For each $n<\omega$, define the frames
\[
\begin{array}{rcl}
\c F_n&=&(R_n(w_n),R_n\restriction R_n(w_n)),
\\
\c C_n&=&(R_n(e_n),R_n\restriction R_n(e_n)).
\end{array}
\]
$\c F_n$ is a generated subframe of $(W_n,R_n)$,
so also a $\axK\axD4\axG_1$-frame; it is connected since $(W_n,R_n)$ validates $\axG_1$.
As $e_n$ is $R_n$-maximal, 
$\c C_n$ is a nondegenerate cluster, so trivially
a connected $\axK\axD4\axG_1$-frame, and (as $R_n$ is transitive) a generated subframe of 
$\c F_n$.
We conclude from proposition~\ref{prop:di rep} that $\c F_n$ and $\c C_n$ are fully \rb\ over $X$, for all $n<\omega$.

Now fix arbitrary $x_0\in X$.
Let $O$ be an open \nhd\ of $x_0$.
Since $X$ is a metric space, all singletons are closed, and since it is 
dense in itself,  lemma~\ref{lem:inf} tells us that
$O$ is infinite, so we can
pick $y\in O\setminus\{x_0\}$.
Then $O\setminus\{y\}$ is open, $\{x_0\}\subseteq O\setminus\{y\}$,
and $\{x_0\}$ is closed.
By lemma~\ref{lem:ms}, $X$ is normal,
so there is open $P$
with $x_0\in P\subseteq\int\cl P\subseteq \cl P\subseteq O\setminus\{y\}\subset O$ (the last inclusion being strict).
Note that $\int\cl P$ is \ro\ in $X$.
So \emph{every open \nhd\ of $x_0$ properly contains the closure of some \ro\ \nhd\ of $x_0$.}
Using this repeatedly,  we may choose \ro\ subsets $O_n,P_n$  of $X$
(for $n<\omega$)
containing $x_0$, with $O_0=X$, and  with the following properties:
\begin{enumerate}
\item $\cl O_{n+1}\subset P_n$ and $\cl P_n\subset O_n$
(the inclusions are strict)  for each $n<\omega$.

\item $O_n\subseteq N_{1/n}(x_0)$ for each $n>0$.
\end{enumerate}
It follows that  for every open \nhd\ $O$ of $x_0$, there is $n<\omega$ with $O_n\subseteq O$.
That is, the $O_n$ form a base of open \nhd s of $x_0$.

For each $n<\omega$ define open sets
\[
\begin{array}{rcl}
U_n&=&O_n\setminus\cl P_{n+1},
\\
S_n&=&O_n\setminus\cl P_n.
\end{array}
\]
See figure~\ref{fig:cone}.
\begin{figure}[t]
\begin{center}
\includegraphics[width=11cm]{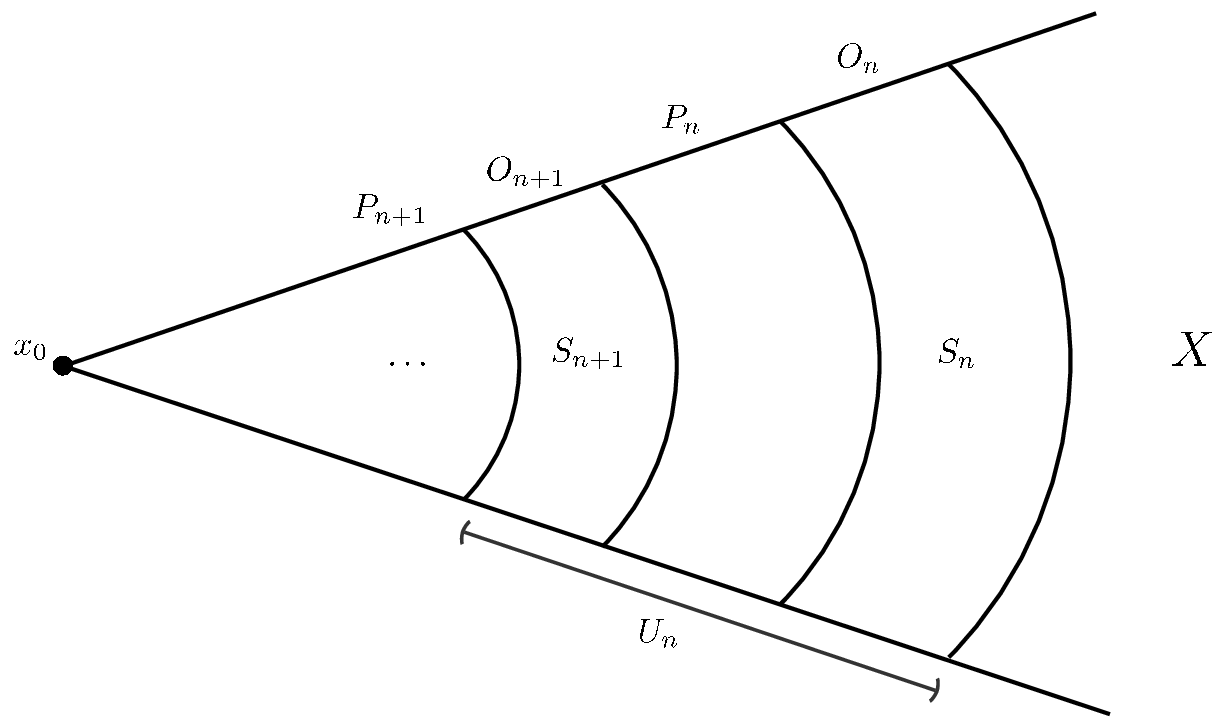}
\end{center}
\caption{rough guide to the sets $O_n,P_n,U_n,S_n$}\label{fig:cone}
\end{figure}
It is easily seen that
\begin{eqnarray}
\bigcup_{n<\omega}(O_n\setminus O_{n+1})&=&X\setminus\{x_0\},\label{e:fact1 on On}
\\
\bigcup_{n\leq m<\omega}U_m&=& O_n\setminus\{x_0\}
 \quad\mbox{ for each }n<\omega.\label{e:fact2 on On}
\end{eqnarray}
The following claim lists some more basic facts about our situation.

\Claim2 For each $n<\omega$:
\begin{enumerate}
\item $U_n\cap U_{n+1}=S_{n+1}\neq\emptyset$.

\item $S_n\cup S_{n+1}\subseteq U_n$,

\item $\cl S_n\cap\cl S_{n+1}=\emptyset$,

\item $S_n$, $S_{n+1}$, and $S_n\cup S_{n+1}$ are \ro\ subsets of $U_n$,

\item $U_n\setminus\cl(S_n\cup S_{n+1})\neq\emptyset$.
\end{enumerate}

\pfclaim 
\begin{enumerate}
\item Easy.

\item From the definitions we have
$S_n=O_n\setminus\cl P_n\subseteq O_n\setminus\cl P_{n+1}=U_n$
and 
$S_{n+1}=O_{n+1}\setminus\cl P_{n+1}\subseteq O_n\setminus\cl P_{n+1}=U_n$.

\item It is clear that
\begin{equation}\label{e:clBn}
\cl S_n\subseteq\cl O_n\setminus P_n.
\end{equation}
Applying this for $n+1$ and $n$ gives
$\cl S_{n+1}\cap\cl S_n\subseteq
\cl O_{n+1}\setminus P_n\subseteq P_n\setminus P_n=\emptyset$.

\item $O_n$ and $P_n$ are \ro\ subsets of $X$,
so by lemma~\ref{lem:ro1}, $S_n=O_n\setminus\cl P_n$ is a 
\ro\ subset  of $X$ too.
Since $\cl S_n\cap\cl S_{n+1}=\emptyset$ by part~2,
lemma~\ref{lem:ro1}(\ref{lem:ro:part2}) 
yields that $S_n\cup S_{n+1}$ is also a \ro\ subset of $X$.
Since each of these three sets is a subset of $U_n$ by part~2,
by lemma~\ref{lem:ro1}(\ref{lem:ro:part3}) it is also \ro\ in $U_n$.

\item By~\eqref{e:clBn} (for $n$ and $n+1$), $\cl S_n$ and $\cl S_{n+1}$ are disjoint
from $P_n\setminus\cl O_{n+1}$,
so by additivity of closure,
$U_n\setminus\cl(S_n\cup S_{n+1})
=U_n\setminus(\cl S_n\cup\cl S_{n+1})\supseteq P_n\setminus\cl O_{n+1}\neq\emptyset$.

\end{enumerate}

\Claim3 There are surjective 
\rep s $\rho_n$ of $\c F_n$ over $U_n$ ($n<\omega$) such that 
\begin{enumerate}
\item  $\rho_n\restriction S_{n+1}$ is a 
 \rep\ of $\c C_n$ over $S_{n+1}$,

\item 
$\tp_n(\rho_n(x))=\tp_{n}(\rho_{n+1}(x))$ for all $x\in S_{n+1}$.
\end{enumerate}

\pfclaim
We define the $\rho_n$ by induction on $n$.
First let $n=0$.
Since $\c C_0$ is fully \rb\ over $X$, we can choose a \rep\
$\sigma:S_1\to\c C_0$. Because $\c C_0$ is a
nondegenerate cluster, $\sigma$ is actually a $U_0$-basic \rep\
(see remark~\ref{rmk:basic reps}).
By claim~2, 
$S_1$ is a \ro\ subset of $U_0$,
and $U_0\setminus\cl S_1\neq\emptyset$.
Now $\c F_0$ is also fully \rb\ over $X$, so $\sigma$ extends 
to a surjective \rep\ $\rho_0$ of $\c F_0$ over $U_0$.
Clearly, condition~1 above is met.

Let $n<\omega$ and assume inductively that for each $m\leq n$,
a surjective \rep\ $\rho_m$ of $\c F_m$
over $U_m$ has been constructed,
such that $\rho_m\restriction S_{m+1}$
is a  \rep\ of $\c C_m$ over $S_{m+1}$ 
and $tp_m(\rho_m(x))=tp_{m}(\rho_{m+1}(x))$ for all $x\in S_{m+1}$ whenever $m<n$.
We will define $\rho_{n+1}$ to continue the sequence.

Note first that since $\c C_n$ is a non-degenerate cluster, $\rho_n\restriction S_{n+1}$ is
$U_n$-basic
 --- see remark~\ref{rmk:basic reps}.
 It is also  surjective.
 For, let $w \in R_n(e_n)$ be given.
Take $x\in S_{n+1}$ (note that $S_{n+1}$ is non-empty by claim~2).
As $\c C_n$ is a non-degenerate cluster,
$R_n(\rho_n(x),w)$, so as $\rho_n\restriction S_{n+1}$ is a \rep, 
$(S_{n+1},(\rho_n\restriction S_{n+1})^{-1}),x\models\did w$.
This certainly implies that $\rho_n(y)=w$ for some $y\in S_{n+1}$.

For each $w\in R_n(e_n)$, define
\[
\begin{array}{rclcl}
D_w&=&\{x\in S_{n+1}:\rho_n(x)=w\}&\subseteq&S_{n+1},
\\
H_w&=&\{v\in R_{n+1}(e_{n+1}):\tp_n(v)=\tp_n(w)\}&\subseteq& W_{n+1},
\\
\c H_w&=&(H_w,R_{n+1}\restriction H_w).
\end{array}
\]
See figure~\ref{fig:maps}.
\begin{figure}[ht]
\begin{center}
\includegraphics[width=11cm]{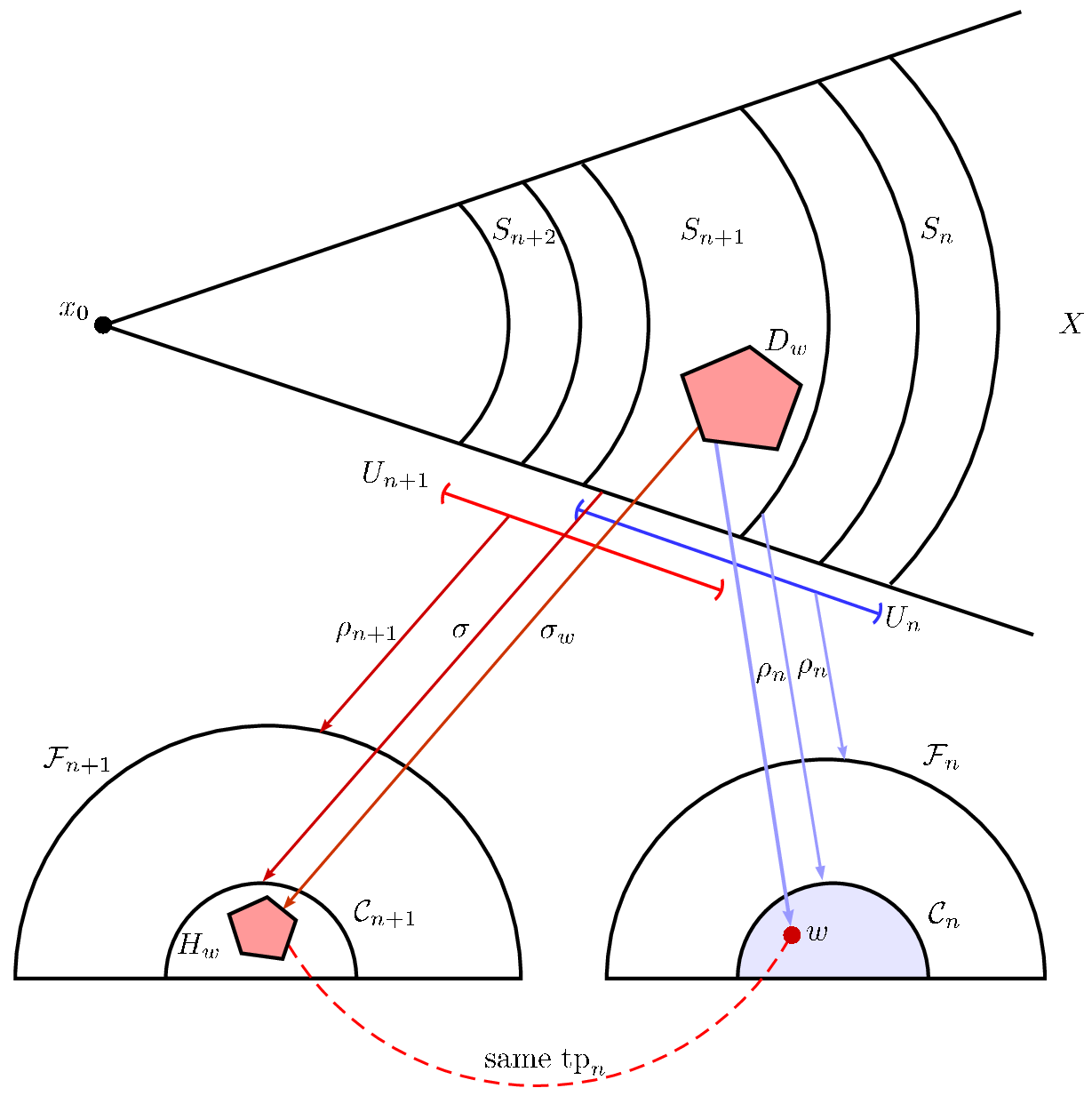}
\end{center}
\caption{illustration for claim 3}\label{fig:maps}
\end{figure}
Because $\rho_n\restriction S_{n+1}$ is surjective onto $\c C_n$,
each set $D_w$ is non-empty,
and plainly, $S_{n+1}$ is partitioned by the $D_w$ ($w\in R_n(e_n)$).
Because $\tau{}^{n+1}_n=\tau{}^n_n$, each $H_w$ is non-empty 
and $\bigcup_{w\in R_n(e_n)}H_w=R_{n+1}(e_{n+1})$.
(The sets $H_w$ may not be pairwise disjoint, but any two of them are equal or disjoint.)

Let $w\in R_n(e_n)$ and consider $D_w$ as a subspace of $X$.
We show that it is dense in itself.
Let $x\in D_w$ and suppose for contradiction that $\{x\}$ is open in $D_w$.
So there is open $O\subseteq X$ with $O\cap D_w=\{x\}$,
and as $S_{n+1}$ is open, we can suppose that $O\subseteq S_{n+1}$.
Now by the inductive hypothesis,
$\rho_n\restriction S_{n+1}$ is a \rep\ of $\c C_n$ over $S_{n+1}$.
Because $\c C_n$ is a non-degenerate cluster, $R_nww$, so
$(X,(\rho_n\restriction S_{n+1})^{-1}),x\models\did w$.
So there is $y\in O\setminus\{x\}$ with
$\rho_n(y)=w$.
But then $y\in O\cap D_w=\{x\}$, a contradiction. 

So $D_w$ is a \nice\ space in its own right.
Since $\c C_{n+1}$
is a nondegenerate cluster, so is its subframe $\c H_w$.
Hence, $\c H_w$ is trivially a finite connected $\axK\axD4\axG_1$ frame.
So by proposition~\ref{prop:di rep}, there is a surjective \rep\
\[
\sigma_w:D_w\to H_w
\]
 of $\c H_w$ over $D_w$.
We have
$(D_{w},\sigma_{w}^{-1}),x\models\did v$ for every $x\in D_w$ and $v\in H_w$.
 By lemma~\ref{lem:bod reps invar},
 \begin{equation}\label{e:subspace rep}
(X,\sigma_{w}^{-1}),y\models\did v\quad
\mbox{ for every }x\in D_w\mbox{ and }v\in H_w.
\end{equation}

Now let 
\[
\sigma=\Big(\bigcup_{w\in R_n(e_n)}\sigma_w\Big)
\;:\;S_{n+1}\to R_{n+1}(e_{n+1}).
\]
The sets $D_w$  partition $S_{n+1}$, so 
$\sigma$ is a well defined and total map.
It has the following property.
Let $x\in S_{n+1}$.
Writing $\rho_n(x)=w$, say,
we have $x\in D_w$ and $\sigma(x)=\sigma_w(x)\in H_w$, so
$\tp\nolimits_n(\sigma(x))=\tp\nolimits_n(w)$
by definition of $H_w$.
That is,
\begin{equation}\label{e:types pres by sig}
\tp\nolimits_n(\sigma(x))=\tp\nolimits_n(\rho_n(x))\quad
\mbox{ for each }x\in S_{n+1}.
\end{equation}

We show that $\sigma$ is a \rep\ of $\c C_{n+1}$ over $S_{n+1}$.
Since $\c C_{n+1}$ is a non-degenerate cluster,
we need show only that
$(X,\sigma^{-1}),x\models\did v$ for every $x\in S_{n+1}$ and 
$v\in R_{n+1}(e_{n+1})$.

So take such $x,v$.
Suppose that $\rho_n(x)=w$, say, so $x\in D_w$.
Choose $w'\in R_n(e_n)$ such that $v\in H_{w'}$ (it may not be unique). 
As $\c C_n$ is a cluster,
$R_n(w,w')$.
As $\rho_n\restriction S_{n+1}$ is a \rep\ of $\c C_n$ over $S_{n+1}$,
we have $(X,(\rho_n\restriction S_{n+1})^{-1}),x\models\did w'$.
That is,
$x\in\did D_{w'}$.
But by~\eqref{e:subspace rep},
$(X,\sigma^{-1}),y\models\did v$ for every $y\in D_{w'}$.
It follows that
$(X,\sigma^{-1}),x\models\did\did v$, and hence  
$(X,\sigma^{-1}),x\models\did v$ as required.

\smallskip

So $\sigma$ is indeed a \rep\ of $\c C_{n+1}$ over $S_{n+1}$.
As $\c C_{n+1}$ is fully \rb\ over $X$, we may
choose a 
\rep\ $\sigma'$ of $\c C_{n+1}$
over $S_{n+2}$.
By claim~2, $S_{n+1}\cap S_{n+2}=\emptyset$, so
by lemma~\ref{lem:reps simple}, $\sigma\cup\sigma'$ is a well defined
\rep\ of $\c C_{n+1}$
over the \ro\ subset $S_{n+1}\cup S_{n+2}$   of
$U_{n+1}$.
Also, $U_{n+1}\setminus\cl(S_{n+1}\cup S_{n+2})\neq\emptyset$.
And since $\c C_{n+1}$ is a nondegenerate cluster,
$\sigma\cup\sigma'$ is $U_{n+1}$-basic (see remark~\ref{rmk:basic reps} again).
We can now use the fact that $\c F_{n+1}$ is fully \rb\ over $X$ to extend $\sigma\cup\sigma'$ is to
a surjective \rep\
$\rho_{n+1}$ of $\c F_{n+1}$ over $U_{n+1}$.
Then
$\rho_{n+1}\restriction S_{n+2}=\sigma'$
is a  \rep\ of $\c C_{n+1}$ over $S_{n+2}$, 
and by~\eqref{e:types pres by sig},
$tp_n(\rho_n(x))=\tp_n(\sigma(x))=tp_{n}(\rho_{n+1}(x))$ for all $x\in S_{n+1}$.
This proves claim~3. 

\bigskip

Let $n<\omega$.  Define an assignment $g_n$ on $U_n$
by 
\begin{equation}\label{e:def g}
g_n(p)=\rho_n^{-1}(h_n(p))\mbox{ for each atom }p.
\end{equation}
By the claim, if $p\in L_n$, then for each $x\in S_{n+1}$
we have
$x\in g_n(p)$ iff $\rho_n(x)\in h_n(p)$,
iff $p\in\tp_n(\rho_n(x))=\tp_n(\rho_{n+1}(x))$,
iff $\rho_{n+1}(x)\in h_{n+1}(p)$, iff $x\in g_{n+1}(p)$.
So
\begin{equation}\label{eq:agree}
S_{n+1}\cap g_n(p)=S_{n+1}\cap g_{n+1}(p)\quad\mbox{ for each }p\in L_n.
\end{equation}
Finally, define an assignment $g$ on $X$
as follows.  Let $p$ be an atom.
\begin{itemize}

\item For $x\in X\setminus\{x_0\}$, define $x\in g(p)$ iff
$x\in g_n(p)$, where $x\in O_n\setminus O_{n+1}$.

Since the $O_n\setminus O_{n+1}$ are pairwise disjoint,   and
$\bigcup_{n<\omega}(O_n\setminus O_{n+1})=X\setminus\{x_0\}$ by \eqref{e:fact1 on On},
this is well defined.

\item Define $(X,g),x_0\models p$ iff $p\in\Gamma$.
\end{itemize}

\Claim4 Let $n<\omega$, let  $x\in U_n$, and 
let $\varphi$ be a formula whose atoms lie in $L_n$.
Then $(X,g),x\models\varphi$ iff $\c M_n,\rho_n(x)\models\varphi$.

\pfclaim 
Let $p\in L_n$ be arbitrary.
Recall that $U_n=O_n\setminus\cl P_{n+1}$.
By definition of $g$,
if $x\in O_n\setminus O_{n+1}$ then  $x\in g(p)$ iff 
$x\in g_n(p)$.
If instead $x\in O_{n+1}$, then 
$x\in O_{n+1}\setminus\cl P_{n+1}=S_{n+1}\subseteq O_{n+1}\setminus O_{n+2}$, 
and since $p\in L_{n+1}$ too, the definition of $g$  gives
$x\in g(p)$ iff $x\in g_{n+1}(p)$.
But by \eqref{eq:agree},
this is iff $x\in g_n(p)$ again.
So $g$ and $g_n$ agree on $U_n$ as far as atoms in $L_n$ are concerned,
and as $U_n$ is open, it follows easily that 
$(X,g),x\models\varphi$ iff $(U_n,g_n),x\models\varphi$.
Since $\rho_n$ is a \rep\ over $U_n$ of the generated subframe $\c F_n$
of $(W_n,R_n)$, by lemma~\ref{lem:reps simple} it is also a \rep\
of $(W_n,R_n)$ over $U_n$.
So by \eqref{e:def g} and proposition~\ref{prop:fmla pres}, $(U_n,g_n),x\models\varphi$
iff $\c M_n,\rho_n(x)\models\varphi$. 
 This proves the claim.

\medskip

\Claim5 For all $\varphi$ we have $(X,g),x_0\models\varphi$ iff $\varphi\in\Gamma$.

\pfclaim
By induction on $\varphi$.  For atoms, the result follows from the definition of~$g$.
The boolean operators are handled in the usual way by induction, using 
the maximal consistency of $\Gamma$; they are the only cases in which the inductive hypothesis is used.

We now tackle the case $\bod\varphi$.
It is sufficient (and seems more intuitive) to deal with $\did\varphi$ instead.
Suppose first that  $\did\varphi\in \Gamma$.
Choose $n<\omega$ such that $\did\varphi\in\Gamma_n$.
Let $i\geq n$ be arbitrary.
Then $\did\varphi\in\Gamma_i$, so $\c M_i,w_i\models\did\varphi$,
and hence there is $v\in R_i({w_i})$ with $\c M_i,v\models\varphi$.
As $\rho_i:U_i\to R_i(w_i)$ is surjective (see claim~3), there is
$x\in U_i$ with $\rho_i(x)=v$.
Since $\did\varphi\in\Gamma_i$, the atoms of $\varphi$ lie in $L_i$,
so claim 4 applies: $(X,g),x\models\varphi$.
We conclude that for every $i\geq n$ there is $x\in U_i$ with $(X,g),x\models\varphi$.
As $U_i\subseteq O_i\setminus\{x_0\}$ and the $O_i$ form a base of \nhd s of $x_0$, 
it follows that $(X,g),x_0\models\did\varphi$.

Conversely, suppose that $(X,g),x_0\models\did\varphi$.
For each $n<\omega$, $O_n$ is an open \nhd\ of $x_0$,
so there is $x\in O_n\setminus\{x_0\}$ with $(X,g),x\models\varphi$.
Since $O_n\setminus\{x_0\}=\bigcup_{n\leq i<\omega}U_i$ by \eqref{e:fact2 on On}, we have $x\in U_i$
for some $i\geq n$.
It follows that there are infinitely many $i<\omega$ such that 
 $(X,g),x\models\varphi$  for some $x\in U_i$.
 Since the atoms of $\varphi$ lie in $L_i$ for cofinitely many $i$,
there must be infinitely many $i$ with $\c M_i,v\models\varphi$ for some $v\in R_i({w_i})$ (by claim~4), 
and so  $\c M_i,w_i\models\did\varphi$ (by Kripke semantics),
and so $\neg\did\varphi\notin\Gamma_i$ (since $\c M_i,w_i\models\Gamma_i$).
Since $\Gamma$ is the union of the chain $\Gamma_0\subseteq\Gamma_1\subseteq\cdots$,
we have  $\neg\did\varphi\notin\Gamma$.
As $\Gamma$ is maximal consistent, it follows that $\did\varphi\in\Gamma$.

Finally, consider the case $\didt\Delta$, where $\Delta$ is 
any non-empty finite set of formulas.
Suppose first that $\did\Delta\in\Gamma$.
We only sketch the proof here, referring the reader
to the case of $\did\varphi$ for more details.
Pick any $\delta\in\Delta$.
Then  as in the case of $\did\varphi$,
each of the following holds for cofinitely many $i<\omega$:
\begin{itemize}
\item $\didt\Delta\in\Gamma_i$
\item $\c M_i,w_i\models\didt\Delta$
\item there is $v\in R_i(w_i)$
with $\c M_i,v\models\delta\wedge\didt\Delta$
\item there is $x\in U_i$ with $(X,g),x\models\delta\wedge\didt\Delta$.

\end{itemize}
As the latter holds for every $\delta\in\Delta$,
it follows that $(X,g),x_0\models\didt\Delta$.

Conversely, suppose $(X,g),x_0\models\didt\Delta$.
Then as in the $\did\varphi$ case, 
there are infinitely many $i<\omega$ such that
$(X,g),x\models\didt\Delta$ for some $x\in U_i$.
Since the atoms of $\didt\Delta$ lie in $L_i$ for cofinitely many $i<\omega$,
it follows by claim~4 that there are infinitely many $i$ 
such that 
there is $v\in R_i(w_i)$
with $\c M_i,v\models\didt\Delta$,
and hence ---
by the semantics of $\didt$ --- $\c M_i,w_i\models\didt\Delta$.
As in the $\did\varphi$ case, we obtain $\neg\didt\Delta\notin\Gamma_i$ 
for infinitely many $i$, so $\neg\didt\Delta\notin\Gamma$,
and so $\didt\Delta\in\Gamma$ by maximal consistency of $\Gamma$.
The claim is proved, and the theorem with it.
\end{proof}

\subsection{Strong completeness for $\c L_\bod$}\label{ss:compactness bod}
We can now easily derive the analogous result
for `modal' $\c L_\bod$-formulas, essentially by showing
that $\axK\axD4\axG_1t$ is a conservative extension of $\axK\axD4\axG_1$.

\begin{theorem}\label{thm:sc bod}
Let $X$ be a non-empty \nice\ space. 
Then the Hilbert system $\axK\axD4\axG_1$ is strongly complete over $X$
for $\c L_{\bod}$-formulas, and sound if $\axG_1$ is valid in $X$.
\end{theorem}

\begin{proof}
For soundness, see theorem~\ref{thm:compl KD4G1(t)}.
For strong completeness, 
let $\Gamma$ be a countable $\axK\axD4\axG_1$-consistent set of $\c L_\bod$-formulas.
Let $\Gamma_0\subseteq\Gamma$ be finite and put $\gamma=\bigwedge\Gamma_0$.
Then $\gamma$ is $\axK\axD4\axG_1$-consistent, so 
by the results of section~\ref{sec:fmp Gn}
it is satisfied in some finite $\axK\axD4\axG_1$-frame $\c F$.
Plainly, $\c F$ is also a $\axK\axD4\axG_1t$-frame, and it follows that
$\gamma$ is $\axK\axD4\axG_1t$-consistent.
So $\Gamma$ is $\axK\axD4\axG_1t$-consistent.
By theorem~\ref{thm:strong did}, $\Gamma$ is satisfiable over $X$.
\end{proof}

\subsection{Strong completeness for $\c L_\bo^\dit$ and $\c L_\bo^\mu$}\label{sec:strong c}
This also follows, using the translations $-^d$ and $-^t$ of section~\ref{sec:translations}.

\begin{theorem}\label{thm:str compl boxes}
Let $X$ be any \nice\ space.
\begin{enumerate}
\item The Hilbert system $\axS4t$ is sound and strongly complete over $X$
for $\c L_\bo^\dit$-formulas.

\item \label{stro compl boxes part2}
The Hilbert system $\logicbomu$ is sound and strongly complete over $X$
for $\c L_\bo^\mu$-formulas.

\item (Kremer, \cite{Kremer2010:stro})
The Hilbert system $\axS4$ is sound and strongly complete over $X$
for  $\c L_\bo$-formulas.

\end{enumerate}

\end{theorem}

\begin{proof}
Soundness is clear in all cases: cf.~theorem~\ref{thm:compl S4mu, S4t}.
We prove strong completeness.
For part~1,
let $\varphi$ be an $\axS4t$-consistent $\c L_\bo^\dit$-formula.
By the results of section~\ref{sec:fmp S4},
$\varphi$ is satisfiable in some finite S4 Kripke frame $\c F$.
Recall from section~\ref{sec:translations} the translation $-^d$:
it takes $\c L_\bo^\dit$-formulas to $\c L_\bod^\didt$-formulas.
Since $\c F$ is reflexive, it follows from  lemma~\ref{lem:-d in refl frames}
that $\varphi^d$ is
equivalent to $\varphi$ in $\c F$.
So $\varphi^d$ is satisfiable in $\c F$.
Plainly, $\c F$ is also a $\axK\axD4\axG_1t$ frame,
so $\varphi^d$ is $\axK\axD4\axG_1t$-consistent.

Since $-^d$ commutes with $\wedge$,
it is now easily seen that if $\Gamma\subseteq\c L_\bo^\dit$ is a countable 
$\axS4t$-consistent set then
$\Gamma^d=\{\gamma^d:\gamma\in\Gamma\}\subseteq\c L_\bod^\didt$ is a countable
$\axK\axD4\axG_1t$-consistent set.  
By theorem~\ref{thm:strong did}, $\Gamma^d$ is satisfiable over $X$.
Since $X$ is T$_D$, by lemma~\ref{lem:trans equivalent top}
each $\gamma\in\Gamma$ is equivalent to $\gamma^d$ in $X$, so
$\Gamma$ is also satisfiable over $X$.

\medskip

For part~2,
for a set $\Gamma\subseteq\c L_\bo^\mu$ 
we write $\Gamma^t=\{\gamma^t:\gamma\in\Gamma\}\subseteq\c L_\bo^\dit$,
where the translation $-^t:\c L_\bo^\mu\to\c L_\bo^\dit$  is as in section~\ref{ss:DO09}.
Let $\Gamma\subseteq\c L_\bo^\mu$ be a countable $\axS4\mu$-consistent set.
Let $\Gamma_0\subseteq\Gamma$ be any finite subset.
By assumption, the formula $\bigwedge\Gamma_0$ is $\axS4\mu$-consistent.
So by theorem~\ref{thm:S4mu sc},
there is a finite $\axS4$ frame $\c F$ in which $\bigwedge\Gamma_0$ is satisfied.
By fact~\ref{fact:DO}, 
$\varphi^t$ is equivalent to $\varphi$ in
$\c F$, for each $\varphi\in\c L_\bo^\mu$.
So  $\bigwedge(\Gamma_0^t)$ is also satisfied in $\c F$.
Since $\c F$ is plainly an $\axS4t$ frame,
it follows that
$\bigwedge(\Gamma_0^t)$ is $\axS4t$-consistent.
As $\Gamma_0$ was arbitrary, $\Gamma^t$ is $\axS4t$-consistent.

By part~1, $\Gamma^t$ is satisfied in $X$.
But by corollary~\ref{cor:t transl top equiv}, 
each $\gamma\in\Gamma$ is equivalent to $\gamma^t$ in $X$.
So $\Gamma$ is also satisfied in $X$.

\medskip
Part~3 can be proved similarly,
by showing in the same way that for $\c L_\bo$-formulas, $\axS4$-consistency
implies $\axS4t$-consistency, and then appealing to part 1.
\end{proof}

\subsection{Universal modality}

%%Redefine rob's macro
\def\g{{\sf g}}

We do not include the universal modality in our strong completeness results,
for good reason.

\begin{theorem}\label{thm:comp fails}
There is a set $\Sigma$ of
$\c L_{\bo\forall}$-formulas such that
for every non-empty compact locally connected \nice\ space $X$,
each finite subset of $\Sigma$ is satisfiable in $X$, but $\Sigma$ as a whole is not.
\end{theorem}
\emph{Compact} means that
if $\c S$ is a set of open sets with $\bigcup\c S=X$,
then $X=\bigcup\c S_0$ for some finite $\c S_0\subseteq\c S$.
\emph{Locally connected} means that every open neighbourhood of a point $x$ 
contains a connected (in the subspace topology) open neighbourhood of $x$.
An example of a compact locally connected \nice\ space is
the subspace  $[0,1]$ of $\R$.

\begin{proof}
The proof is based on the following model $\c M=(W,R,h)$, where 
we suppose that
$\Var=\{\r,\g,\bb\}\cup\{p_i:i<\omega\}$.
\begin{enumerate}
\item $W=\{a_n,b_n:n<\omega\}$, where the $a_n$ and $b_n$ are pairwise distinct

\item $R$ is the reflexive closure of $\{(a_n,b_n),(a_n,b_{n+1}):n<\omega\}$

\item $h(\r)=\{b_{3n}:n<\omega\}$,
$h(\g)=\{b_{3n+1}:n<\omega\}$,
$h(\bb)=\{b_{3n+2}:n<\omega\}$,
and $h(p_n)=\{b_{3n},b_{3n+1}\}$ for each $n<\omega$.
\end{enumerate}

\begin{figure}[ht]
\setlength{\unitlength}{.6pt}
\begin{picture}(100,140)(-60,-45)
\put(0,0){\circle{6}}

\multiput(0,0)(80,0){6}{
\put(42,77){\vector(1,-2){37}}
\put(38,77){\vector(-1,-2){37}}
\put(40,80){\circle{6}}
\put(80,0){\circle{6}}
}

\put(15,80){$a_0$}
\put(95,80){$a_1$}
\put(175,80){$a_2$}
\put(255,80){$a_3$}
\put(335,80){$a_4$}
\put(415,80){$a_5$}

\put(-25,0){$b_0$}
\put(55,0){$b_1$}
\put(135,0){$b_2$}
\put(215,0){$b_3$}
\put(295,0){$b_4$}
\put(375,0){$b_5$}
\put(455,0){$b_6$}

%\put(30,110){$p_0$}
%\put(110,110){$p_1$}
%\put(190,110){$p_2$}
%\put(270,110){$p_3$}
%\put(350,110){$p_4$}
%\put(430,110){$p_5$}
%\put(510,110){$p_6$}

\put(0,-40){$p_0$}
\put(80,-40){$p_0$}
%\put(160,-40){$p_2$}
\put(240,-40){$p_1$}
\put(320,-40){$p_1$}
%\put(400,-40){$p_5$}
\put(480,-40){$p_2$}

\put(0,-20){$\r$}
\put(240,-20){$\r$}
\put(480,-20){$\r$}

\put(80,-20){$\g$}
\put(320,-20){$\g$}

\put(160,-20){$\bb$}
\put(400,-20){$\bb$}

%\put(35,90){$\y$}
%\put(275,90){$\y$}
%
%\put(115,90){$\cy$}
%\put(355,90){$\cy$}
%
%\put(195,90){$\m$}
%\put(435,90){$\m$}

\put(500,41){\vector(-1,-2){19}}
\put(520,40){$\cdots$}

\end{picture}
\caption{$\c M$}\label{fig:w}
\end{figure}
\noindent The model is  shown in figure~\ref{fig:w} ---
it goes off to the right forever, roughly repeating after every three steps.
Of course $R$ is reflexive.
Note that the underlying frame is connected.

We let $\Sigma$ be the set comprising the following formulas:
\begin{enumerate}
\renewcommand{\theenumi}{$\Sigma$\arabic{enumi}}
\renewcommand{\labelenumi}{$\Sigma$\arabic{enumi}.}

%
%
%
%\item\label{sigma1} $\c M$ is partitioned by the colours: $\forall(\r\vee\g\vee\bb\vee\m\vee\y\vee\cy)$, $\forall\neg(\r\wedge\g)$, etc.

%\item $\r,\g,\bb$ are open --- $\forall(\r\to\bo\r)$, etc
%\item $\y,\cy,\m$ are closed --- $\forall(\di\y\to\y)$, etc

\item\label{sigma2}  $\exists( \di p_i\wedge\di\r\wedge\di\g)$ 
%$\exists(\cy\wedge p_{3i+1})$,
%and $\exists(\m\wedge p_{3i+2})$
 for each $i<\omega$
 
%\item \label{sigma3} $\forall\neg(p_i\wedge \di p_j)$ for each distinct  $i,j<\omega$, 

\item\label{sigma3a}   $\forall\neg(\di p_i\wedge\di p_j)$ for $i<j<\omega$
 
%  \item\label{sigma3b}  $\forall(p_i\to\bo p_i)$ for  $i<\omega$ 

% \item\label{sigma4}  $\forall\neg\bo \y$ 
% and its \emph{transforms} $\forall\neg\bo \cy$, $\forall\neg\bo \m$

%\item \label{sigma5} $\forall(\y\to\di\r\wedge\di\g)$

%\item $\forall(\y\wedge(\di p_{3i}\vee\di p_{3i+1})\to\bo\neg p_j)$
%for all $i<\omega$ and $j\in\omega\setminus\{3i,3i+1\}$,
%and similarly for $\cy,\y$.

%\item\label{sigma6} $\forall(\neg p_i\wedge\neg\di\bb\to \y)$ for $i<\omega$

%\item\label{sigma7} $\forall(\y\to\neg\di p_i)$ for all $i=2\pmod3$, and transforms

%\item\label{sigma8} $\forall(\y\wedge\di p_i\to
%\bo(\y\vee  p_i))$
%for  $i<\omega$
%, and transforms

%\item\label{sigma9} $\forall(\di\bb\wedge(\di p_i\vee\di p_{i+1})\to\cy\wedge\di p_{i+1}\wedge\di p_{i+2})$
%for all $i=0\pmod3$, and transforms

\item\label{sigma10} $\forall\neg(\di\r\wedge\di\g\wedge\di\bb)$

\item \label{sigma15} $\forall(\di p_i\wedge\bo\neg\bb\to\bo\di p_i)$
for $i<\omega$.
\end{enumerate}
They are plainly valid  in $\c M$.
Hence $\Sigma$ is satisfied  in $\c M$, at every point.
Moreover, any finite subset  $\Sigma_0\subseteq\Sigma$ is satisfied in a finite submodel of $\c M$
obtained by taking a large enough `initial segment' of $\c M$ ending on the right at a $b$-world.  
Check especially formulas of the form $\forall\exists$.
In particular, \ref{sigma15} is valid
in such a submodel. Or one can use that it is a generated submodel.
The submodel is finite and its frame
validates $\axS4.\axU\axC$, so every formula satisfied in it ---
for example, $\bigwedge\Sigma_0$ --- is $\axS4.\axU\axC$-consistent.
Hence, by theorem~\ref{thm:compl S4.UC, S4t.UC}, 
every finite subset of $\Sigma$ is satisfiable in $X$.

Assume for contradiction that $\Sigma$ is satisfied in some model $(X,h)$ on $X$.
Below, we will write $x\models\varphi$ instead of $(X,h),x\models\varphi$.
By \ref{sigma2}, for each $i<\omega$ there is $x_i\in X$ 
with $x_i\models \di p_i\wedge\di\r\wedge\di\g$. 
As $X$ is compact, it contains a point $z$ such that 
for every open \nhd\ $N$ of $z$, the set $\{i<\omega:x_i\in N\}$ is infinite.
Then $z\models\di\r\wedge\di\g$ as well.
By \ref{sigma10}, 
$z\models\bo\neg\bb$.
As $X$ is locally connected,
there is  a connected open neighbourhood $N$ of $z$ 
with $y\models\neg\bb$ for all $y\in N$.

Take $i<j<\omega$ with $x_i,x_j\in N$.
Let 
$U=\{x\in N:x\models\di p_i\}$.
Then $U$ is an open subset of $N$, because
for every $u\in U$ we have $u\models\di p_i\wedge\bo\neg\bb$,
and \ref{sigma15} gives $u\models\bo\di p_i$.
And $N\setminus U$ is also open, because $U'=\{x\in X:x\models\di p_i\}$ is closed
and $N\setminus U=N\setminus U'$.
 We have $x_i\in U$, but by \ref{sigma3a},
$x_j\in N\setminus U$.
So $N$ is the union of two disjoint non-empty open sets ($U$ and $N\setminus U$), contradicting
its connectedness.
\end{proof}

\begin{corollary}\label{cor:stro compl fails}
Let $X$ be a non-empty compact locally connected \nice\ space,
and $\c L\subseteq\Lbig$ a language
containing $\c L_{\bo\forall}$ or $\c L_{\bod\forall}$.
Then no Hilbert system for $\c L$
is sound and strongly complete over $X$.
\end{corollary}

\begin{proof}
Assume for contradiction that
the Hilbert system $H$ is sound and strongly complete over $X$.
Let $\Sigma$ be as in theorem~\ref{thm:comp fails}
(use the translation $-^d$ if necessary to ensure it is a set of $\c L$-formulas).
Since every finite subset of $\Sigma$ is  satisfiable in $X$,
and $H$ is sound over $X$, it follows that $\Sigma$ is $H$-consistent.
But $H$ is strongly complete over $X$, so $\Sigma$ is satisfiable over $X$, 
contradicting the theorem.
\end{proof}

\section{Conclusion}\label{sec:the end}

This paper has presented some completeness theorems
for various spatial logics over \nices\  spaces.
Table~\ref{table1} summarises them.
The numbers in parentheses refer to our earlier results.
The first line of the table is of course known, included here to give a more complete picture. 
%For $\c L_{\bo\forall}$, weak completeness of $\axS4.\axU\axC$ in the separable case
%was proved in  \cite[theorem 18]{Sheh:everywhere99}.
For handy reference, table~\ref{table2} 
summarises the ingredients of each logic.

\begin{table}[ht]
\begin{center}
\begin{tabular}{lllll}
\hline
Language&Logic&sound&complete&strongly complete
\\
\hline
$\c L_\bo$&$\axS4$&yes&yes \cite{McKiT44}&yes \cite{Kremer2010:stro}
\\[3pt]
$\c L_\bo^\mu$&$\axS4\mu$&yes
&yes \eqref{thm:compl S4mu, S4t}&yes  \eqref{thm:str compl boxes}
\\[3pt]
$\c L_\bo^\dit$&$\axS4t$&yes
&yes \eqref{thm:compl S4mu, S4t}&yes  \eqref{thm:str compl boxes}
\\[3pt]
$\c L_{\bo\forall}$&$\axS4.\axU\axC$&if $X$ connected
 &yes \eqref{thm:compl S4.UC, S4t.UC}&not in general \eqref{cor:stro compl fails}
\\[3pt]
$\c L_{\bo\forall}^\dit$
&$\axS4t.\axU\axC$  &if $X$ connected 
 &yes \eqref{thm:compl S4.UC, S4t.UC}&not in general \eqref{cor:stro compl fails}
\\[3pt]
$\c L_\bod$
& $\axK\axD4\axG_1$  &if $\axG_1$ valid in $X$ 
 &yes \eqref{thm:compl KD4G1(t)}&yes \eqref{thm:sc bod}
\\[3pt]
$\c L_{\bod}^\didt$
& $\axK\axD4\axG_1t$  &if $\axG_1$ valid in $X$ 
 &yes \eqref{thm:compl KD4G1(t)}&yes \eqref{thm:strong did}
\\[3pt]
$\c L_{\bod\forall}$
& $\axK\axD4\axG_1.\axU\axC$  &if $X$ connected \& validates $\axG_1$
 &yes \eqref{thm:compl KD4G1(t).UC}&not in general \eqref{cor:stro compl fails}
\\[3pt]
$\c L_{\bod\forall}^\didt$
& $\axK\axD4\axG_1t.\axU\axC$  &if $X$ connected \& validates $\axG_1$
 &yes \eqref{thm:compl KD4G1(t).UC}&not in general \eqref{cor:stro compl fails}
\\[5pt]
\hline
\end{tabular}
\caption{Soundness and completeness for a non-empty \nice\ space $X$}\label{table1}
\end{center}
\end{table}

\begin{table}[ht]
\begin{center}
\begin{tabular}{ll}
\hline
S4&$\bo\varphi\to\varphi$, $\bo\varphi\to\bo\bo\varphi$
\\
$\mu$&fixed point axiom and rule: see definition~\ref{def:mu systems}
\\
$t$&tangled closure axioms from section~\ref{ss:tangle logics}
\\
U&$\forall\varphi\to\bo\varphi$, S5 axioms for $\forall$, 
$\forall$-generalisation rule
\\
C&$\forall(\Box^*\ph\lor\Box^*\neg\ph)\to(\forall\ph\lor\forall\neg\ph)$,
where $\bo^*\varphi=\ph\land\Box\ph$
\\
$\axG_1$&all uniform substitution instances of
$\big(\bod\bigvee_{i=0}^1\bo Q_i\big)
\to\bigvee_{i=0}^1\bod\neg Q_i$,
\\
&\hskip1em where $Q_i=p_i\wedge\neg p_{1-i}$ ($i=0,1$)
\\[2pt]
\hline
\end{tabular}
\caption{Parts of the logics}\label{table2}
\end{center}
\end{table}

There are of course many problems left open by our work, and we present some of them here.

\subsection{Extensions}

\begin{problem}
Can the results be extended to  more general topological spaces?
\end{problem}
For example, consider the topological space $T$ defined as follows.
For ordinals $\alpha,\beta$ write ${}^\alpha\beta$
for the set of all maps $f:\alpha\to\beta$.
The set of points of $T$ is $\bigcup_{n\leq\omega}{}^n2$,
and the open sets are unions of sets of the form
$\{f\in T:f\supseteq g\}$ for some  $g\in\bigcup_{n<\omega}{}^n2$. 
This space is not even T$_D$, though it is T0 
(that is, no two distinct points have the same open \nhd s)
and dense in itself.
\begin{problem}
What is the logic of $T$ in the various languages discussed above?
\end{problem}

\begin{problem}
Can the results be extended to stronger languages,
for example, the mu-calculus with $\bod$ and/or \pa,
languages with the difference modality or graded modalities, 
hybrid languages, and so on?
\rm Results of Kudinov \cite{Kud06:topdiff,Kud08:Rnfa}
are relevant.
Recently, Kudinov and Shehtman
\cite{KudShe14}
 proved numerous
results about logics of topology with 
$\bo$, $\bod$, \pa, and the `difference modality' $[\neq]$.
In particular, they determine the logic of $\R^n$ for $n\geq2$
in the language with $\bod$ and $[\neq]$.
However, results for general \nices\ spaces appear to be lacking.
%This paper has some overlap with ours, but here
%we consider the mu-calculus and 
%tangled closure operators, while \cite{KudShe14} is more concerned
%with other languages.

\end{problem}

\subsection{Strong completeness}\label{subsec:strongcomp}

 Our definition of strong completeness 
is limited to countable sets of formulas. 
We have not investigated the extent to which the strong completeness results 
in section~\ref{sec:strong compl}  generalise
to uncountable sets, but
an argument based on the Erd\H os--Rado theorem \cite{ErdRad56}
will show that   for any given dense-in-itself topological space $X$ and any Hilbert system
$H$ that is sound over $X$,  there is an (uncountable) cardinal $\kappa$ such that the set
$\{\di p_i:i<\kappa\}\cup\{\bo\neg(p_i\wedge p_j):i<j<\kappa\}$
is $H$-consistent but not satisfiable in $X$. So strong completeness
will fail over any given $X$, for large enough sets of formulas.

\begin{problem}
Let $X$ be a \nice\ space.
For which uncountable cardinals $\kappa$
can our strong completeness results for $X$ be extended
to sets of at most $\kappa$ formulas?
\end{problem}

Our strong completeness results for languages with $\bod$ are limited to 
logics with $\axG_1$. We could ask for more:
\begin{problem}
Let $X$ be a \nice\ space and let $\c L$ be 
$\c L_\bod$ or $\c L_{\bod}^{\didt}$.  Is the $\c L$-logic 
of $X$ strongly complete over $X$?
\end{problem}
By theorems~\ref{thm:strong did} and~\ref{thm:sc bod},
the answer is `yes' if $X$ validates $\axG_1$.

We saw in corollary~\ref{cor:stro compl fails}
that  in the language $\c L_{\bo\forall}$, 
there are many dense-in-themselves
metric spaces over which  $\axS4.\axU\axC$ is not strongly complete.
So we ask:

\begin{problem}
Can strong completeness for languages with \pa\ 
be proved for each \nice\ space in some reasonably large
class, and for $\R^n$ for $n\geq1$?  
\end{problem}

\begin{problem}
Is $\axS4.\axU\axC$ strongly complete for Kripke semantics in the language $\c L_{\bo\forall}$? 
\end{problem} 

Even without \pa,
the example in section~\ref{ss:can frame}
can be used to show that strong completeness
fails in Kripke semantics
for all our systems for languages containing $\c L_{\bo}^{\dit}$.
But we saw that strong completeness does hold for some of these systems over
dense-in-themselves
metric spaces.
Taking the example of  $\axS4t$ for $\c L_{\bo}^\dit$, it is striking that this logic 
is sound and complete for two different 
 semantics (the class of finite  S4 frames, 
 and any non-empty  \nice\ space), but  strongly complete  for only the latter.

\begin{problem}
Is there any general connection between strong completeness for topological semantics and for
Kripke semantics?
\end{problem}

\subsection{Complexity}
Decidability of the logics in table~\ref{table1}
follows from the finite model property results of section~\ref{sec:fmp}
and their finite (schema) axiomatisations.
But we have not investigated their complexity.
\begin{problem}
What is the complexity of the logics discussed in this paper?
\end{problem}
Of course, the complexity of some are known (e.g., S4 is {\sc PSpace}-complete).

\section*{Acknowledgments}

We would like to thank Nick Bezhanishvili for
stimulating conversations and for graciously sharing his interest in 
spatial semantics for the mu-calculus, and Valentin Shehtman for 
some inspiring lectures on spatial logic and help with references. 
The second author thanks the 
School of Mathematics and Statistics at the Victoria University of Wellington
for their warm and generous hospitality 
during his visit in January--March 2014
when the authors worked on the 
research presented here,
and the UK EPSRC for an overseas travel grant,
EP-L020750-1, which supported this visit.

\bibliographystyle{amsplain}
\bibliography{paper}

\end{document}